\documentclass[12pt, reqno]{amsart}%
\usepackage[english]{babel}
\usepackage{amsthm}     % theorem environments
\usepackage{amsmath}    % math features
\usepackage{amssymb}
\usepackage{dsfont}
\usepackage{graphicx}
\usepackage{esint}
\usepackage[shortlabels]{enumitem}
\usepackage{amscd,mdwlist,relsize}
\usepackage[margin=1.2in]{geometry}
\usepackage{color}
\usepackage{cite}
\usepackage{bbm, mathrsfs,pgf,tikz}
\usepackage{tcolorbox}

\usepackage[backgroundcolor=white, bordercolor=blue,
linecolor=blue]{todonotes}
\usepackage{url}
\usepackage{dsfont}
\usepackage{cases}
\usetikzlibrary{arrows}

\usepackage[pagebackref=false, pdfpagelabels, plainpages=false]{hyperref} % hyperlinks in PDF
\hypersetup{
colorlinks=true,
linkcolor=red,%---------------------%couleur des liens dans le document (references
filecolor=magenta,%-----------------%couleur des files à ouvrir
urlcolor=cyan,%---------------------%couleur des liens internets à ouvrir
citecolor=blue,%--------------------%couleur des liens pour les bibliographies
}%-------------------------------------%pour les courleur des liens ; ne jamias oublier de mettre le package hyperref avant

\theoremstyle{plain}  % italic text
\newtheorem{theorem}{Theorem}[section]
\newtheorem{lemma}[theorem]{Lemma}

\newtheorem{corollary}[theorem]{Corollary}

\theoremstyle{definition}  % upright text
\newtheorem{definition}[theorem]{Definition}
\newtheorem{example}[theorem]{Example}

\newtheorem{remark}[theorem]{Remark}   % numbered remark
\numberwithin{equation}{section}

\DeclareMathOperator{\dist}{dist}
\DeclareMathOperator{\supp}{supp}

\DeclareMathOperator{\diag}{diag}

\DeclareMathOperator{\loc}{loc}

\DeclareMathOperator{\pv}{\operatorname{p.\!v.}}

\DeclareMathOperator{\cE}{\mathcal{E}}

\DeclareMathOperator{\cN}{\mathcal{N}}

\DeclareMathOperator{\cJ}{\mathcal{J}}

\DeclareMathOperator{\R}{\mathbb{R}}

\newcommand{\iil}{\iint}
\renewcommand{\d}{\mathrm{d}}

\newcommand{\eps}{\varepsilon}

\newcommand{\WnuOmR}{W^{p}_{\nu}(\Omega|\R^d )}

\newcommand{\WnuOmO}{W^{p}_{\nu,\Omega}(\Omega|\R^d )}% Zero outside omega
\newcommand{\WnuOmRO}{W_{\nu,0}^{p}(\Omega|\R^d )}% the closure of smooth fct of compact supp

\newcommand{\TnuOm}{T^{p}_{\nu} (\Omega^c)}

\newcommand{\WnuOmRa}{W^p_{\nu_\eps}(\Omega|\R^d )}
\newcommand{\WnuOmRao}{W^p_{\nu_\eps,0}(\Omega|\R^d)}
\newcommand{\TnuOma}{T^p_{\nu_\eps} (\Omega^c)}
\newcommand{\WnuOma}{W^{p}_{\nu_\eps}(\Omega)}

\newcommand{\vertiii}[1]{{\left\vert\kern-0.25ex\left\vert\kern-0.25ex\left\vert #1 \right\vert\kern-0.25ex\right\vert\kern-0.25ex\right\vert}}
\addto\extrasenglish{%

}
\makeatletter
\@namedef{subjclassname@2020}{\textup{2020} Mathematics Subject Classification}
\makeatother

\begin{document}

\title{Optimal stability of complement value problems for $p$-L{\'e}vy operators.}

\author{Guy Foghem}
\address{{\tiny Brandenburgische Technische Universit\"at Cottbus--Senftenberg, Fakult\"{a}t 1: MINT Fachgebiet Mathematik, Platz der Deutschen Einheit 1, 03046 Cottbus, Germany.} \href{https://orcid.org/0000-0002-8917-7309}{ORCID}}
\email{guy.foghem[at]b-tu.de}
\thanks{Financial support from the Deutsche Forschungsgemeinschaft (DFG) through the Walter Benjamin Programme (project FO~1699/1-1) is gratefully acknowledged.}
\begin{abstract}
 We establish the optimal convergence of solutions to integro-differential equations (IDEs) governed by symmetric  integrodifferential $p$-L\'evy operators, $1 < p < \infty$, in the presence of nonlocal Dirichlet or Neumann boundary conditions. For illustrative purposes, consider the particular case of the (fractional) $p$-Laplacian $(-\Delta)^s_p$ with $0 < s \le 1$. If $(-\Delta)^s_p u_s = f_s $ in $\Omega \subset \R^d,$
augmented with a Dirichlet or Neumann data $g_s$ then under suitable assumptions on $\Omega$, $f_s$ and $g_s$,  we show that  $(u_s)_s$ strongly converges  as $s \to 1^-$ in the the optimal, that is, $\|u_s - u_1\|_{W^{s,p}(\Omega)} \to 0$.
\smallskip
Another subsequent goal underpinning our approach is the robustness of the nonlocal trace spaces; specifically, we  also show that the nonlocal trace spaces converge, in an appropriate sense, to the local trace space.
\end{abstract}


\maketitle

%\vspace{-6mm}
\tableofcontents
\vspace{-1cm}
\section{Introduction}\label{sec:intro}
\subsection{Motivation} In this article we are interested in studying optimal convergence of weak solutions $u_\eps: \R^d\to \R$, $\eps > 0$  satisfying the nonlocal to integro-differential equations (IDEs) driven by $p$-L\'evy operators, of the form
\begin{align}\label{eq:main-problem-nonlocal}
\tag{$P_{\nu_\eps,\tau}$}
\mu L_\eps u_\eps = f_\eps \quad \text{ in } \Omega
\quad \text{and} \quad
\tau \mu \mathcal{N}_\eps u_\eps + (1-\tau) u_\eps = g_\eps
\quad \text{ on } \Omega^c ,
\end{align}
with $\tau \in \{0,1\}$ yielding the nonlocal Dirichlet complement  condition for $\tau = 0$ and the nonlocal Neumann complement condition for $\tau = 1$. Here, $\Omega \subset \R^d$, $d \ge 1$, is an open set with complement $\Omega^c := \R^d \setminus \Omega$ and boundary $\partial\Omega$. Unless otherwise stated $1 < p < \infty$ with conjugate exponent $p' = \frac{p}{p-1}$, $\eps > 0$ is a small parameter aimed to be sent to zero and  the universal constant $\mu^{-1}=K_{d,p}$ plays a crucial role in our asymptotic analysis, see \cite{guy-thesis, Fog23} and is given by
\begin{align*}
K_{d,p}= \frac{1}{|\mathbb{S}^{d-1}|}\int_{\mathbb{S}^{d-1}} |w_d|^p\d \sigma_{d-1}(w)= \frac{ \Gamma\big(\frac{d}{2}\big)  \Gamma\big(\frac{p+1}{2}\big)}{\Gamma\big(\frac{d+p}{2}\big) \Gamma\big(\frac{1}{2}\big)}.
\end{align*}
Throughout this paper, the operators $L_\eps$ and $\mathcal{N}_\eps$ are respectively the nonlocal $p$-L\'evy operator and nonlocal $p$-normal derivative across $\Omega$, defined by
\begin{align*}
L_\eps u(x)&=2\pv \int_{\R^d} |u(x)-u(y)|^{p-2}(u(x)-u(y)) \nu_\eps(x-y)\d y
&&(x\in \R^d),\\
\mathcal{N}_\eps u(x)&=2 \int_{\Omega}
|u(y)-u(x)|^{p-2}(u(x)-u(y)) \nu_\eps(x-y)\d y
&&(x\in \Omega^c),
\end{align*}
where the sequence $(\nu_\eps)_\eps$ is family of radial $p$-L\'evy kernels approximating the identity. We say that a family of   $p$-L\'evy integrable kernels $(\nu_\eps)_{\eps>0}$, is $p$-L\'evy approximation of the unity if  each $\nu_\eps : \R^d\setminus\{0\}\to [0,\infty)$, $\eps>0$ is radial and we have
\begin{align}\label{eq:plevy-approx}
\int_{\R^d}(1\land|h|^p) \nu_\eps (h)\d h=1\quad\text{and for all $\delta>0,$}\quad \lim_{\eps\to0} \int_{|h|>\delta} \hspace{-2ex}\nu_\eps (h)\d h=0.
\end{align}
Note that condition \eqref{eq:plevy-approx} may be replaced alternatively by the analogous condition
 \begin{align}\label{eq:plevy-approx-bis}
\lim_{\eps\to0}\int_{\R^d}(1\land|h|^p) \nu_\eps (h)\d h=1\quad\text{and for all $\delta>0,$}\quad \lim_{\eps\to0} \int_{|h|>\delta}\hspace{-2ex} \nu_\eps (h)\d h=0.
\end{align}
In order to connect the nonlocal formulation with its local counterpart, it is shown in \cite{Fog25} that, letting $\eps \to 0$ in \eqref{eq:main-problem-nonlocal}, the sequence $(u_\eps)_\eps$ converges in $L^p(\Omega)$ to a function $u\in W^{1,p}(\Omega)$ solving the corresponding local boundary problem
\begin{align}\tag{$P_\tau$}
\label{eq:main-problem-local}
-\Delta_p u = f \quad \text{ in } \Omega, \quad\text{ and }\quad
\tau  \partial_{n,p} u + (1-\tau) u = g \quad \text{ on } \partial\Omega,
\end{align}
 where one immediately recognizes the local Dirichlet boundary condition for $\tau = 0$ and the local Neumann boundary condition for $\tau = 1$, $\Delta_p u =\operatorname{div}(|\nabla u|^{p-2}\nabla u)$ denotes the usual $p$-Laplace operator, while for sufficiently smooth $\partial \Omega$ with outer normal derivative $n$, $\partial_{n,p} u= |\nabla u|^{p-2}\nabla u\cdot n$ is the corresponding $p$-normal derivative of $u$ on $\partial \Omega$.

\subsection{Relevance}
\par Note that our choice of the sequence $(\nu_\eps)_\eps$ satisfying \eqref{eq:plevy-approx} is not fortuitous and is essentially optimal. Indeed, according to \cite[Theorem 9.6]{Fog25}  the class $(\nu_\eps)_\eps$ is intrinsic to the fact that, the condition \eqref{eq:plevy-approx} is valid if and only if the following characterization  holds
\begin{align}\label{eq:limit-goal}
\lim_{\eps\to0} \int_{\R^d} \int_{\R^d} |u(x+h)- u(x)|^p\nu_\eps(h)\d h\d x
=K_{d,p}\|\nabla u\|^p_{L^p(\R^d)}
&& \text{$\forall \, u\in W^{1,p}(\R^d)$}.
\end{align}
This characterization \eqref{eq:limit-goal} therefore extends the seminal work initiated by Brezis, Bourgain and Mirunescu in \cite{BBM01}.
A similar approach to this type of characterization is explored in \cite{DDP23,DDG24,GeSt26}. It turns out that the sole condition \eqref{eq:plevy-approx}, we typically obtain that $L_{\varepsilon} u$ converges to $-K_{d,p}\Delta_p u$ pointwise and weakly, while as highlighted above $\cN_{\varepsilon} u $ somewhat converges to $K_{d,p}\partial_{n,p} u= K_{d,p}|\nabla u|^{p-2}\nabla u\cdot n$ in some weak sense.  Namely,  \eqref{eq:plevy-approx} also guarantees the pointwise convergence
\begin{align*}
    \lim_{\varepsilon \to 0} L_{\varepsilon} u(x)
    = - K_{d,p}\, \Delta_p u(x),
\end{align*}
for $u \in L^\infty(\R^d) \cap C^2(B_1(x)),$ $ x \in \R^d$, provided that in the case $1<p<2$ we have $\nabla u(x) \neq 0$.
This type of pointwise estimate is not entirely straightforward and is partly an impetus for studying the convergence of the sequence of solutions $(u_\eps)_\eps$.
Some particular instances of pointwise convergence involving $p$-L\'evy operators of the type $L_\eps u$
has been established  for instance in
\cite[Theorem~2.8]{BS22} for the fractional $p$-Laplacian, \cite[Corollary~6.3]{dTL21} for the case where $\nu_\eps(h)= \tfrac{d+p}{|\mathbb{S}^{d-1}|} \eps^{-d-p}\mathds{1}_{B_\eps(0)}(h)$, and the variant in \cite[Section~7]{IN10} for the case of  the constrained fractional $p$-Laplacian. Let us mention two prototypical examples of sequences $(\nu_\eps)_\eps$ satisfying \eqref{eq:plevy-approx} that are relevant. Further examples can be found in \cite{Fog23, guy-thesis, FGKV20}. Let $\nu : \mathbb{R}^d \setminus \{0\} \to [0,\infty)$ be a radial and $p$-L\'evy normalized, i.e., $\int_{\R^d} (1 \wedge |h|^p)\, \nu(h)\d h = 1.$ The first example is obtained by considering $(\nu_{\eps})_{\eps}$ to the rescaling of  $\nu$  with
\begin{align*}
\begin{split}
\nu_\eps(h) =
\begin{cases}
\eps^{-d-p}\nu\big(h/\varepsilon\big)& \text{if}~~|h|\leq \varepsilon,\\
\eps^{-d}|h|^{-p}\nu\big(h/\varepsilon\big)& \text{if}~~\varepsilon<|h|\leq 1,\\
\eps^{-d}\nu\big(h/\varepsilon\big)& \text{if}~~|h|>1.
\end{cases}
\end{split}
\end{align*}

\noindent The second example of interest is the sequence $(\nu_\eps)_\eps$ of fractional kernels with
\begin{align*}
\nu_\eps(h) = a_{\varepsilon, d,p} |h|^{-d-(1-\eps)p}\quad\text{with }\quad a_{\varepsilon, d,p} = \frac{p\eps(1-\eps)}{|\mathbb{S}^{d-1}|}.
\end{align*}
The resulting operator $L_\eps$ is a multiple of fractional $p$-Laplace operator $(-\Delta)^s_p$, $s\in (0,1)$. More precisely, $L_\eps u= \frac{a_{\eps,d,p}}{\widetilde{C}_{d,1-\eps,p}} (-\Delta)^s_pu$, with $s=1-\eps$. In other words,
 \begin{align*}
(-\Delta)_p^s u(x)	=\widetilde{C}_{d,p,s}\pv \int_{\R^d}\frac{|u(x)-u(y)|^{p-2}(u(x)-u(y))}{|x-y|^{d+sp}}\d y = \tfrac{\widetilde{C}_{d,1-\eps,p}}{ a_{\eps,d,p}}L_\eps u(x),
\end{align*}
where our chosen normalizing constant is
{\tiny
$\widetilde{C}_{d,p,s} =
\begin{cases}
C_{d,p,s}, & \text{if } sp \ge 1, \\
C_{d,2,\frac{sp}{2}}, & \text{if } sp < 1
\end{cases}
$} with
\begin{align*}
C_{d,p,s}=\frac{s(1-2s)\Gamma(\frac{d+sp}{2})}{\pi^{\frac{d-1}{2}}\Gamma(\frac{sp+1}{2})\Gamma(p(1-s)) \cos(s\pi)}=
\frac{
2^{2s}\,\Gamma\!\left(\frac{d+sp}{2}\right)
\,\Gamma\!\left(\frac{2s+1}{2}\right)
\,\Gamma\!\bigl(2(1-s)\bigr)
}{
\pi^{\frac{d}{2}}
\,\bigl|\Gamma(-s)\bigr|
\,\Gamma\!\left(\frac{sp+1}{2}\right)
\,\Gamma\!\bigl(p(1-s)\bigr)
}.
\end{align*}
As noted in \cite{Fog25,Fog26} our normalization constant $\widetilde{C}_{d,p,s}$,  is chosen so as to ensure that if we put $|u|^p_{W^{s,p}(\R^d)}= \cE_{\R^d}^{s,p}(u,u)$; see Section \ref{sec:conv-weak-solution-frac}, the following properties hold.
\begin{itemize}
\item $\widehat{(-\Delta)^s u}(\xi)= |\xi|^{2s}\widehat{u}(\xi)$  for all $u\in C_c^\infty(\R^d)$, where $ \widehat{u}(\xi) =\int_{\R^d} e^{-ix\cdot \xi} u(x)\d x$.
\item  $(-\Delta)^s_p u(x)\to -\Delta_p u(x)\quad$  and  $\quad|u|_{W^{s,p}(\R^d)}\xrightarrow{s\to1^-}\|\nabla u\|_{L^{p}(\R^d)}$ as  $s\to1^-$.
\item $(-\Delta)^s_p u(x)\to |u(x)|^{p-2} u(x)\quad$  and  $\quad|u|_{W^{s,p}(\R^d)}\to \|u\|_{L^{p}(\R^d)}$ as  $s\to0^+$.
\item The following asymptotic behaviors hold
\begin{align*}
\lim_{s\to1^-} \frac{K_{d,p}C_{d,p,s}}{s(1-s)}=\lim_{s\to 0^+}\frac{\Gamma(p+1)C_{d,p,s}}{s(1-s)}=\lim_{s\to 0^+} \frac{2C_{d,2, \frac{sp}{2}}}{s(1-s)}=\frac{2p}{|\mathbb{S}^{d-1}|}.
\end{align*}
\end{itemize}
In particular, $ \frac{a_{\eps,d,p} }{\widetilde{C}_{d,1-\eps,p}}\to K_{d,p}$  as $\eps\to 0$. Others normalizing constants for the fractional $p$-Laplacian are proposed in \cite{War16,dTGCV21}. See  also in \cite{DJS25} where the need for asymptotic  analysis near $s \to 0^+$, in connection with the logarithmic $p$-Laplacian, is considered.

\subsection{Main results}
Our main goal in the present work is to establish the optimal convergence, as $\eps\to 0$, of the family $(u_\eps)_\eps$ to $u$, the solution of the local problem \eqref{eq:main-problem-local}. The price to pay for establishing  the optimal convergence is the need to require following suitable weak or strong asymptotic convergence assumptions on the data $f_\eps$,  $g_\eps,  f $ and  $g$; see Section~\ref{sec:asymp-conv-weak} for the precise definitions. (a) We assume that the sequence $(f_\eps)_\eps$, with $f_\eps \in \bigl(W^{p}_{\nu_\eps}(\Omega \mid \R^d)\bigr)'$, converges asymptotically weakly to $f \in \bigl(W^{1,p}(\Omega)\bigr)'$. This includes the particular situation where $f_\eps, f\in L^{p'}(\Omega)$ and $f_\eps\rightharpoonup f$ in $L^{p'}(\Omega)$.
(b) In the case $\tau=1$ we assume that the sequence $(g_\eps)_\eps$ with $g_\eps \in \bigl(T^{p}_{\nu_\eps}(\Omega^c)\bigr)'$, converges asymptotically weakly to $g \in \bigl(W^{1-1/p ,p}(\partial\Omega)\bigr)'$.
This includes the particular situation where $g_\eps= \mathcal{N}_\eps \varphi $ and $ g= K_{d,p}\partial_{n,p}\varphi $ for any $\varphi\in C^2_b(\R^d)$ where in the singular case $1<p<2$ we assume $\nabla \varphi(x)\neq 0$ for all $x\in \overline{\Omega}$.
(c) In the case $\tau=0$ we assume that the sequence $(g_\eps)_\eps$ with $g_\eps \in T^{p}_{\nu_\eps}(\Omega^c)$ converges asymptotically strongly to $g \in W^{1-1/p ,p}(\partial\Omega)$. This includes the particular situation where $g_\eps, g\in W^{1,p}(\R^d\setminus \overline{\Omega})$ and $g_\eps\to  g$ strongly in $W^{1,p}(\R^d\setminus \overline{\Omega})$.   We emphasize that  $W^{1-1/p ,p}(\partial\Omega)$ and $T^{p}_{\nu_\eps}(\Omega^c)$ are, respectively, the trace spaces of the (non)local Sobolev spaces $W^{1,p}(\Omega)$ and $\WnuOmRa$; see Section  \ref{sec:robust-trace} for more details. Further development of these spaces can be found in \cite{guy-thesis, Fog21b, FoKa24, Fog25}. We recall that
\begin{align*}
\WnuOmRa
&= \Big\lbrace u: \R^d \to \R \text{ meas.} \,: \, u|_\Omega\in L^p(\Omega)\,\!\text{ and }\,\!  \cE^\eps(u,u) <\infty \Big\rbrace,\\
W^p_{\nu_\eps,0}(\Omega|\R^d)
&= \{ u\in \WnuOmRa~: ~u=0~~\text{a.e. on } \R^d\setminus \Omega\},\\
\WnuOma
&= \Big\lbrace u\in L^p(\Omega)\,\!:\,\!  \cE^\eps_\Omega(u,u) <\infty \Big\rbrace,
\end{align*}
where we consider the forms
\begin{align*}
\mathcal{E}^{\eps}(u,v) &= \iil_{(\Omega^c\times \Omega^c)^c}\hspace*{-2ex}|u(x)-u(y)|^{p-2}(u(x)-u(y))(v(x)-v(y) )\, \nu_\eps(x-y) \d y \, \d x,
\\ \mathcal{E}^{\eps}_{\Omega}(u,v) &=  \iil_{\Omega \Omega} |u(x)-u(y)|^{p-2}(u(x)-u(y))(v(x)-v(y)) \nu_\eps(x-y)\, \d y \, \d x.
\end{align*}
Note that $\mathcal{E}^{\eps}(u,v) =\mathcal{E}^{\eps}_\Omega(u,v) +\mathcal{E}^{\eps}_{cr}(u,v)+ \mathcal{E}^{\eps}_{cr}(v,u) $ where we put
\begin{align*}
\cE^\eps_{cr}(u,v)&= \iil_{\Omega \Omega^c} |u(x)-u(y)|^{p-2}(u(x)-u(y))(v(x)-v(y)) \, \nu_\eps(x-y)\d y \, \d x.
\end{align*}
The spaces $\WnuOmRa$ and $\WnuOma$ are equipped respectively with is the (semi)norms
\begin{align*}
\|u\|_{\WnuOmRa} &=
\big( \|u\|^p_{L^p(\Omega)}+ \cE^\eps(u,u)\big) ^{1/p} \\
\|u\|_{\WnuOma} &=
\big( \|u\|^p_{L^p(\Omega)}+ \cE^\eps_\Omega(u,u)\big) ^{1/p}.
\end{align*}

Another purpose of this paper is to prove the robustness of the nonlocal trace spaces $\TnuOma$ as $\eps \to 0$; which plays a crucial role in our approach to proving the optimal convergence result. More precisely, we prove that the family of nonlocal trace spaces $\TnuOma$, $\eps>0$, converges, in an appropriate sense, to the local trace space $W^{1-1/p,p}(\partial\Omega)$. The optimal convergence requires a careful analysis of a form of demi-convergence of the nonlocal forms $\cE^\eps(\cdot,\cdot)$ to the local form $K_{d,p}\cE^0(\cdot,\cdot)$ where
\begin{align*}
\cE^0(u,v)=\int_\Omega |\nabla u(x)|^{p-2}\nabla u(x)\cdot \nabla v(x)\d x.
\end{align*}
The latter is obtained by establishing several uniform estimates involving $\cE^\eps(\cdot,\cdot)$ that are independent of $\eps$.
This paper contains several new contributions. However, to ensure good flow in the introduction, we list them under the following main categories, under various mild assumptions on $\Omega$.
\begin{enumerate}[$(A)$]
\item \textbf{Optimal convergence.} In Section \ref{sec:conv-weak-solution} we show that, the sequence $(u_\eps)_\eps$ strongly converges to $u\in W^{1,p}(\Omega)$ satisfying \eqref{eq:main-problem-local}, in optimal  sense, that is,
\begin{align}\label{eq:intro-op-con}
\lim_{\eps\to0}\|u_\eps-u\|_{\WnuOma}= \lim_{\eps\to0}\|u_\eps-\overline{u}\|_{\WnuOmRa}= \lim_{\eps\to0} \cE^\eps_{cr}(u_\eps,u_\eps)=0,
\end{align}
where $\overline{u}\in W^{1,p}(\R^d)$ is a suitable $W^{1,p}$-extension of $u\in W^{1,p}(\Omega)$. For the case $\tau=1$ we require in addition that $\int_\Omega u_\eps \d x= \int_\Omega u \d x= 0$. Our approach also applies to the  Neumann problem associated with the regional $p$-L{\'e}vy operator

\begin{align*}
L_{\Omega, \eps} u(x):= 2\pv \int_\Omega |u(x)-u(y)|^{p-2}(u(x)-u(y))\, \nu_\eps(x-y)\d y.
\end{align*}
It is worth emphasizing that  no stronger convergence can be expected in general. Indeed, for sufficiently smooth $\Omega$, it is shown in \cite{Fog25} that $u_\eps \in \WnuOmRa$, i.e., $u_\eps|_\Omega\in \WnuOma$,
while $u \in W^{1,p}(\Omega)\subset \WnuOma$.

\item \textbf{Robust trace spaces.}
In Section \ref{sec:robust-trace} we show that, the nonlocal trace spaces $\TnuOma$ resp. the corresponding trace operator $\operatorname{Tr} : \WnuOmRa \to \TnuOma,$
with $ \operatorname{Tr}(g) = g|_{\Omega^c},$
converge, to the  local trace space $W^{1-1/p,p}(\partial\Omega)$ resp. the usual local  trace operator $\gamma_0 : W^{1,p}(\Omega) \to W^{1-1/p,p}(\partial\Omega)$ in the following sense:
if $\|\operatorname{Tr}(g_\eps-g)\|_{\TnuOma}\to 0$ with $g\in W^{1,p}(\R^d)$ and $g_\eps\in \WnuOmRa$ then
\begin{align*}
\lim_{\eps\to0}	\|\operatorname{Tr}(g_\eps)\|_{\TnuOma}=
\|\gamma_0(g)\|_{W^{1-1/p ,p}(\partial\Omega)}
\end{align*}
Furthermore, one also finds the existence of robust lifting. Namely, there exist operators $R_\eps: \TnuOma\to \WnuOmRa$ with $\operatorname{Tr}\circ R_\eps= \operatorname{Id}$  and $R_0:W^{1-1/p , p}(\partial \Omega)\to W^{1,p}(\Omega)$ with $\gamma_0\circ R_0= \operatorname{Id}$ such that for all $g\in W^{1,p}(\R^d\setminus\overline{\Omega})$
\begin{align*}
\lim_{\eps\to0}\|R_\eps(g)-R_0(g)\|_{\WnuOma}= \lim_{\eps\to0}\|R_\eps(g)-R_{0,g}(g)\|_{\WnuOmRa}=0,
\end{align*}
where $R_{0,g}(g)\in W^{1,p}(\R^d)$ is the extension by $g$ of $R_{0}(g)$.
\item  \textbf{Demi-convergence of forms.}
In Section~\ref{sec:conv-nonloc-to-loc-forms},  we able to obtain a sort of asymptotic demi-convergence of the nonlocal forms $\cE^\eps(\cdot,\cdot)$ to the local form $\cE^0(\cdot,\cdot)$. %(w_\eps)_\eps
More precisely, given $v$ and $w$, if $w_\eps\in \WnuOmRa$ satisfies  $w_\eps \to w$ in $L^p_{\operatorname{loc}}(\Omega)$  and $\sup_{\eps>0} \| w_\eps\|_{\WnuOmRa}<\infty $, then
under the condition either that $v,w$ and $\Omega$ are sufficiently smooth or that ($w\in W^{1,p}_0(\Omega)$ or $v\in W^{1,p}_0(\Omega)$), we have
%for all $v\in W^{1,p}(\R^d)$
\begin{align*}
&\lim_{\eps\to0} \cE^\eps(v,w_\eps)= \lim_{\eps\to0}\cE^\eps_\Omega (v,w_\eps)= K_{d,p}\cE^0(v,w),\\
&\lim_{\eps\to0} \cE^\eps_{cr}(w_\eps,v)
=\lim_{\eps\to0}\cE^\eps_{cr} (v,w_\eps)
=0.
\end{align*}

%\item \textbf{Convergence of dual spaces.} We also show the dual space of $\WnuOmRa$ (resp. $W^p_{\nu_\eps,0}(\Omega|\R^d)$)converges to the dual space of $W^{1,p}(\Omega)$ (resp. $W^{1,p}_0(\Omega)$. Namely, for all $f\in (L^p(\Omega))'$) we have
%\begin{align*}
%\| f\|_{(W^{1,p}(\Omega))'}=\lim_{\eps\to0} \| f\|_{(\WnuOma)'} \,\, \text{ and }\,\,
%\| f\|_{(W^{1,p}_0(\Omega))'}
%&=\lim_{\eps\to0} \| f\|_{(W^p_{\nu_\eps,0}(\Omega|\R^d))'}.
%\end{align*}
\item \textbf{Case of the fractional $p$-Laplace.}
Our optimal convergence result is new even in the case of the fractional $p$-Laplacian $(-\Delta)^s_p$, with $0 < s < 1$, which is treated in Section~\ref{sec:conv-weak-solution-frac}.  Replacing, \emph{mutatis mutandis}, $f_\eps$ and $g_\eps$ by $f_s$ and $g_s$, if we assume that $w_1 = u$ solves \eqref{eq:main-problem-local} and that $w_s$ satisfies
\begin{align}\label{eq:main-fract}
(-\Delta)^s_p w_s = f_s \quad \text{ in } \Omega, \quad\text{and}\quad
\tau \,\cN_p^s w_s + (1-\tau) w_s = g_s \quad \text{on } \Omega^c,
\end{align}
then $(w_s)_s$ strongly converges to $w_1=u$ in the following sense
\begin{align*}
\lim_{s\to1^-}\|w_s -w_1\|_{W^{s,p}(\Omega)} = \lim_{s\to1^-}\|w_s- \overline{w}_1\|_{W^{s,p}(\Omega|\R^d)}=0,
\end{align*}
where $\overline{w}_1\in W^{1,p}(\R^d)$ is a suitable  $W^{1,p}$-extension of $w_1$.
\end{enumerate}

%\vspace{1mm}

\noindent Let us briefly highlight some key points underlying our strategy for obtaining the above results. We obtain the convergence of $u_\eps {\,'}s$ in $L^p(\Omega)$ by implicitly relying on a form of $\Gamma$-convergence from nonlocal functionals $(\mathcal{J}^\eps_\tau )_\eps$ to a local functional $\mathcal{J}_\tau$ w.r.t the topology of $L^p(\Omega)$; where  $\mathcal{J}^\eps_\tau$  and $\mathcal{J}_\tau$, $\tau=0,1$, are roughly defined as follows
\begin{align*}
 \cJ^\eps_1(v)&=\frac{\mu}{p}  \mathcal{E}^\eps(v,v) -\langle f_\eps, v \rangle_\eps - \langle g_\eps, v\rangle_\eps,
&& \qquad\cJ_1(v)=\frac{1}{p}\mathcal{E}^0(v,v)-  \langle f, v\rangle- \langle g, v\rangle,\\
\mathcal{J}^\eps_0(v)
&= \frac{\mu}{p}  \cE^\eps(v,v)-\langle f_\eps , v-g \rangle_\eps, &&\qquad \cJ_0(v)=\frac{1}{p}\mathcal{E}^0(v,v) -\langle f ,v-g\rangle.
\end{align*}
\noindent Here are two key observations.  First, a function minimizes $\mathcal{J}^\varepsilon_\tau$ (resp. $\mathcal{J}_\tau$) if and only if it is a weak solution of \eqref{eq:main-problem-nonlocal} (resp.  \eqref{eq:main-problem-local}.
Secondly, as consequence of the nonlocality effect functions entering the nonlocal form $\cE^\eps(\cdot,\cdot)$ are defined on $\R^d$ whereas functions entering local form $\cE^0(\cdot,\cdot)$ are defined on $\Omega$,  such $\Gamma$-convergence cannot hold in the classical sense, due to this mismatch between the domains of definition of the functionals $\mathcal{J}^\eps$ and $\mathcal{J}$; see Section \ref{sec:conv-weak-solution} for more details. Furthermore, this nonlocality effect is genuinely reflected in the collapse of the cross-forms $\cE^\eps_{cr}(\cdot,\cdot)$ across $\partial\Omega$ as $\eps \to 0$. As highlighted in \eqref{eq:intro-op-con}, this collapse is also manifested in the convergence of weak solutions as one has $\cE^\eps_{cr}(u_\eps,u_\eps) \to 0$ as $\eps \to 0$. Afterwards, our approach to the optimal convergence strongly relies on the aforementioned asymptotic demi-convergence of the energy forms $\cE^\eps(\cdot,\cdot)$ to  $\cE^0(\cdot,\cdot)$, together with uniform linearization of the forms $\cE^\eps(\cdot,\cdot)$  see Theorem  \ref{thm:equiv-conv-in-form}.  This analysis also enables us to show  in Theorem \ref{thm:conv-dual-spaces} that dual spaces $(W^p_{\nu_\eps,0} (\Omega|\R^d))'$, $(\WnuOmRa)'$  and $(\WnuOma)'$ converge to local counterpart $ (W^{1,p}_0(\Omega))'$ and  $ (W^{1,p}(\Omega))'$ respectively. We also benefit from various robust Poincar{\'e} inequalities established in \cite{Fog25}.
Furthermore, in the case of Dirichlet problem, i.e. when $\tau=0$,  the robustness of the trace spaces is naturally needed in addition.  Meanwhile  the convergence of forms also implies the robustness of the trace spaces, i.e., $(\TnuOma)_\eps$ converges to $W^{1-1/p,p}(\partial\Omega)$, whereas the robustness of the lifting operator  $(R_\eps)_\eps$ partially relies on optimal convergence of solutions $u_\eps$ in the case of Dirichlet problem, i.e. when $\tau=0$. Finally, the case of fractional $p$-Laplacian follows directly by using the sequence
$\nu_\eps(h)=\tfrac{p\eps(1-\eps)}{|\mathbb{S}^{d-1}|}|h|^{-d-p(1-\eps)}$.

\vspace{-1mm}
\subsection{Literature}
We now provide a brief account of the literature. This work continues and generalizes the investigation initiated in our previous article \cite{Fog25} where we only prove strong convergence in $L^p(\Omega)$ under relaxed assumptions on $f_\eps$ and $ g_\eps$.  See also \cite{guy-thesis,FoKa24} for $p=2$ wherein the asymptotic for general nonlocal elliptic operators is treated. To the best of our knowledge, our optimal convergence result is new event for the fractional $p$-Laplacian. There are some recent progresses about convergence of solutions related to the case of the fractional $p$-Laplacian. For instance, the convergence rate for the Dirichlet problem for $p=2$ is discussed in \cite{BuFe25}. The differentiability of the mapping $s\mapsto w_s$, with $w_s$ given as in \eqref{eq:main-fract} for $\tau=1$ is provided in \cite{JSW25} for $p=2$.
Convergence  of weak solutions to nonlocal  problems with complement Neumann condition to the local ones appear in \cite{guy-thesis,FoKa24,GrHe24}.
Note however that the convergence of nonlocal Neumann problems associated with regional type operators can be found in \cite{AMRT08, AMRT10,DLS15}. There is a substantial amount of works treating the convergence from nonlocal to local of weak solutions to homogeneous Dirichlet problems associated with the fractional $p$-Laplacian $(-\Delta)^s_p$; we refer interested reader for instance to  \cite{BPS16,FeSa20,BS22,BO20,BO21,SaVe22}.  For convergence of solutions to elliptic problems, see \cite{guy-thesis,Voi17,Gru25}. The work \cite{ScDu25} addresses also the strong nonlocal-to-local convergence for nonlocal problems with a shrinking horizon, arising in the peridynamic framework and driven by boundary-localized nonlocal convolution operators.  The robustness of the fractional trace spaces in connection with the Dirichlet problem is established in \cite{GrKa25}, thereby extending the earlier work from \cite{GrHe24} dealing with the case $p=2$. Further details can found in \cite{Gru25}.  The cornerstone  of \cite{GrHe24,GrKa25} is that, the scaling property of the fractional kernel $h\mapsto(1-s)|h|^{-d-sp}$ and the the Whitney decomposition, enable the authors to refine the earlier work \cite{DyKa19} by  giving an explicit characterization of the nonlocal trace space. Finally, the asymptotic demi-convergence above generalizes and extends the convergence of forms obtained in \cite{Fog23,guy-thesis}. This type of convergence was originally studied in \cite{BBM01}. It is, however, important to note that earlier work \cite{AAS67}, though nowadays largely overlooked, had already investigated the nonlocal-to -local convergence of fractional energy forms in the fractional Sobolev spaces with $p=2$, thereby, as later emphasized in \cite{BBM01}, anticipating the need of the factor $(1-s)^{1/p}$ in front of the fractional seminorm to obtain the correct asymptotic behavior. Further, the work \cite{AAS67} also addresses the problem of convergence from nonlocal to local energy forms, as subsequently studied in \cite{BBM01} and the asymptotic compactness result as studied in \cite{Pon04}. Several generalizations concerning the nonlocal-to-local convergence of energy forms have recently emerged. For instance, related developments on Mosco and/or $\Gamma$-convergence can be found in \cite{Wei22,guy-thesis,FGKV20,Pon04-gamma,Voi17}. Various extensions in other directions are discussed in \cite{BMR20,DB22,DD22,Fog23,PS17,Mil05}, as well as in the variant settings considered in \cite{NPSV18,LMP19} and the references therein.

\subsection*{Acknowledgement:}
Part of this work was carried out during the author’s postdoctoral stay at Technische Universit{\"a}t Dresden. The author gratefully acknowledges the excellent working conditions and hospitality provided by the host institution.

%Part of this work was completed during the author’s postdoctoral appointment at TU Dresden. The author would like to thank Technische Universit{\"a}t Dresden for its hospitality and support.

\section{Convergence of nonlocal  forms} \label{sec:conv-nonloc-to-loc-forms}
In this section, we discuss in various perspectives the convergence of nonlocal energy forms.  More importantly,are interested in the asymptotic of the energy forms:  max form (or special form) $\cE^\eps(\cdot,\cdot)$,  plus form $\cE^\eps_+(\cdot,\cdot)$, Gagliardo form  $\cE^\eps_\Omega(\cdot,\cdot)$ and boundary crossing form  $\cE^\eps_{cr}(\cdot,\cdot)$ and local gradient form $\cE^0(\cdot,\cdot)$ defined as follows
\begin{align*}
\mathcal{E}^{\eps}(u,v) &= \iil_{(\Omega^c\times \Omega^c)^c} |u(y)-u(x)|^{p-2}(u(y)-u(x)) (v(y)-v(x) )\nu_\eps(x-y) \d y \, \d x,
\\ \mathcal{E}^{\eps}_+(u,v) &= \iil_{\Omega \R^d} |u(y)-u(x)|^{p-2}(u(y)-u(x))(v(x)-v(y))\nu_\eps(x-y)\d y \, \d x,
\\ \mathcal{E}^{\eps}_{\Omega}(u,v) &=  \iil_{\Omega \Omega} |u(y)-u(x)|^{p-2}(u(y)-u(x)) (v(y)-v(x)) \nu_\eps(x-y)\d y \, \d x, \\
\mathcal{E}^{\eps}_{cr}(u,v) &= \iil_{\Omega \Omega^c} |u(y)-u(x)|^{p-2}(u(y)-u(x))(v(x)-v(y))\nu_\eps(x-y)\d y \, \d x,
\\\mathcal{E}^{0}(u,v)&= \int_\Omega |\nabla u(x)|^{p-2}\nabla u(x)\cdot \nabla v(x)\d x.
\end{align*}

\noindent Note  in passing that if when $\Omega=\R^d$ then  $\mathcal{E}^{\eps}(u,v)= \mathcal{E}^{\eps}_+(u,v) =\mathcal{E}^{\eps}_{\Omega}(u,v)$. For a general  measurable kernel  $k:\R^d\times \R^d\setminus\diag \to [0\,, \infty)$,  we introduce the nonlinear nonlocal form
\begin{align*}
\cE_k (u, v):= \iil_{\R^d \R^d} |u(x)-u(y)|^{p-2}(u(x)-u(y)) (v(x)-v(y))k(x,y)\d y\d x.
\end{align*}
It is important to keep in mind that H\"older's inequality implies that
\begin{align}\label{eq:holder-form-estimate}
|\cE_k(u,v)|\leq \cE_k(u,u)^{1/p'}\cE_k(v,v)^{1/p}.
\end{align}

\begin{remark} The nonlocal forms $\cE^\eps(\cdot,\cdot)$,  $\cE^\eps_+(\cdot,\cdot)$, $\cE^\eps_\Omega(\cdot,\cdot)$  and  $\cE^\eps_{cr}(\cdot,\cdot)$ merely appear as special instances of the nonlocal form $\cE_k(\cdot,\cdot)$. Indeed one easily verifies that
\begin{align*}
\cE_k (u, v)=\cE^\eps(u, v)\quad \text{when}\quad k(x,y)
&= \max( \mathds{1}_{\Omega} (x),\mathds{1}_{\Omega} (y))\nu_\eps(x-y),\\
\cE_k (u, v)=\cE^\eps_+(u, v)\quad \text{when}\quad k(x,y)&= \tfrac{1}{2}(\mathds{1}_{\Omega} (x)+\mathds{1}_{\Omega} (y))\nu_\eps(x-y),\\
\cE_k (u, v)=\cE^\eps_\Omega(u, v)\quad \text{when}\quad k(x,y)&= \min( \mathds{1}_{\Omega} (x),\mathds{1}_{\Omega} (y))\nu_\eps(x-y),\\
\cE_k (u, v)=\cE^\eps_{cr}(u, v)\quad \text{when}\quad k(x,y)&= \min( \mathds{1}_{\Omega} (x),\mathds{1}_{\Omega^c} (y))\nu(x-y).
\end{align*}
Furthermore, in each of these four examples, one obtains the following uniform estimate
\begin{align}\label{eq:taylor-exapnsion}
\cE_k(u,u)\leq \cE^\eps_{\R^d}(u,u) \leq 2^p \|u\|^p_{W^{1,p}(\R^d)}, \qquad\text{$u\in W^{1,p}(\R^d)$},
\end{align}
independently on $\eps$. This is  readily implied by the  fact that $\int_{\R^d}(1\land|h|^p)\nu_\eps(h)\d h=1$ and the following Taylor expansion
\begin{align*}
\int_{\R^d}|u(x)-u(x+h)|^p\d x\leq 2^p(1\land |h|^p) \|u\|_{W^{1,p}(\R^d)}.
\end{align*}
\end{remark}

\subsection{Uniform linearization of the  forms} In this subsection, we establish a uniform linearization of the aforementioned forms, which further eases their asymptotic analysis.  The proof of the following Lemma can be found in \cite[Lemma A.3]{Fog25}.
\begin{lemma}\label{lem:elm-est-simons}
For all $x,y\in \R^d$, the following inequalities hold true.
\begin{numcases}{\big||x|^{p-2}x-|y|^{p-2}y\big|\leq}
\kappa_p |x-y|(|x|+|y|)^{p-2} & $p\in [2,\infty),$ \label{eq:upper-elem-degen}%\kappa_p =(p-1)
\\
\kappa_p |x-y|^{p-1}&  $p\in [1,2]$. \label{eq:upper-elem-sing}
\end{numcases}

\begin{numcases}{\big(|x|^{p-2}x-|y|^{p-2}y\big)\cdot (x-y)\geq}
\kappa'_p |x-y|^p&  $p\in [2,\infty)$, \label{eq:under-elem-degen}
\\
\kappa'_p |x-y|^2(|x|+|y|)^{p-2} &  $p\in [1,2]$.\label{eq:under-elem-sing}
\end{numcases}
where the constants $\kappa_p$ and $\kappa'_p$ are defined as follows
\begin{align*}
\kappa_p	=  p-1 \quad \text{and}\quad \kappa'_p=\min(2^{-1}, 2^{2-p})\qquad &&p\in [2,\infty)\\
\kappa_p=2^{2-p}(3-p)\leq 2^{3-p}\quad \text{and}\quad  \kappa'_p=p-1&&p\in [1,2].
\end{align*}
\end{lemma}

\noindent An important consequence of  Lemma \ref{lem:elm-est-simons} is the following result
\begin{corollary}\label{cor:elm-est-taylor} For all $x,y\in \R^d,$ the following inequalities hold true
\begin{align*}
|x+y|^p- |x|^p- p|x|^{p-2} x\cdot y&\leq
p\kappa_p |y|^{\min(p,2)} (|y|+|x|)^{\max(p-2,0)}, \\
|x+y|^p- |x|^p- p|x|^{p-2} x\cdot y &\geq 	\frac{p}{2}\kappa'_p |y|^{\max(p,2)} (|y|+|x|)^{\min(p-2,0)}.
\end{align*}
\end{corollary}
\begin{proof}
Let $t\mapsto z_t= y + t(x-y)$ and observe that
\begin{align*}
|x|^p-|y|^p- p|y|^{p-2}y\cdot(x-y)
&=p\int_0^1\big(|z_t|^{p-2}z_t -|y|^{p-2} y\big) \cdot (z_t-y)\, \frac{\d t}{t}.
\end{align*}
It suffices to apply  Lemma \ref{lem:elm-est-simons} together with $|y|+|z_t|\leq (2|y|+|x-y|)$.
\end{proof}

\begin{theorem}
\label{thm:under-upper-form-est}
There holds the following estimates

\medskip

\textbf{Case $p\geq 2$.} For $c_p= 2^{\frac{p(p-2)}{2}}\kappa_p^{\frac{p}{2}}$, we have
\begin{align}\label{eq:under-form-degen}
(\cE_k(u,u-v)-\cE_k(v,u-v)) &\geq \kappa'_p\cE_k(u-v,u-v),%\qquad \text{$p\geq2$}.
\end{align}
\begin{align}\label{eq:upper-form-degen}
(\cE_k(u,u-v)-\cE_k(v ,u-v) )^{\frac{p}{2}}&\leq c_p \cE_k(u-v,u-v)
\big(\cE_k(u,u)
+\cE_k(v,v)\big)^{\frac{p-2}{2}}.
\end{align}

\textbf{Case $1<p<2$.} For  $c'_p=2^{\frac{p(2-p)}{2}} {\kappa'_p}^{\frac{p}{2}} $, we have
\begin{align}\label{eq:under-form-sing}
(\cE_k(u,u-v)-\cE_k(v,u-v))^{\frac{p}{2}}
\geq c'_p \cE_k(u-v,u-v) \big( \cE_k(u,u)+\cE_k(v,v)\big)^{\frac{p-2}{2}} ,
\end{align}
\begin{align}\label{eq:upper-form-sing}
(\cE_k(u,u-v)-\cE_k(v,u-v))
\leq \kappa_p\cE_k(u-v,u-v).
\end{align}

\end{theorem}

\begin{proof}
Throughout we consider $a= u(x)-u(y)$ and $b= v(x)-v(y). $ Assume  $p\geq 2$. The estimate  $(|b|^{p-2}b-|a|^{p-2}a|)(b-a)\geq \kappa'_p|b-a|^p$ (see \eqref{eq:under-elem-degen})  implies
\begin{align*}
(\cE_k(u,u-v)-\cE_k(v,u-v)) &\geq \kappa'_p\cE_k(u-v,u-v).
\end{align*}
Next the inequality  $||b|^{p-2}b-|a|^{p-2}a|\leq 2^{p-2}\kappa_p|b-a|(|a|^p+|b|^p)^{\frac{p-2}{p}}$ (see \eqref{eq:upper-elem-degen})  implies
\begin{align*}
(|b|^{p-2}b-|a|^{p-2}a)(b-a)\leq 2^{p-2}\kappa_p(|b-a|^p)^{\frac{1}{q}}(|a|^p+|b|^p)^{\frac{1}{q'}}
\end{align*}
where we let  $\frac{1}{q}= \frac{2}{p}$ so that $\frac{1}{q'}= \frac{p-2}{p}$. Thus H\"older inequality  yields
\begin{align*}
(\cE_k(u,v-u)-\cE_k(v, v-u)) &\leq 	2^{p-2}\kappa_p \cE_k(u-v,u-v)^{\frac{2}{p}}
\big( \cE_k(u,u)+\cE_k(v,v)\big)^{\frac{p-2}{p}}.
\end{align*}
That is we have
\begin{align*}
(\cE_k(u,u-v)-\cE_k(v,u-v))^{\frac{p}{2}}&\leq 2^{\frac{p(p-2)}{2}}\kappa_p^{\frac{p}{2}} \cE_k(u-v,u-v)\big(\cE_k(u,u)+\cE_k(v,v)\big)^{\frac{p-2}{2}}.
\end{align*}
Now assume $1<p<2$. Using $||b|^{p-2}b-|a|^{p-2}a|\leq \kappa_p|b-a|^{p-1}$ (see \eqref{eq:upper-elem-sing})  implies
\begin{align*}
\big((|b|^{p-2}b-|a|^{p-2}a)(b-a)\big)\leq \kappa_p|b-a|^p.
\end{align*}
The latter yields
\begin{align*}
(\cE_k(u,u-v)-\cE_k(v,u-v)) \leq \kappa_p\cE_k(u-v,u-v).
\end{align*}
The inequality
$(|b|^{p-2}b-|a|^{p-2}a)(b-a)\geq 2^{p-2} \kappa'_p |b-a|^2(|a|^p+|b|^p)^{\frac{p-2}{p}}$ (see \eqref{eq:under-elem-sing}) can be rewritten as
\begin{align*}
c_p |b-a|^p\leq
\big((|b|^{p-2}b-|a|^{p-2}a)(b-a)\big)^{1/q} (|a|^p+|b|^p)^{1/q'},
\end{align*}
where $\frac{1}{q}= \frac{p}{2}$ so that $\frac{1}{q'}= \frac{2-p}{2}$ and $c_p=2^{\frac{p(2-p)}{2}}{\kappa'_p}^{\frac{p}{2}} $.
Thus H\"older inequality  yields
\begin{align*}
\big(\cE_k(u,u-v)-\cE_k(v,u-v)\big)^{\frac{p}{2}}
\big( \cE_k(u,u)+\cE_k(v,v)\big)^{\frac{2-p}{2}}
\geq c_p \cE_k(u-v,u-v).
\end{align*}
\end{proof}

\noindent The next result can be retrieved as a  variant of Theorem  \ref{thm:under-upper-form-est}.
\begin{theorem}
\label{thm:under-upper-form-est-bis}
Let $\alpha_p=\max(1,\frac{p}{2})-1$  and $\beta_p=\min(1,\frac{p}{2})-1$, then we have

\begin{align}\label{eq:upper-form-mixed}
\frac{	(\cE_k(u,u-v)-\cE_k(v,u-v))^{\alpha_p+1}}{	\big( \cE_k(u,u)+\cE_k(v,v)\big)^{\alpha_p} }
\leq 	2^{ p\, \alpha_p}\kappa_p^{\alpha_p+1} \cE_k(u-v,u-v).
\end{align}
\begin{align}\label{eq:under-form-mixed}
\frac{	(\cE_k(u,u-v)-\cE_k(v,u-v))^{\beta_p+1}}{	\big( \cE_k(u,u)+\cE_k(v,v)\big)^{\beta_p} }
\geq 	2^{ p\, \beta_p} {\kappa'_p}^{\beta_p+1}\cE_k(u-v,u-v).
\end{align}
Analogously  \eqref{eq:upper-form-mixed}--\eqref{eq:under-form-mixed} are true for the local form $\cE^0(\cdot,\cdot)$ in place of $\cE_k(\cdot,\cdot)$.
\end{theorem}

\noindent The next result is crucial as it characterizes the asymptotic convergence involving the $\cE^\eps$-forms by uniformly linearizing the nonlinearity.
\begin{theorem}
\label{thm:equiv-conv-in-form}
Assume that $(\cE^\eps(u_\eps, u_\eps))_\eps$  and $(\cE^\eps(u, u))_\eps$  are bounded, that is,
\begin{align*}
0<M:=\sup_{\eps>0} \big(\cE^\eps(u_\eps, u_\eps)+ \cE^\eps(u,u)\big)<\infty,
\end{align*}
with  $(u_\eps)_\eps$ and $u$ given. Then there is a constant $C= C(M,p)>0$ such that
\begin{align*}
C^{-1}\mathcal{E}^\eps(u_\eps-u, u_\eps-u)\leq \big(\mathcal{E}^\eps(u_\eps, u_\eps-u) - \mathcal{E}^\eps(u, u_\eps-u)\big) \leq C\mathcal{E}^\eps(u_\eps-u,u_\eps-u).
\end{align*}
Furthermore,  there holds the equivalence
\begin{align*}
\mathcal{E}^\eps(u_\eps-u, u_\eps-u) \xrightarrow{\eps\to0}0 \quad \text{if and only if} \quad 	\big(\mathcal{E}^\eps(u_\eps, u_\eps-u) - \mathcal{E}^\eps(u,u_\eps-u)\big) \xrightarrow{\eps\to0}0.
\end{align*}
The statement continues to hold for $\cE^\eps_+(\cdot,\cdot)$, $\cE^\eps_\Omega(\cdot,\cdot)$ and $\cE^\eps_{cr}(\cdot,\cdot)$ in place of $\cE^\eps(\cdot,\cdot)$.
\end{theorem}

\begin{proof}
The claims emanate straightforwardly from  Theorem \ref{thm:under-upper-form-est}.
\end{proof}

\subsection{Extension domain and $d$-set}
Let us recall that an open set (not necessarily bounded) $\Omega \subset \R^d$ is called an $W^{1,p}$-extension domain if there exist a linear operator $E:W^{1,p}(\Omega)\to W^{1,p}(\R^d)$ and a constant $C: = C(d,p,\Omega)>0$ such that
\begin{align}\label{eq:w1p-extension}
Eu\mid_{\Omega} &= u \quad\hbox{and} \quad \|Eu\|_{W^{1,p}(\R^d)}\leq C \|u\|_{W^{1,p}(\Omega)} \quad\text{for all}\quad u \in W^{1,p}(\Omega).
\end{align}
\noindent Examples of extension domains include, the half space $\R^d_+= \{(x',x_d)\in \R^d\,: x_d>0 \}$ and  bounded Lipschitz domains. In particular euclidean balls, rectangles in $\R^d$ are extension domains. The geometrical characterization of extension domains has been extensively studied in the last decades.  It turns out that, the $W^{1,p}$-extension property of an open set $\Omega$ infers certain regularity of the boundary $\partial\Omega$. Indeed, according to \cite[Theorem 2]{HKT08}, an important geometrical feature of a  $W^{1,p}$-extension domain $\Omega\subset \R^d$, is that it satisfies  the volume density condition. In other words an  $W^{1,p}$-extension domain $\Omega\subset \R^d$ is  a $d$-set,  i.e.,  there exists a constant $c>0$ such that
\begin{align*}
|\Omega \cap B(x, r)| \geq c r^d  \qquad\text{for all $x \in \partial \Omega$ and $0 < r < 1$.}
\end{align*}
Several references on extension domains for Sobolev spaces can be found in \cite{Zho15}.
It becomes clear in the sequel, thanks to the Lebesgue Differentiation Theorem, that if $\Omega$ is a $d$-set, then it is a Jordan set\footnote{A bounded set $S \subset \R^d$ is Jordan measurable \cite{Zor16} iff $\partial S$ has Lebesgue measure zero.}, meaning that, $|\partial \Omega| = 0$. We follow the approach in \cite{Zho15} or \cite[Lemma 9]{HKT08bis}.

\begin{lemma}\label{lem:d-set-null-set}
If  $\Omega\subset \R^d$ is a $d$-set then $\partial \Omega$ has Lebesgue measure zero,  i.e., 	$|\partial \Omega| = 0.$
\end{lemma}

\begin{proof}
Let $x\in \overline{\Omega}\setminus \Omega =\partial \Omega$ and  consider $(x_j)_j\subset \Omega$ such that  $x_j\to x$ as $j\to \infty$ so that
\begin{align*}
|\Omega \cap B_r(x)|=\lim_{j\to \infty}|\Omega \cap B_r(x_j)| \geq cr^d.
\end{align*}
Thereby, the $d$-set condition yields that for all $x\in \overline{\Omega}$ we have
\begin{align*}
1> \frac{|\Omega \cap B_r(x)|}{|B_r(x)|}\geq   \frac{cr^d}{|B_r(x)|}=\frac{c}{|B_1(0)|}=: c_1>0.
\end{align*}
Accordingly,  we have
\begin{align*}
\frac{|(\overline{\Omega} \setminus\Omega)\cap B(x, r)|}{|B(x, r)|}
= \frac{|\overline{\Omega} \cap B(x, r)|}{|B(x, r)|}  -\frac{|\Omega\cap B(x, r)|}{|B(x, r)|}
&\leq 1 -c_1 < 1.
\end{align*}
In virtue of the Lebesgue differentiation theorem, for almost all  $x\in \partial\Omega= \overline{\Omega}\setminus \Omega $,
\begin{align*}
\mathds{1}_{\partial\Omega}(x)&= \limsup_{r \to 0}\frac{1}{|B_r(x)|}\int_{B_r(x)} \mathds{1}_{\partial\Omega}(y)\d y\\
&= \limsup_{r \to 0} \frac{|(\overline{\Omega} \setminus\Omega)\cap B(x, r)|}{|B(x, r)|}   \leq 1 - c_1 < 1.
\end{align*}
It turns out that each $x\in \overline{\Omega}\setminus \Omega $ is not a Lebesgue point of
$\mathds{1}_{\overline{\Omega}\setminus \Omega}$ and we have
\begin{align*}
\limsup_{r \to 0}\frac{1}{|B_r(x)|}\int_{B_r(x)} \mathds{1}_{\partial\Omega}(y)\d y=	\mathds{1}_{\partial\Omega}(x)=0 .
\end{align*}
However it is well-known that, the set of non-Lebesgue points of $\mathds{1}_{\overline{\Omega}\setminus \Omega}$, which is $ \overline{\Omega} \setminus\Omega=\partial\Omega$, has measure 0, that is $|\partial \Omega|= |\overline{\Omega}\setminus\Omega|=0$.
\end{proof}

\subsection{Collapse across the boundary}
\noindent In this section we investigate the deterioration of the nonlocal form  $\mathcal{E}^{\eps}_{cr}(\cdot,\cdot)$ encoding the collapse phenomenon across the boundary $\partial\Omega$.  Let us first note that \cite[Theorem 5.23]{guy-thesis}, see also the variant in \cite{Brezis-const-function, BBM01,Fog23},
\begin{align}
&\lim_{\eps\to0}\mathcal{E}^{\eps}_{\R^d}(u,u)
% \iil_{\R^d \R^d} |u(x)-u(y)|^p\nu_\eps(x-y)\d y\d x
=K_{d,p}\,\int_{\R^d}|\nabla u(x)|^p\d x.
\label{eq:form-conv-full}
\end{align}
We need the following result involving the collapse across the boundary.
\begin{lemma}\label{lem:collap-bdary}
Assume $\Omega\subset \R^d$ is open and $\partial\Omega= \partial \overline{\Omega}$. Then for  $u,v\in W^{1,p}(\R^d)$,
\begin{align*}
&\lim_{\eps\to0}2\mathcal{E}^{\eps}_{cr}(u,v) = K_{d,p}\,\int_{\partial \Omega}|\nabla u(x)|^{p-2}\nabla u(x)\cdot \nabla v(x)\d x.
\end{align*}
In particular, if  $v\in W^{1,p}_0(\Omega))$ or  $|\partial\Omega|=0$ then we have

\begin{align*}
\lim_{\eps\to0}\iil_{\Omega \Omega^c} |u(y)-u(x)|^{p-2}(u(y)-u(x)) (v(y)-v(x))\nu_\eps(x-y)\d y\,\d x= 0,\\
\lim_{\eps\to0}\iil_{\Omega \Omega^c} |v(y)-v(x)|^{p-2}(v(y)-v(x)) (u(y)-u(x))\nu_\eps(x-y)\d y\,\d x= 0.
\end{align*}
\end{lemma}

\begin{proof}
Since $\partial\Omega= \partial \overline{\Omega}$,  \cite[Theorem 3.5]{Fog23} infers that
\begin{align*}
&\lim_{\eps\to0}2\mathcal{E}^{\eps}_{cr}(w,w)
=K_{d,p}\,\int_{\partial \Omega}|\nabla w(x)|^p\d x\qquad w\in W^{1,p}(\R^d).
\end{align*}
Now we look at the  case $u\neq v$ which  from this. The elementary inequality   $|b|^p- |a|^p- p|a|^{p-2}a(b-a)\geq 0$ (see Corollary \ref{cor:elm-est-taylor}), yields for $t>0$ and $\sigma\in\R$
\begin{align*}
&2\cE_{cr}^\eps(u+t\sigma v, u+t\sigma v)- 2\cE_{cr}^\eps(u, u)- 2p t\cE_{cr}^\eps(u,\sigma v)
\geq0.
\end{align*}
Accordingly, passing to the $\lim\hspace{-0.5ex}\sup$ as  $\eps\to 0$ yields
\begin{align*}
\frac{K_{d,p}}{t} \int_{\partial \Omega} \Big(|\nabla(u+t\sigma v)(x)|^p-|\nabla u(x)|^p\Big)\d x\geq 2p  \limsup_{\eps\to0} \cE_{\Omega}^\eps(u,\sigma v).
\end{align*}
Letting $t\to 0$ and taking $\sigma=\pm 1$ imply
$$\sigma pK_{d,p}\, \int_{\partial \Omega}|\nabla u(x)|^{p-2}\nabla u(x)\cdot \nabla v(x)\d x\geq 2p\limsup_{\eps\to0} \sigma\cE_{\Omega}^\eps(u, v),$$
yielding that
\begin{align*}
&\lim_{\eps\to0} 2\cE_{cr}^\eps(u, v)
= K_{d,p}\,\int_{\partial \Omega}|\nabla u(x)|^{p-2}\nabla u(x)\cdot \nabla v(x)\d x.
\end{align*}
In particular the right hand side vanishes when $|\partial\Omega|=0$.
Now in the case $v\in W^{1,p}_0(\Omega)$, by density argument, it is sufficient to assume that $v\in C_c^\infty(\Omega) $. To this end, consider $\delta= \dist(\partial\Omega, \supp v)>0$ so that by using \eqref{eq:taylor-exapnsion} we have
\begin{align*}
\lim_{\eps\to0}\mathcal{E}^{\eps}_{cr}(u,v) \leq
&\lim_{\eps\to0}\mathcal{E}^{\eps}_{cr}(u,u)^{1/p'}
\mathcal{E}^{\eps}_{cr}(v,v)^{1/p} \\&
\leq 2^p\|u\|^{p/p'}_{W^{1,p}(\R^d)}\|v\|_{L^p(\Omega)}
\lim_{\eps\to0}\Big(\int_{|h|>\delta} \nu_\eps(h)\d h\Big)^{1/p}=0.
\end{align*}
\end{proof}
\begin{remark}
The condition $\partial \Omega=\partial\overline{\Omega}$ prevents $\Omega$ from being on both sides of the domain; and hence rules out domains with  slits on  both $\Omega$ and $\R^d\setminus\overline{\Omega}$. A simple example of domain with a slit is
$\Omega= \{ (x',x_d)\in B_1(0):x_d\neq0\}$ whose slit is the set $S=\{ (x',x_d)\in B_1(0)\,:\, x_d=0\}$ and for which $\partial(\R^d\setminus\Omega)=\partial\Omega= \partial B_1(0)\cup S$, whereas $\partial\overline{\Omega}= \partial\overline{B_1(0)}= \partial B_1(0)$.
\end{remark}

\noindent As a consequence of Lemma \ref{lem:collap-bdary}, we obtain the next result.
\begin{theorem}\label{thm:unif-collapse-bdry}
Let $v\in W^{1,p}(\R^d)$. Assume that $v\in W^{1,p}_0(\Omega)$ or that  $|\partial\Omega|=0$.
Let $(u_\eps)_\eps$, $u_\eps\in W^p_{\nu_\eps}(\Omega|\R^d)$ be an asymptotically bounded sequence,  that is,
\begin{align*}
M:=\sup_{\eps>0}  \cE^\eps_{cr}(u_\eps, u_\eps)<\infty.
\end{align*}
Then $(\cE^\eps_{cr}(u_\eps, v))_\eps$ and $(\cE^\eps_{cr}(v, u_\eps))_\eps$  collapse across  the boundary, meaning that
\begin{align*}
\lim_{\eps\to0}\mathcal{E}^{\eps}_{cr}(u_\eps,v)= \lim_{\eps\to0}\mathcal{E}^{\eps}_{cr}(v,u_\eps) =0.
\end{align*}
\begin{comment}
This remains true in particular if  either (i), (ii) or (iii) holds:
\begin{enumerate}[(i)]
\item $\Omega$ is $W^{1,p}$-extension domain.
\item $\partial \Omega=\partial\overline{\Omega}$ and  $\R^d\setminus \overline{\Omega}$ is an $W^{1,p}$-extension domain.
\item $\Omega$ is any open set and $v\in W^{1,p}_0(\Omega)$.
\end{enumerate}
\end{comment}
\end{theorem}

\begin{proof}
Put $\psi(t)=|t|^{p-2}t$. Using the inequality \eqref{eq:holder-form-estimate} and Lemma \ref{lem:collap-bdary} we get
\begin{align*}
\Big(\iint_{\Omega\Omega^c}
\big| \psi(u_{\eps}(x)-&u_{\eps}(y))(v(x)-v(y))\big|\nu_{\eps}(x-y)\d y\,\d x)^{p}\\
+ \Big(\iint_{\Omega\Omega^c}
\big| \psi(&u_{\eps}(x)-u_{\eps}(y))(v(x)-v(y))\big|\nu_{\eps}(x-y)\d y\,\d x\Big)^{p'}\\
\leq &(M^{p'/p} +M^{p/p'} ) \Big(\iil_{\Omega\Omega^c} |v(x)-v(y)|^p\nu_{\eps}(x-y)\d y\,\d x\Big)
\xrightarrow{\eps\to0}0\,.
\end{align*}
\end{proof}

\subsection{Nonlocal to local convergence of forms}
\label{sec:conv-forms}
Now we aim to investigate the convergence of  the nonlocal forms
$\mathcal{E}^{\eps}(\cdot, \cdot ), $ $\mathcal{E}^{\eps}_+(\cdot, \cdot ) $ and $\mathcal{E}^{\eps}(\cdot, \cdot )$ to  the local form $\cE^0(\cdot, \cdot )$. Let $BV(\Omega)$ is the usual space of function bounded variations on $\Omega$ and $|u|_{BV(\Omega)}$ denotes to the total variation of $u.$ The next result is  reminiscence  of \cite[Theorem 3.37]{guy-thesis} see also \cite{Pon04}.

\begin{theorem}\label{thm:BBM-liminf}
Let  $\Omega\subset \R^d$ be  open. Let   $(u_\eps)_\eps\subset L^p(\Omega)$, $1\leq p<\infty$, be such that
\begin{align*}
\sup_{\eps>0} \big(\|u_\eps\|^p_{L^p(\Omega)}+ \cE^\eps_\Omega(u_\eps,u_\eps)\big)<\infty.
\end{align*}
\noindent If $u_\eps\to u $ in $L^p_{\loc}(\Omega)$  as $\eps\to0$, then  $u\in W^{1,p}(\Omega)$ for $1<p<\infty$ and $u\in BV(\Omega)$  for $p=1$. Moreover, $u$  satisfies
\begin{align*}
K_{d,p}\|\nabla u\|^p_{L^p(\Omega)}\leq \liminf_{\eps\to0}\cE^\eps_\Omega(u_\eps,u_\eps)
&\quad \text{ for} \quad 1<p<\infty\\
K_{d,1}|u|_{BV(\Omega)}\leq \liminf_{\eps\to0}\cE^\eps_\Omega(u_\eps,u_\eps)
&\quad \text{ for} \quad p=1.
\end{align*}
\end{theorem}
\begin{proof}
By Fatou's Lemma, $\|u\|_{L^p(\Omega)}\leq \liminf_{\eps\to0}\|u_\eps\|_{L^p(\Omega)}$, i.e. $u\in L^p(\Omega)$. Let $(\Omega_j)_j$ be an exhaustion  of open bounded subsets of $ \Omega$. Since $\|u_\eps-u\|_{L^p(\Omega_j)}\xrightarrow{\eps\to0}0$, $j\geq1$, then by  \cite[Theorem 3.37]{guy-thesis} see also \cite{Pon04} we have $u\in W^{1,p}(\Omega_j)$, for $1<p<\infty$,
\begin{align*}
K_{d,p}\|\nabla u\|^p_{L^p(\Omega_j)}
&\leq \liminf_{\eps\to0}\iil_{ \Omega_j\Omega_j} |u_\eps(x)-u_\eps(y)|^p\nu_\eps(x-y)\d y\,\d x \\
&\leq \liminf_{\eps\to0}\iil_{ \Omega\Omega} |u_\eps(x)-u_\eps(y)|^p\nu_\eps(x-y)\d y\,\d x.
\end{align*}
Letting $j\to\infty$ yields  the sought estimate and  thus that $u\in W^{1,p}(\Omega)$. The case $p=1$ follows analogously.
\end{proof}

\noindent The next result combines the ideas of \cite[Theorem 1.5]{Fog23} and \cite[Lemma 2.8]{FeSa20}.
\begin{theorem}\label{thm:BBM-dual-limit}  Let $\Omega\subset \R^d$ be open and $u,v\in W^{1,p}(\R^d)$.
Assume either that one of the following conditions $(i)$,  $(ii)$ or $(iii)$ holds.
\begin{enumerate}[$(i)$]
\item $\Omega$ is an $W^{1,p}$-extension domain.
\item $\partial \Omega=\partial\overline{\Omega}$ and  $\R^d\setminus \overline{\Omega}$ is an $W^{1,p}$-extension domain.
\item $u\in W^{1,p}_0(\Omega)$ or $v\in W^{1,p}_0(\Omega)$.
\end{enumerate}
Then the following limits hold
\begin{align}
&\lim_{\eps\to0}\mathcal{E}^{\eps}_{\Omega}(u,v)
=\lim_{\eps\to0}\mathcal{E}^{\eps}(u,v) = \lim_{\eps\to0}\mathcal{E}^{\eps}_+(u,v) = K_{d,p}\,\mathcal{E}^{0}(u,v). \label{eq:form-conv-plus}
\end{align}
\end{theorem}
\begin{remark}
It is worth mentioning that, the convergence above is achieved when either $\partial\Omega$ is sufficiently regular, or when  $u$ or $v$ are sufficiently smooth near the boundary. No regularity is required on $\partial\Omega$ when $u$ or $v$ vanishes on  $\partial\Omega$.
As illustrated in $(iii)$,  $\partial\Omega$  need not be regular when $u$ or $v$ vanishes on  $\partial\Omega$.
\end{remark}
%\medskip

\begin{proof}
\noindent
Given that, the boundary of an extension domains is a null set, in the cases $(i)$ and $(ii)$ we have $|\partial\Omega|=|\partial\overline{\Omega}|=0$.   Thus in either cases $(i), (ii)$ and $(iii)$ Lemma \ref{lem:collap-bdary} implies that $\cE^{\eps}_{cr}(v,u)\to 0$ and $\cE^{\eps}_{cr}(u,v) \to0$ as $\eps\to0$. Now observe that
\begin{align*}
\mathcal{E}^{\eps}(u,v)&= \mathcal{E}^{\eps}_\Omega(u,v)+\mathcal{E}^{\eps}_{cr}(u,v)+\cE^{\eps}_{cr}(v,u),\\
\mathcal{E}^{\eps}_+(u,v)&= \mathcal{E}^{\eps}_\Omega(u,v)+\mathcal{E}^{\eps}_{cr}(u,v).
\end{align*}
Thence, it is sufficient to establish the result solely for
$\mathcal{E}^{\eps}_\Omega(u,v)$. The elementary inequality $|b|^p- |a|^p- p|a|^{p-2}a(b-a)\geq 0$ (see Corollary \ref{cor:elm-est-taylor}) yields,
\begin{align*}
&\cE_{\Omega}^\eps(u+t\sigma v, u+t\sigma v)- \cE_{\Omega}^\eps(u, u)- p t\cE_{\Omega}^\eps(u,\sigma v)
\geq0\qquad\text{$t>0$,\, $\sigma\in\R$}.
\end{align*}
If  $\Omega$ is a $W^{1,p}$-extension domain then by  \cite[Theorem 5.23]{guy-thesis}, \cite[Theorem 1.2]{Fog23} or  the variant in \cite{BBM01} we find that $\cE^\eps_\Omega(w, w) \xrightarrow{\eps\to0} K_{d,p}\cE^0(w, w)$ for $w\in W^{1,p}(\Omega)$. Accordingly, by Theorem \ref{thm:BBM-liminf}, passing to the $\lim\hspace{-0.5ex}\inf$  yields
\begin{align*}
K_{d,p}\frac{\|\nabla(u+t\sigma v)\|^p_{L^p(\Omega)}-\|\nabla u\|^p_{L^p(\Omega)}}{t}\geq p  \limsup_{\eps\to0} \cE_{\Omega}^\eps(u,\sigma v).
\end{align*}
Letting $t\to 0$ and taking $\sigma=\pm 1$ yield $\lim_{\eps\to0} \cE_{\Omega}^\eps(u, v)= K_{d,p}\cE^0(u, v)$ since
\begin{align*}
\sigma pK_{d,p}\cE^0(u, v)\geq p\limsup_{\eps\to0} \sigma\cE_{\Omega}^\eps(u, v).
\end{align*}
\end{proof}

\noindent We obtain the following nonlocal characterization of Sobolev spaces.

\begin{theorem}[{Nonlocal characterization of Sobolev spaces}]\label{thm:BBM-nonloc-conseq}
Let $\Omega\subset \R^d$ be open and  $1\leq p<\infty$. For some  $\eps_0>0$, consider the spaces\footnote{Note  that $W^{1,p}_{\mathrm{n,sup}}(\Omega)$ is always a Banach space which may not be the case for $W^{1,p}_{\mathrm{n,inf}}(\Omega)$. }
\begin{align*}
W^{1,p}_{\mathrm{n,sup}}(\Omega)&=\{u\in L^p(\Omega)\,:\,
\|u\|_{W^{1,p}_{\mathrm{n, sup}}(\Omega)}<\infty \}\, \, \Big(=\bigcap_{\eps\in (0, \eps_0)} W^{p}_{\nu_\eps}(\Omega)\Big),\\
W^{1,p}_{\mathrm{n,inf}}(\Omega)&=\{u\in L^p(\Omega)\,:\,
\|u\|_{W^{1,p}_{\mathrm{n,inf}}(\Omega)}<\infty \}.
\end{align*} where the norms
$\|\cdot\|_{W^{1,p}_{\mathrm{n,sup}}(\Omega)}$  and $\|\cdot\|_{W^{1,p}_{\mathrm{n,inf}}(\Omega)}$ are defined by
\begin{align*}
\|u\|_{W^{1,p}_{\mathrm{n,sup}}(\Omega)}
&=\Big( \|u\|^p_{L^p(\Omega)}+ \sup_{\eps\in (0, \eps_0)} \iil_{ \Omega \Omega}|u(x)-u(y)|^p\nu_\eps(x-y)\d y \d x \Big)^{1/p}\\
\|u\|_{W^{1,p}_{\mathrm{n,inf}}(\Omega)}
&=\Big( \|u\|^p_{L^p(\Omega)}+ \liminf_{\eps\to 0} \iil_{ \Omega \Omega}|u(x)-u(y)|^p\nu_\eps(x-y)\d y \d x \Big)^{1/p},
\end{align*}
 Especially in the  fractional case,  we have $ W^{1,p}_{\mathrm{n,sup}}(\Omega)= \bigcap_{s\in (0,1)} W^{s,p}(\Omega)$.  Then the following embeddings are  continuous
\begin{align*}
W^{1,p}_{\mathrm{n,sup}}(\Omega)
\hookrightarrow W^{1,p}_{\mathrm{n,inf}}(\Omega)
\hookrightarrow
\begin{cases}
W^{1,p}(\Omega) & p\neq1,\\
BV(\Omega) &p=1.
\end{cases}
\end{align*}
If in addition  $\Omega$ is a $W^{1,p}$-extension domain for $p\neq1$ (resp. a $BV$-extension domain for $p=1$) then the spaces coincide with equivalence of norms, that is,
\begin{align*}
W^{1,p}_{\mathrm{n,sup}} (\Omega) =W^{1,p}_{\mathrm{n,inf}}(\Omega) =
\begin{cases}
W^{1,p}(\Omega) & p\neq1,\\
BV(\Omega) &p=1.
\end{cases}
\end{align*}
\end{theorem}

\begin{proof}
We only prove for $1<p<\infty$ as the case $p=1$ is analogous. By \cite[Theorem 3.3]{Fog23}  if $u\in W^{1,p}_{\mathrm{n,inf}} (\Omega)$, then $u\in W^{1,p}  (\Omega)$ and we have
\begin{align*}
K_{d,p}\int_\Omega|\nabla u(x)|^p\d x
&\leq  \liminf_{\eps\to 0} \iil_{ \Omega \Omega}|u(x)-u(y)|^p\nu_\eps(x-y)\d y \d x\\
&\leq  \sup_{\eps\in (0,\eps_0)} \iil_{ \Omega \Omega}|u(x)-u(y)|^p\nu_\eps(x-y)\d y \d x.
\end{align*}
This proves the embeddings. Moreover, when $\Omega$ is an $W^{1,p}$-extension domain, then by Theorem \ref{thm:BBM-dual-limit} (see also \cite[Theorem 1.3]{Fog23}) we find that
\begin{align*}
K_{d,p}\int_\Omega|\nabla u(x)|^p\d x
&=\liminf_{\eps\to 0} \iil_{ \Omega \Omega}|u(x)-u(y)|^p\nu_\eps(x-y)\d y \d x\\
&\leq  \sup_{\eps\in (0,\eps_0)} \iil_{ \Omega \Omega}|u(x)-u(y)|^p\nu_\eps(x-y)\d y \d x \leq C \int_\Omega|\nabla u(x)|^p\d x,
\end{align*}
where the last estimate follows from \cite[Lemma 3.47]{guy-thesis} or \cite[Lemma  2.12]{Fog23}.
\end{proof}

\subsection{Asymptotic demi-convergence}
Now we investigate the asymptotics of the nonlocal forms $\cE^\eps(u,v)$ and $\cE^\eps_\Omega(u,v)$ when $u$ or $v$ also depends on $\eps$.  Let us begin with the following elementary result.
\begin{lemma}\label{lem:elem-taylor-expan}
Let   $a,b\in \R$,  $t\neq 0$ and $R>0$.The following assertions are true.
\begin{enumerate}
\item For all  $a,b\in \R$, there holds that
\begin{align*}
|a+b|^p= |a|^p+ p|a|^{p-2}ab+O(|b|^p)+ \sum_{k=2}^{\lfloor p\rfloor}\binom{p}{k}|a|^{p-2k} a^{k}b^{k},
\end{align*}
where $\lfloor p\rfloor=\max\{m\in\mathbb{N}\,:\, m\leq p\}$ and $\binom{p}{k}=\frac{\Gamma(p+1)}{k!\Gamma(p-k+1)}$.
\item  If $1<p< 2$ then for  all $|b|\leq R$  and $a, t\in \R$  we have
\begin{align*}
|a+tb|^p=|a|^p+ pt|a|^{p-2}ab+O(|tR|^p).
\end{align*}
If $p\geq 2$ then for all $|a|\leq R$, $|b|\leq R$ and all $|t|\leq 1$ we have
\begin{align*}
|a+tb|^p=|a|^p+ pt|a|^{p-2}ab+O(|t|^2R^p).
\end{align*}
\item
For all $|a|\leq R$,  $|b|\leq R$ and $|t|\leq 1$ we have
\begin{align*}
|a+tb|^p- |a|^p- pt|a|^{p-2}ab= O(t^{\min(p,2)}R^p).
\end{align*}
\end{enumerate}
Note that the constant behind $O(\cdot)$ only depends on $p$.
\end{lemma}

\begin{proof}
Note that  $(2)$ and $(3)$  are implied by $(1)$ or can also  be easily retrieved from Corollary \ref{cor:elm-est-taylor}.  Now, consider the  continuous function $\zeta:\R\to \R$ with
\begin{align*}
\zeta(t)=\frac{|1+t|^p-  \sum_{k=0}^{\lfloor p\rfloor}\binom{p}{k}t^{k} }{|t|^p}.
\end{align*}
Note that, if $\lfloor p\rfloor<p$ then $\zeta(t)\xrightarrow{|t|\to\infty}1$  and if $p=\lfloor p\rfloor$  then $\zeta(t)\xrightarrow{|t|\to\infty}0$,
whereas,  using Taylor's expansion for $|t|<1$ we deduce
\begin{align*}
\zeta(t)=\frac{(1+t)^p-  \sum_{k=0}^{\lfloor p\rfloor}\binom{p}{k}t^{k} }{|t|^p}
= \frac{\sum_{k=\lfloor p\rfloor+1}^\infty\binom{p}{k}t^{k} }{|t|^p}\xrightarrow{|t|\to0}0.
\end{align*}
Therefore, the map $t\mapsto\zeta(t)$ is bounded say $|\zeta(t)|\leq C$ for some $C>1$, yielding
\begin{align*}
|1+t|^p\leq C|t|^p+ \sum_{k=0}^{\lfloor p\rfloor}\binom{p}{k}t^{k}.
\end{align*}
In other words we have
\begin{align*}
|1+t|^p= 1+pt+ \sum_{k=2}^{\lfloor p\rfloor}\binom{p}{k}t^{k}+  O(|t|^p).
\end{align*}
From this, it is a routine to deduce $(1)$, that is,  for all $a,b\in\R$  we have
\begin{align*}
|a+b|^p= |a|^p+ pb|a|^{p-2}a+O(|b|^p)+ \sum_{k=2}^{\lfloor p\rfloor}\binom{p}{k}|a|^{p-2k} a^{k}b^{k}.
\end{align*}
\end{proof}

\noindent The following result generalizes that of
\cite[Lemma 2.8]{FeSa20}, yielding a sort of nonlocal-to-local asymptotic demi-convergence of the forms $\cE^\eps(\cdot,\cdot)$  and $\cE^\eps_+(\cdot,\cdot)$ to $\cE^0(\cdot,\cdot)$.
\begin{theorem} \label{thm:asymp-conv-form}
Let $u, v\in W^{1,p}(\R^d)$ and the sequence $ (v_\eps)_\eps$ with  $v_\eps \in W_{\nu_\eps}^p(\Omega|\R^d)$ be such that  $v_\eps|_\Omega \to v|_\Omega$ in $L^p_{\loc}(\Omega)$ as $\eps\to0$ and
\begin{align*}
\sup_{\eps>0} \big( \|v_\eps\|^p_{L^p(\Omega)}+ \cE^\eps(v_\eps, v_\eps)\big)<\infty.
\end{align*}
Assume either that one of the following conditions $(i)$,  $(ii)$ or $(iii)$ holds.
\begin{enumerate}[$(i)$]
\item $\Omega$ is an $W^{1,p}$-extension domain.
\item $\partial \Omega=\partial\overline{\Omega}$ and  $\R^d\setminus \overline{\Omega}$ is an $W^{1,p}$-extension domain.
\item We have $u\in W^{1,p}_0(\Omega)$.
\end{enumerate}
Then  the following convergences hold true
\begin{align*}
\lim_{\eps\to0}\cE^\eps(u, v_\eps) =\lim_{\eps\to0}\cE^\eps_+(u, v_\eps)
= K_{d,p} \cE^0(u, v)\quad\text{for all $u\in W^{1,p}(\R^d)$}.
\end{align*}
\end{theorem}

\begin{proof}
We only prove for $\cE^\eps(\cdot,\cdot)$ since  the one for $\cE^\eps_+(\cdot,\cdot)$  is analogous.  We apply Lemma \ref{lem:elem-taylor-expan} with $a= U(x,y)= u(x)-u(y)$ and $b=t\sigma V_\eps(x,y)= t\sigma (v_\eps(x)-v_\eps(y))$ with $|t|\leq 1$ and $ \sigma =\pm1 $ yielding
\begin{align*}
0\leq \cE^\eps(u+t\sigma v_\eps, u+t\sigma v_\eps)-\cE^\eps(u, u)-p\cE^\eps(u, t\sigma v_\eps)= O( t^{p}\cE^\eps(u , u))+ F^\eps_p (u,tv_\eps).
\end{align*}
The remainder $F^\eps_p (u,tv_\eps)$ is given by: $F^\eps_p (u,tv_\eps)=0$ for $1<p<2$ or else by
\begin{align*}
F^\eps (u,tv_\eps) &=  \sum_{k=2}^{\lfloor p\rfloor}\binom{p}{k} t^k
\iil_{(\Omega^c\times \Omega^c)^c}\hspace{-2ex}
|U(x,y)|^{p-2k} U(x,y)^{k}V_\eps(x,y)^{k}\nu_\eps(x-y)\d y \d x,
\end{align*}
 Using the H\"older inequality with $q= \frac{p}{p-k}$ and $q'= \frac{p}{k}$ we obtain
\begin{align*}
|F^\eps (u,tv_\eps)|
&\leq \sum_{k=2}^{\lfloor p\rfloor}\binom{p}{k}  |t|^k \cE^\eps(u,u)^{p-k} \cE^\eps(v_\eps,v_\eps)^k.
\end{align*}
Next, by \eqref{eq:taylor-exapnsion} we have $\cE^\eps(u,u)\leq 2^p \|u\|_{W^{1,p}(\R^d)}$, while $(\cE^\eps(v_\eps,v_\eps))_\eps$ is uniformly bounded. Thus, there exists $ M > 0$ such that for $|t|\leq 1$,
\begin{align*}
\sup_{\eps>0} \cE^\eps(u+ t\sigma v_\eps,u+ t\sigma v_\eps)\leq 2^p\sup_{\eps>0} \big(\cE^\eps(u,u) + \cE^\eps(v_\eps,v_\eps)\big)\leq M.
\end{align*}
From this and the penultimate estimate we deduce
\begin{align*}
|F^\eps (u,tv_\eps)|
&\leq \sum_{k=2}^{\lfloor p\rfloor}\binom{p}{k}  |t|^k \cE^\eps(u,u)^{p-k} \cE^\eps(v_\eps,v_\eps)^k\leq C|t|^2,
\end{align*}
for some generic constant $C>0$ only depending on $M$ and $p.$ It follows that
\begin{align*}
0\leq \cE^\eps(u+t\sigma v_\eps, u+t\sigma v_\eps)&-\cE^\eps(u, u)-pt\cE^\eps(u, \sigma v_\eps)\leq  C |t|^{p}+ C|t|^2.
\end{align*}
Therefore, for  all $0<t<1$, we get that
\begin{align*}
\cE^\eps(u+t\sigma v_\eps, u+t\sigma v_\eps)-\cE^\eps(u, u)  =   pt\cE^\eps(u, \sigma v_\eps)+ O( t^{\min(p,2)}).
\end{align*}
Observing that $u+t\sigma v_\eps\to u+t\sigma v$ as $\eps\to0$ in $L^p_{\loc}(\Omega)$ and  $\sup_{\eps>0} \big(\cE^\eps(u+ t\sigma v_\eps,u+ t\sigma v_\eps)\big)\leq 2^pM$,  we apply
Theorem \ref{thm:BBM-liminf} and Theorem \ref{thm:BBM-dual-limit} so  that
\begin{align*}
p \liminf_{\eps\to0}\cE^\eps(u, \sigma v_\eps)
&= \liminf_{\eps\to0}	\frac{\cE^\eps(u+t\sigma v_\eps, u+t\sigma v_\eps)-\cE^\eps(u, u)}{t}- O( t^{\min (p-1,1)}) \\
&\geq \liminf_{\eps\to0}	\frac{\cE^\eps_\Omega(u+t\sigma v_\eps, u+t\sigma v_\eps)-\cE^\eps(u, u)}{t}- O( t^{\min (p-1,1)}) \\
& \geq K_{d,p}	\frac{\cE^0(u+t\sigma v, u+t\sigma v)-\cE^0(u, u)}{t} - O( t^{\min (p-1,1)}).
\end{align*}
Letting $t\to 0$ yields $ \lim_{\eps\to0}\cE^\eps(u,v_\eps)=  K_{d,p}	\cE^0(u, v)$ since $ \sigma =\pm1$  and
\begin{align*}
p  \sigma \liminf_{\eps\to0}\cE^\eps(u,v_\eps) \geq p \sigma  K_{d,p}	\cE^0(u, v).
\end{align*}
\end{proof}

\noindent The following  is the  analog of Theorem \ref{thm:asymp-conv-form} for the form $\cE^\eps_\Omega(\cdot,\cdot)$, the nonlocal-to-local asymptotic demi-convergence of the forms $\cE^\eps_\Omega (\cdot,\cdot)$ to $\cE^0(\cdot,\cdot)$.
\begin{theorem}\label{thm:asymp-conv-reg-form}
Let $\Omega\subset \R^d$ be open,  $u,v\in W^{1,p}(\Omega)$ . Let the sequence $ (v_\eps)_\eps$ with  $v_\eps \in W_{\nu_\eps}^p(\Omega)$ such that  $v_\eps \to v$ in $L^p_{\loc}(\Omega)$ as $\eps\to0$ and
\begin{align*}
\sup_{\eps>0} \big(\|v_\eps\|^p_{L^p(\Omega)}+ \cE^\eps_\Omega(v_\eps,v_\eps)\big)<\infty.
\end{align*}
If $u\in W^{1,p}_0(\Omega)$ or $\Omega$ is an $W^{1,p}$-extension domain then we have
\begin{align*}
\lim_{\eps\to0}\cE^\eps_\Omega(u,  v_\eps)
= K_{d,p}\cE^0(u, v).
\end{align*}
\end{theorem}

\vspace{1mm}
\par  In view of Theorem \ref{thm:BBM-liminf} and Theorem \ref{thm:asymp-conv-reg-form} we obtain the following compactness result, which refines \cite[Theorem 5.35 \& 5.40]{guy-thesis} as well as \cite[Theorem 1.2]{Pon04}.
\begin{theorem}\label{thm:asymp-compactness}
Let $\Omega\subset \R^d$ is any open set. In the case $d=1$ we assume in addition that $(\nu_\eps)_\eps$ satisfies the following uniform lower bound condition:
\begin{align*}
\text{There exist $c_0, \theta_0 \in (0,1)$ such that}
\quad
\inf_{0 < \varepsilon < 1} \inf_{\theta_0 \leq \theta \leq 1} \nu_\varepsilon(\theta x) \geq c_0
\quad \text{for all } x \in \R.
\end{align*}
 Assume that  $(u_\eps)_\eps\subset L^p(\Omega)$, $1\leq p<\infty$ satisfies
$\sup_{\eps>0} \big(\|u_\eps\|^p_{L^p(\Omega)}+ \cE^\eps_\Omega(u_\eps,u_\eps) \big)<\infty.$
\noindent Then there exist $u\in L^p(\Omega)$ with,  $u\in W^{1,p}(\Omega)$ if $1<p<\infty$ and  $u\in BV(\Omega)$ if $p=1$, and a subsequence $(\eps_n)_n$ with $\eps_n\to0$ such that $u_{\eps_n}\to u $ in $ L^p_{\loc}(\Omega)$ and
\begin{align*}
& \liminf_{\eps\to0}\cE^{\eps}_\Omega( u_\eps, u_\eps) \geq K_{d,p} \cE^{0}( u, u),\quad 1<p<\infty,\\
%\quad \text{and}\quad
&\liminf_{\eps\to0}\cE^{\eps}_\Omega( u_\eps, u_\eps) \geq K_{d,1} |u|_{BV(\Omega)},\quad p=1.
\end{align*}

\noindent 	In addition, the following  assertions hold.
\begin{enumerate}[$(i)$]
\item  If $\Omega$ is a $W^{1,p}$-extension set then $\cE^{\eps_n}_\Omega( v, u_{\eps_n}) \to K_{d,p} \cE^{0}( v, u)$, $v\in W^{1,p}(\Omega)$.
\item If $\Omega=\R^d$ then  $\|u_{\eps_n}-u\|_{L^p(\Omega')}\xrightarrow{n\to\infty}0$ for $\Omega'\subset \R^d $ with   $|\Omega'|<\infty$.
\item If $\Omega$ is bounded Lipschitz  then  $\|u_{\eps_n}-u\|_{L^p(\Omega)}\xrightarrow{n\to\infty}0$.
\item Further if $\Omega$ is bounded Lipschitz  and,  $(u_\eps)_\eps\subset L^p(\R^d)$ satisfies $u_\eps=0$ a.e. on $\Omega^c$ for each $\eps$, and $\sup_{\eps>0}\|u_\eps\|_{ W^p_{\nu_\eps}(\R^d)}<\infty$ then $ u\in W^{1,p}_0(\Omega)\cap W^{1,p}(\R^d)$.
\end{enumerate}
\end{theorem}
Note that $(iv)$ follows the upcoming Theorem \ref{thm:conv-trace}.

\vspace{-1ex}
%%%%%%%%%%%%%%%%%%%%%%%%%%%%%%%%
\section{Robustness of  nonlocal trace spaces} \label{sec:robust-trace}
%%%%%%%%%%%%%%%%%%%%%%%%%%%%%%%%
In this section we establish the robustness of  our nonlocal trace space of $\WnuOmRa$, as $\eps\to0$ toward the trace space of the trace space $W^{1 ,p}(\partial\Omega)$ with $1\leq p<\infty$. Note that in the case $d=1$ we require the additional condition:
\begin{align*}
\text{There exist $c_0, \theta_0 \in (0,1)$ such that}
\quad
\inf_{0 < \varepsilon < 1} \inf_{\theta_0 \leq \theta \leq 1} \nu_\varepsilon(\theta x) \geq c_0
\quad \text{for all } x \in \R.
\end{align*}
\subsection{Local trace theorem}
Let us quickly recall the following version of classical trace theorem, which will be useful  later.
 \begin{theorem}[{Classical trace theorem, see \cite[Chap III]{BF13}}]\label{thm:trace-loc-thm}
Assume $\Omega\subset \R^d$ is bounded Lipschitz. There exists a linear and continuous trace operator $\gamma_0: W^{1,p}(\Omega)\to L^p(\partial\Omega)$ such that  $\gamma_0 u= u|_{\partial\Omega}$ for all $u\in C^1(\overline{\Omega})\cap W^{1,p}(\Omega)$  and
\begin{align*}
\|\gamma_0(u)\|_{ L^p(\partial \Omega)}\leq C\|u\|^{1-1/p}_{L^p(\Omega)}
\|u\|^{1/p}_{W^{1,p}(\Omega)}.
\end{align*}
Moreover, the following assertions are true
\begin{enumerate}[$(a)$]
\item  $ \ker(\gamma_0 ) = W^{1,p}_0(\Omega)$, and  $\gamma_0 (W^{1,p}(\Omega))\overset{def}{=} W^{1,1-1/p }(\partial \Omega)$.
\item The mappings $W^{1,p}(\Omega) \xrightarrow{\gamma_0}W^{1-1/p ,p}(\partial\Omega)
\xrightarrow{\operatorname{Id}} L^p(\partial \Omega)$
are continuous.
\item The operator $\gamma_0$ has a right  lifting, i.e., there is $R_0:W^{1-1/p ,p}(\partial\Omega)\to W^{1,p}(\Omega)$ bounded and continuous mapping such that $ \gamma_0\circ R_0= Id$.
\end{enumerate}
\end{theorem}
\noindent It is important to emphasize that a lifting operator $ R_0 $ can be chosen to be linear (though not necessarily) only when
$1 < p < \infty $, whereas $R_0$ can never be linear in the case $ p = 1$. This is known as the Peetre's theorem \cite{Pee79}; a more  recent and more accessible proof can be found \cite{PeWo02}. For obvious and  practical reasons, we equip the local trace space $W^{1-1/p ,p}(\Omega)$  the following norm
\begin{align*}
\|g\|_{W^{1-1/p ,p}(\partial \Omega)}= \inf\big\{ \|u\|^*_{W^{1,p}(\Omega)} \,\,:\,\, \gamma_0(u)= g\quad \text{a.e on $\partial \Omega$}\quad u\in W^{1,p}(\Omega)\big\}
\end{align*}
where we consider the Sobolev norm
\begin{align*}
\|u\|^*_{W^{1,p}(\Omega)} =\Big(\|u\|^p_{L^p(\Omega)}+ K_{d,p}\|\nabla u\|^p_{L^p(\Omega)}\Big)^{1/p}.
\end{align*}
Let us recall that when $\Omega$ is bounded Lipschitz and  $1<p<\infty$, it is possible to obtain the explicit equivalence of the following trace norm, see for instance \cite{MiRu15,Mir18},
\begin{align*}
\|u\|_{W^{1-1/p ,p}(\partial\Omega)}\asymp\Big( \int_{\partial \Omega} |u(x)|^p\d \sigma(x)+  \iint_{\partial \Omega\, \partial \Omega} \frac{|u(x)-u(y)|^p}{|x-y|^{d+ p-2}}\d \sigma(y)\, \d \sigma(x)\Big)^{1/p}.
\end{align*}
Interestingly, according to \cite{Gag57}, it turns out that in the case $p=1$,   the trace space $W^{0,1}(\partial\Omega)$ of $W^{1,1}(\Omega)$ coincides with $L^1(\partial\Omega)$, i.e. $\gamma_0 (W^{1,1}(\Omega))= L^1(\partial\Omega)$; we strongly recommend see \cite{Mir15} for elegant proof of the latter statement.

\subsection{Nonlocal trace theorem} In the nonlocal situation, given $\nu$ a symmetric $p$-L{\'e}vy kernel, i.e., $\nu\geq 0$, $\nu(h)=\nu(-h)$ and  $\nu\in L^1(\R^d,1\land |h|^p)$,  it is much natural to consider the nonlocal Sobolev space
\begin{align*}
\WnuOmR=\big\lbrace u: \R^d \to \R \text{ meas.} \,: \, u|_\Omega\in L^p(\Omega)\,\!\text{ and }\,\!  \cE(u,u) <\infty \big\rbrace,
\end{align*}
where $\cE(\cdot,\cdot )$ is the nonlinear nonlocal form
\begin{align*}
\cE(u,u)&= \iil_{(\Omega^c\times\Omega^c)^c} |u(x)-u(y)|^p \nu(x-y)\d y \d x.
\end{align*}
In addition we also consider $\WnuOmRO$ the space of elements $u\in \WnuOmR$ such that $u=0$ a.e. on $\Omega^c$ that is,
\begin{align*}
\WnuOmRO
&= \{ u\in \WnuOmR~: ~u=0~~\text{a.e. on } \R^d\setminus \Omega\}.
\end{align*}
\noindent To be rigorous, it is natural  to identify $\WnuOmR\equiv W^p_\nu(\Omega|\Omega_\nu)$, where $\Omega_\nu=\Omega+\supp\nu $ and  $W^p_\nu(\Omega|\Omega_\nu)$ is the space of measurable functions $u:\Omega_\nu^*= \Omega_\nu\cup\Omega\to\R$ such that $u|_\Omega\in L^p(\Omega)$ and $\cE_+(u,u)<\infty$ with
\begin{align*}
\cE_+(u,u)	= \iil_{\Omega\Omega_\nu} |u(x)-u(y)|^p \nu(x-y)\d y \d x=\iil_{\Omega\R^d} |u(x)-u(y)|^p
\nu(x-y)\d y \d x.
\end{align*}
In this way th space $\WnuOmR$ can be endowed with one the equivalent seminorms
\begin{align*}
\|u\|_{\WnuOmR}&= \Big( \|u\|^p_{L^p(\Omega)}+ \cE(u,u)\Big)^{1/p} \hspace{-1ex}\asymp   \Big( \|u\|^p_{L^p(\Omega)}+  \cE_+(u,u)	\Big)^{1/p}
\hspace{-1ex}=: \|u\|_{W^p_\nu(\Omega|\Omega_\nu)}.
\end{align*}
\noindent Interested reader is referred  to \cite{Fog25,guy-thesis,FoKa24,Fog21b} for further details on these spaces.
Due to its nonlocal character, the nonlocal  trace space of $\WnuOmR$ assumes functions defined on $\R^d\setminus\Omega$, or strictly speaking on the nonlocal boundary $\Omega_{\nu}^e=\Omega_\nu\setminus\Omega$ of $\Omega$ with respect to $\nu$.  The main reason is that elements of $\WnuOmR$ are essentially defined on $\R^d$, or strictly speaking  on the nonlocal hull
$\Omega_\nu^{*} =\Omega_\nu\cup\Omega$ when taking into account the support of $\nu$; see \cite{Fog25}.
\begin{definition}[Trace space of $\WnuOmR$]
The trace space of $\WnuOmR$ denoted  $\TnuOm$ is the space of restrictions to $\R^d\setminus \Omega$ of functions of $\WnuOmR$. More precisely,
\begin{align*}
\TnuOm = \{ v = u|_{\Omega^c} ~~\hbox{with }~~ u \in \WnuOmR\}.
\end{align*}
By  identifying  $\WnuOmR\equiv W^p_\nu(\Omega|\Omega_\nu)$ leads to the identification $\TnuOm\equiv T^p_\nu(\Omega_{\nu}^e)$ where $\Omega_{\nu}^e=\Omega_\nu\setminus\Omega$  is the nonlocal boundary of $\Omega$ with respect to $\nu$ and
\begin{align*}
T^p_\nu(\Omega_{\nu}^e)
= \{ v = u|_{\Omega_{\nu}^e} ~~\hbox{with }~~ u \in W^p_\nu(\Omega|\Omega_\nu)\}.
\end{align*}
We endow $\TnuOm $ with its natural seminorm,
\begin{align*}
\|v\|_{\TnuOm } &= \inf\{ \|u\|_{\WnuOmR }: ~~ u \in \WnuOmR ~~ \hbox{ with }~~ v = u|_{\Omega^c} \}\\
&= \inf\{ \|u\|_{W^p_\nu(\Omega|\Omega_\nu) }: ~~ u \in W^p_\nu(\Omega|\Omega_\nu)~~ \hbox{ with }~~ v = u|_{\Omega^e_{\nu}}\}.
\end{align*}
\end{definition}
\noindent Of course, if the space $(\WnuOmR, \|\cdot\|_{\WnuOmR})$	is a Banach space then so is the space $(\TnuOm,  \|\cdot\|_{\TnuOm})$.  Let us state the nonlocal trace theorem in connection with the trace space $\TnuOm$. Further detailed explanation can be found in \cite{Fog25,FoKa24}.
\begin{theorem}[Nonlocal trace theorem]\label{thm:trace-nonloc-thm}
Let $\Omega\subset \R^d$ be open and  $\omega_\nu \in \{\widetilde{\nu}, \overline{\nu}\}$ where
\begin{align*}
\widetilde{\nu}(x)&=\int_{\Omega}(1\land\nu(x-y)) \d y
\qquad\text{and}\qquad \overline{\nu}(x)= \operatorname{ess}\inf_{y\in \Omega}\nu(x-y).
\end{align*}
The embedding $\WnuOmR\hookrightarrow L^p(\R^d,\omega_\nu)$ is continuous. Furthermore, the nonlocal trace operator $u \mapsto \operatorname{Tr}(u) = u\mid_{\Omega^c}$ fulfills the following properties.
\begin{enumerate}[$(a)$]
\item $ \ker(\operatorname{Tr} ) = \WnuOmRO$ and $\operatorname{Tr}(\WnuOmR)= \TnuOm$.
\item The mappings $\WnuOmR \xrightarrow{\operatorname{Tr}}\TnuOm\xrightarrow{\operatorname{Id}} L^p(\Omega^c,\omega_\nu)$
are continuous.
\item If $\Omega$ is bounded in one direction, then there is a bounded and continuous mapping $R: \TnuOm\to \WnuOmR$ such that $ \operatorname{Tr}\circ R= Id$.
\end{enumerate}
\end{theorem}

\begin{remark}
Note that the kernel $\nu$ is not necessarily fully supported in $\R^d$. Consequently, the sets $\Omega^c$ and
$\Omega^e_\nu := \Omega_\nu \setminus \Omega$ do not necessarily coincide.
Using the rigorous identifications
$\WnuOmR \equiv W^p_{\nu}(\Omega | \Omega_\nu) $ and $\TnuOm \equiv T^p_{\nu}(\Omega^e_\nu),$
we also have $L^p(\Omega^c,\omega_\nu) \equiv L^p(\Omega^e_\nu,\omega_\nu),$
since one readily verifies that $\omega_\nu = 0$ on $\Omega^c \setminus \Omega_\nu$. In this setting, the nonlocal trace operator $\operatorname{Tr}$ reduces to the restriction $\operatorname{Tr}(u) = u\big|_{\Omega^e_\nu},$ and therefore $\operatorname{Tr}\big(W^p_\nu(\Omega | \Omega_\nu)\big) = T^p_{\nu}(\Omega^e_\nu).$
Moreover, $\WnuOmRO \equiv W^p_{\nu,0}(\Omega | \Omega_\nu)$
is precisely the subspace of $W^p_\nu(\Omega | \Omega_\nu)$ consisting of functions vanishing on
$\Omega^e_\nu = \Omega_\nu \setminus \Omega$.
For simplicity of notation, we shall continue to write $L^p(\Omega^c,\omega_\nu)$, $\TnuOm$, $\WnuOmR$, $\WnuOmO$, as well as $\operatorname{Tr}$ and $R$, as if $\nu$ were fully supported. This slight abuse of notation does not affect any of the arguments or results presented here, all statements remain true in both settings (see \cite{FoKa24,Fog25} for further details).

\end{remark}

\subsection{Robustness of the nonlocal trace space}
In this section, we prove that  our formulation of the nonlocal trace space is understood in the robust sense. In other words,  prove the convergence of the nonlocal trace space $\TnuOma\equiv T^p_{\nu_\eps}(\Omega^e_{\nu_\eps})$, with $1\leq p<\infty$, to the local  trace space  $W^{1-1/p ,p}(\partial \Omega) $ including the case $p=1$.
Furthermore, we also establish the existence of a robust right lifting  $R_\eps =\operatorname{Ext}_\eps$ for the nonlocal trace operator $\operatorname{Tr}$, i.e.,   $R_\eps= \operatorname{Ext}_\eps$ strongly converges to  $ R_0=\operatorname{Ext}$, which is also a right lifting of the local trace operator $\gamma_0$.   Let us recall that the (semi)norms we consider on the  (non)local trace spaces are given by
\begin{align*}
\|g\|_{T^p_{\nu_\eps}(\Omega^c)} &= \inf\big\{ \|u\|_{W^p_{\nu_\eps}(\Omega|\R^d)} \,\,:\,\, u= g\quad \text{a.e on $\Omega_{\nu_\eps}\setminus \Omega$}\quad u\in W^p_{\nu_\eps}(\Omega|\R^d)\big\},
\\
\|g\|_{W^{1-1/p ,p}(\partial \Omega)}
&= \inf\big\{ \|u\|^*_{W^{1,p}(\Omega)} \,\,:\,\, \gamma_0(u)= \gamma_0(g)\quad \text{a.e on $\partial \Omega$}\quad u\in W^{1,p}(\Omega)\big\},
\end{align*}
where we consider the Sobolev norm
\begin{align*}
\|u\|^*_{W^{1,p}(\Omega)} =\Big(\|u\|^p_{L^p(\Omega)}+ K_{d,p}\|\nabla u\|^p_{L^p(\Omega)}\Big)^{1/p}.
\end{align*}
\begin{theorem}[Robust trace spaces]
\label{thm:conv-trace}
Assume $\Omega\subset\R^d$ is open with a compact Lipschitz  boundary. The nonlocal trace spaces $(T^p_{\nu_\eps}(\Omega^c))_\eps$, $1\leq p<\infty$,  converge to the local trace $W^{1-1/p ,p}(\partial\Omega)$ in the  following sense:

\noindent If $\|\operatorname{Tr}(g_\eps-g)\|_{\TnuOma}\to 0$ as $\eps\to0$, for $g\in W^{1,p}(\R^d)$ and $g_\eps\in \WnuOmRa$ then
\begin{align*}
\lim_{\eps\to0}	\|\operatorname{Tr}(g_\eps)\|_{\TnuOma}=
\|\gamma_0(g)\|_{W^{1-1/p ,p}(\partial\Omega)}.
\end{align*}
This is true in particular if $g_\eps\to g$ in $W^{1,p}(\R^d\setminus \overline{\Omega})$ or merely  $g_\eps=g$.
\end{theorem}
\begin{proof}
Since $\big|\|\operatorname{Tr}(g_\eps)\|_{\TnuOma} -\|\operatorname{Tr}(g)\|_{\TnuOma}\big|\to 0$, it is sufficient to prove for the case $g_\eps=g$. Let $u\in W^{1,p}(\Omega)$ with  $\gamma_0(u)= \gamma_0(g)$. Define $ E_gu\in W^{1,p}(\R^d)\subset W^p_{\nu_\eps}(\Omega|\R^d)$ by
\begin{align*}
E_g u(x)=\begin{cases}
u(x)& x\in \Omega,\\
g(x)& x\in \R^d\setminus \Omega.
\end{cases}
\end{align*}
By convergence of forms $\cE^\eps(\cdot,\cdot)$ from  Theorem \ref{thm:BBM-dual-limit} we find that
\begin{align*}
\lim_{\eps\to0}\cE^\eps(E_g u, E_g u) =K_{d,p}\cE^0(E_g u, E_g u)=   K_{d,p}\|\nabla u\|^p_{L^{p}(\Omega)}.
\end{align*}
Inasmuch as $E_gu |_{\Omega^c}=  \operatorname{Tr}(g)$ and $\|g\|^p_{T^p_{\nu_\eps}(\Omega^c)}  \leq  \|E_g u\|^p_{W^p_{\nu_\eps}(\Omega|\R^d)}$ we get

\begin{align*}
\|g\|^p_{T^p_{\nu_\eps}(\Omega^c)}
%&\leq  \|E_g u\|^p_{W^p_{\nu_\eps}(\Omega|\R^d)}\\
&\leq \cE^\eps(E_g u, E_g u) + \|u\|^p_{L^p(\Omega)}
\xrightarrow{\eps\to0} K_{d,p}\|\nabla u\|^p_{L^{p}(\Omega)}+\|u\|^p_{L^p(\Omega)}.
\end{align*}
In other words, it follows that
\begin{align*}
\limsup_{\eps\to0}	\|g\|_{T^p_{\nu_\eps}(\Omega^c)}\leq  \limsup_{\eps\to0} \|E_g u\|_{W^p_{\nu_\eps}(\Omega|\R^d)}= \| u\|^{*}_{W^{1,p}(\Omega)},
\end{align*}
wherefrom  we deduce the limsup inequality
\begin{align}\label{eq:limsup-trace}
\limsup_{\eps\to0}	\|g\|_{T^p_{\nu_\eps}(\Omega^c)}\leq
\|\gamma_0(g)\|_{W^{1-1/p ,p}(\partial\Omega)}.
\end{align}
Next for  each $0<\eps<1$ consider  $u_\eps \in W^p_{\nu_\eps}(\Omega|\R^d)$ such that $u_\eps= g$ a.e. on $\Omega^c$ and
\begin{align*}
\|u_\eps\|_{W^p_{\nu_\eps}(\Omega)}\leq  \|u_\eps\|_{W^p_{\nu_\eps}(\Omega|\R^d)}
\leq \|g\|_{T^p_{\nu_\eps}(\Omega^c )} + \eps.
\end{align*}
Note that, the uniform estimate  \eqref{eq:taylor-exapnsion}  implies that
\begin{align*}
\|g\|_{T^p_{\nu_\eps}(\Omega^c )} \leq  \|g\|_{W^p_{\nu_\eps}(\Omega|\R^d)}
\leq  \|g\|_{W^p_{\nu_\eps}(\R^d )} \leq 2^{1+1/p}\|g\|_{W^{1,p}(\R^d)}.
\end{align*}
By definition of $u_\eps$ and the fact that $u_\eps = g$ on $\Omega^c$ we find that
\begin{align*}
\|u_\eps\|_{W^p_{\nu_\eps}(\R^d)}
&=\big(\|u_\eps\|^p_{W^p_{\nu_\eps}(\Omega|\R^d)}+ \|g\|^p_{W^p_{\nu_\eps}(\Omega^c)}\big)^{1/p}\\
&\leq
\|g\|_{T^p_{\nu_\eps}(\Omega^c)}+\eps+ \|g\|_{W^p_{\nu_\eps}(\R^d)}.
\end{align*}
From this we deduce that,  there is $M>0$ such that
\begin{align*}
\sup_{0<\eps<1}\|u_\eps\|_{W^p_{\nu_\eps}(\R^d)}\leq  M.
\end{align*}
According to Theorem \ref{thm:asymp-compactness}, there exist  $\eps_j\to 0$ and $u\in W^{1,p}(\R^d)$ such that
\begin{align*}
&\|u_{\eps_j}- u\|_{L^p(\Omega')}\xrightarrow{\eps_j\to 0}0,\qquad\text{for every bounded set  $\Omega'\subset \R^d$,}
\\
&K_{d,p}\|\nabla u\|^p_{L^p(\Omega)}\leq\liminf_{\eps\to0} \cE_{\Omega}^{\eps}(u_\eps,u_\eps),\\
&K_{d,p}\|\nabla u\|^p_{L^p(\R^d)}\leq\liminf_{\eps\to0} \cE_{\R^d}^{\eps}(u_\eps,u_\eps).
\end{align*}
In particular, the above liminf implies
\begin{align*}
\|u\|^*_{W^{1,p}(\Omega)}\leq \liminf_{\eps\to0} \|u_{\eps}\|_{W^{p}_{\nu_\eps}(\Omega)}\leq  \liminf_{\eps\to0} \|g\|_{T^{p}_{\nu_\eps}(\Omega^c)}.
\end{align*}
Next, take a large ball $B\subset \R^d$, neighborhood of $\partial\Omega$, i.e., $\partial \Omega\subset B$. Then $\Omega'=B\setminus \Omega$ (with $\partial\Omega'=\partial \Omega\cup \partial B$)  is also bounded Lipschitz. The  trace Theorem \ref{thm:trace-loc-thm} implies
\begin{align*}
\|\gamma_0(g)-\gamma_0(u)\|_{L^p(\partial\Omega)}
&\leq \|\gamma_0(g)-\gamma_0(u)\|_{L^p(\partial\Omega')}\\
& \leq C\|g-u\|_{W^{1,p}(\Omega')}\leq C
\liminf_{j \to\infty} \|g-u_{\eps_j}\|_{W^{1,p}(\Omega')}=0.
\end{align*}
Inasmuch as  $u_{\eps_j}=g$ a.e. on $\Omega^c$.
Hence   $\gamma_0(g)=\gamma_0(u)$ with $u\in W^{1,p}(\Omega)$ and
\begin{align*}
\|\gamma_0(g)\|_{W^{1-1/p ,p}(\partial\Omega)}\leq \|u\|_{W^{1,p}(\Omega)}\leq \liminf_{\eps\to0} \|u_{\eps}\|_{W^{p}_{\nu_\eps}(\Omega|\R^d)}
\leq  \liminf_{\eps\to0} \|g\|_{T^{p}_{\nu_\eps}(\Omega^c)}.
\end{align*}
This together with \eqref{eq:limsup-trace} yields  the sought result
\begin{align*}
\lim_{\eps\to0} \|g\|_{T^{p}_{\nu_\eps}(\Omega^c)}= \|\gamma_0(g)\|_{W^{1-1/p ,p}(\partial\Omega)}.
\end{align*}
\end{proof}

\begin{remark}\label{rem:ext-plus-trace}
It is worth recalling that any open set $\Omega\subset \R^d$ with a compact Lipschitz boundary is a $W^{1,p}$-extension domain in the sense of \eqref{eq:w1p-extension}, i.e., there  is a continuous extension operator $E: W^{1,p}(\Omega)\to W^{1,p}(\R^d)$. Consider the lifting of the trace operator
$R_0: W^{1-1/p ,p}(\partial\Omega)\to W^{1,p}(\Omega)$ as given in Theorem \ref{thm:trace-loc-thm}. It clearly appears that the extension operator $\widehat{E}=E\circ R_0: W^{1-1/p ,p}(\partial\Omega)\to W^{1,p}(\R^d)$ is also continuous and bounded. In addition we find that
\begin{align*}
\gamma_0(\widehat{E} (g)) &= \gamma_0\circ R_0 (g)= g\quad\text{for all $g\in W^{1-1/p ,p}(\partial\Omega)$}.
\end{align*}
Moreover, since the continuous embeddings $W^{1,p}(\R^d) \hookrightarrow W^{p}_\nu(\R^d) \hookrightarrow \WnuOmR$ clearly hold, it follows  that $\operatorname{Tr}\circ E\circ R_0: W^{1-1/p ,p}(\partial\Omega)\to \TnuOm $ with $g\mapsto  \widehat{E}(g)|_{\Omega^c}$ is also continuous and bounded.
On the other hand, if $\R^d\setminus \overline{\Omega}$ is $W^{1,p}$-extension domain, i.e., there  is a continuous extension operator $E^*:W^{1,p}(\R^d\setminus \overline{\Omega})\to W^{1,p}(\R^d)$ continuous then one readily obtains the continuous embedding $\operatorname{Tr}\circ E^* : W^{1,p}(\R^d\setminus \overline{\Omega})\to \TnuOm$ with $\operatorname{Tr}\circ E^* (g) = g$
for all $g\in W^{1,p}(\R^d\setminus \overline{\Omega})$.
\end{remark}

\noindent Using the notations of Remark \ref{rem:ext-plus-trace}, we obtain the following corollary.
\begin{corollary}\label{cor:conv-trace}
Assume $\Omega\subset\R^d$ be open with compact Lipschitz boundary.
Then the nonlocal trace spaces $(T^p_{\nu_\eps}(\Omega^c))_\eps$, $1\leq p<\infty$, converge to the classical  local trace  space
$W^{1-1/p ,p}(\partial\Omega)$ in the following sense: for all $g\in W^{1-1/p ,p}(\partial\Omega)$ we have
\begin{align*}
\lim_{\eps\to0}	\|\widehat{E}(g)\|_{T^p_{\nu_\eps}(\Omega^c)}=
\|g\|_{W^{1-1/p ,p}(\partial\Omega)}.
\end{align*}
Moreover, this does not depend on the choice of the extension of $g$. Namely, if $\widehat{g}\in W^{1,p}(\R^d)$ is another  extension of $g$ such that $\gamma_0(\widehat{g})= g$ then we have
\begin{align*}
\lim_{\eps\to0}	\|\widehat{E}(g) -\widehat{g}\|_{T^p_{\nu_\eps}(\Omega^c)}=0.
\end{align*}
\end{corollary}
\begin{proof}
Note in passing that, for $u\in W^{1,p}(\Omega)$ we have $\gamma_0(Eu)= \gamma_0(u)$ since $Eu= u$ a.e. on $\Omega$ and $Eu\in W^{1,p}(\R^d)$. Since $R_0$ is the right lifting of $\gamma_0$ we find that
\begin{align*}
\gamma_0\circ \widehat{E} (g)  = \gamma_0( E (R_0(g))) = \gamma_0\circ R_0 (g)= g.
\end{align*}
Therefore, Theorem \ref{thm:conv-trace} implies
\begin{align*}
\lim_{\eps\to0}	\|\widehat{E}(g)\|_{T^p_{\nu_\eps}(\Omega^c)}
=\|\gamma_0\circ\widehat{E}(g)\|_{W^{1-1/p ,p}(\partial\Omega)}=
\|g\|_{W^{1-1/p ,p}(\partial\Omega)}.
\end{align*}
Now observing that $\gamma_0(\widehat{E}(g)- \widehat{g})=0$,  Theorem \ref{thm:conv-trace} also implies
\begin{align*}
\lim_{\eps\to0}	\|\widehat{E}(g) -\widehat{g}\|_{T^p_{\nu_\eps}(\Omega^c)}
=0.
\end{align*}
\vspace{-3mm}
\end{proof}

\noindent Next, we establish the existence and the  robustness of a right lifting to the nonlocal trace operator $ \operatorname{Tr}$. It is worth mentioning that, in general, the right lifting may not be uniquely determined.
\begin{theorem}[Robust right lifting]
Assume $\Omega\subset\R^d$ be open  bounded with Lipschitz boundary. There exist $R_\eps: \TnuOma\to \WnuOmRa$ and $R_0:W^{1-1/p , p}(\partial \Omega)\to W^{1,p}(\Omega)$
right lifting of the (non)local trace operators $\operatorname{Tr}$ and $\gamma_0$  respectively such that
\begin{align*}
\lim_{\eps\to0}\|R_\eps(g)-R_0(g)\|_{\WnuOma}= \lim_{\eps\to0}\|R_\eps(g)-R_{0,g}(g)\|_{\WnuOmRa}=0,
\end{align*}
for all $g\in W^{1,p}(\R^d\setminus\overline{\Omega})$. Here $u_g = R_{0,g}(g)\in W^{1,p}(\R^d)$ is the $g$-extension of $u$ outside $\Omega$, that is  $u_g= R_{0,g}(g)= g$ a.e. on $\Omega^c$. In fact, we merely  take $u_\eps= R_\eps(g)$  and  $u= R_0(g)$ to be respectively weak solutions to the (non)local Dirichlet problems
\begin{align*}
L_\eps u_\eps &=0 \quad \text{in $\,\Omega$ \, and } \quad u_\eps= g\quad \text{on $\, \Omega^c$}, \\
-\Delta_p u&=0 \quad \text{in $\, \Omega$\,\,\, and } \quad u= g\quad \text{on $\,\partial \Omega$}.
\end{align*}
\end{theorem}

\begin{proof}
Clearly by definition of non(local) traces we have $\operatorname{Tr}(u_\eps)= g$ on $\Omega^c$ and $\gamma_0(u)= \gamma_0(g)$ on $\partial\Omega$. In particular, $\operatorname{Tr}\circ R_\eps =Id$ and $\gamma_0 \circ R_0 =Id$. Repeating the arguments from  \cite[Sec. 8, \& Sec. 11]{Fog25}, reveals that each mapping $g\mapsto R_\eps(g)$ and, analogously  $g\mapsto R_0(g)$ are continuous and bounded, in fact,
\begin{align*}
\|R_\eps(g)\|_{\WnuOmRa} &=\|u_\eps\|_{\WnuOmRa}
\leq C\|g\|_{\TnuOma},\\
\|R_0(g)\|_{W^{1,p}(\Omega)}&= \|u\|_{W^{1,p}(\Omega)}
\leq C\|g\|_{W^{1-1/p ,p}(\partial\Omega)},
\end{align*}
for some generic constant $C>0$ not depending on $\eps$.
We conclude that $R_\eps$ and $R_0$ are right lifting for the operator  $\operatorname{Tr}$ and $\gamma_0$ respectively.  Finally, as  we show in the upcoming
Section \ref{sec:conv-weak-solution}, we have
\begin{align*}
\lim_{\eps\to0}\|u_\eps-u\|_{\WnuOma}= \lim_{\eps\to0}\|u_\eps-u_g\|_{\WnuOmRa}=0.
\end{align*}
\end{proof}

 Finally, using the norm $\|u\|^*_{W^{1,p}(\Omega)}$ above we obtain the convergence of dual spaces.
\begin{theorem}[Robust dual spaces]\label{thm:conv-dual-spaces}
Let $\Omega\subset \R^d$ be open. For  $f\in (L^p(\Omega))'$ we have
\begin{align*}
\| f\|_{(W^{1,p}_0(\Omega))'}\leq \liminf_{\eps\to0} \| f\|_{ (W^p_{\nu_\eps,0}(\Omega|\R^d))'}.
\end{align*}
If $\Omega\subset \R^d$ is an $W^{1,p}$-extension domain
 we have
\begin{align*}
\| f\|_{(W^{1,p}(\Omega))'}\leq \liminf_{\eps\to0} \| f\|_{(\WnuOma)'}.
\end{align*}
However, if in addition $\Omega\subset \R^d$ is bounded Lipschitz then we have
\begin{align*}
\| f\|_{(W^{1,p}(\Omega))'}=\lim_{\eps\to0} \| f\|_{(\WnuOma)'} \quad\text{and}\quad
\| f\|_{(W^{1,p}_0(\Omega))'}
&=\lim_{\eps\to0} \| f\|_{(W^p_{\nu_\eps,0}(\Omega|\R^d))'}.
\end{align*}
\end{theorem}

\begin{proof}
If $\Omega$ is a $W^{1,p}$-extension domain then, see \cite{guy-thesis,Fog23} and Theorem \ref{thm:BBM-dual-limit},  $(\WnuOma)'\subset (W^{1,p}(\Omega))'$ and  $\| u\|_{\WnuOma} \to \|u\|^*_{W^{1,p}(\Omega)}$, $u\in W^{1,p}(\Omega)$ so that
\begin{align*}
|\langle f,u\rangle|\leq \liminf_{\eps\to0}\big( \|f\|_{(\WnuOma)'} \cdot \| u\|_{\WnuOma} \big)= \|u\|^*_{W^{1,p}(\Omega)}\liminf_{\eps\to0}\|f\|_{(\WnuOma)'}.
\end{align*}
Thus, $\| f\|_{(W^{1,p}(\Omega))'}\leq \liminf\limits_{\eps\to0} \| f\|_{(\WnuOma)'}. $ The first $\operatorname{liminf}$ follows analogously. Now assume $\Omega$ is bounded Lipschitz. Consider  $u_\eps\in \WnuOma$ such that $\|u_\eps\|_{\WnuOma}=1$ and $  \| f\|_{(\WnuOma)'}-\eps \leq |\langle f,u_\eps\rangle|$.  By the compactness Theorem \ref{thm:asymp-compactness} we can assume there is $u\in W^{1,p}(\Omega)$ such that $\|u\|^*_{W^{1,p}(\Omega)}\leq \liminf\limits_{\eps\to0}\|u_\eps\|_{\WnuOma}=1$ and $\|u_\eps-u\|_{L^{p}(\Omega)}\xrightarrow{\eps\to0}0$. The latter implies $ \langle f,u_\eps-u\rangle\to 0$ since $f\in (L^p(\Omega))'$ and hence we get
$$\limsup\limits_{\eps\to0} \| f\|_{(\WnuOma)'}\leq |\langle f,u\rangle|\leq\| f\|_{(W^{1,p}(\Omega))'}.$$ Likewise, if we instead selected
$u_\eps\in W^p_{\nu_\eps,0}(\Omega |\R^d)$, i.e. $u_\eps=0$ a.e. on $\Omega^c$, such that $\|u_\eps\|_{W^p_{\nu_\eps}(\R^d)}=\|u_\eps\|_{W^p_{\nu_\eps,0}(\Omega |\R^d)}=1$ and $  \| f\|_{(W^p_{\nu_\eps,0}(\Omega |\R^d))'}-\eps \leq |\langle f,u_\eps\rangle|$, then $u\in W^{1,p}(\R^d)$ and $u=0$ on a.e. $\Omega$, so that Theorem \ref{thm:conv-trace} implies  $\gamma_0 (u)=0$, i.e., $u\in W^{1,p}_0(\Omega)$. From this we also deduce that
$$\limsup\limits_{\eps\to0} \| f\|_{(W^p_{\nu_\eps,0}(\Omega |\R^d))'}\leq |\langle f,u\rangle|\leq\| f\|_{(W^{1,p}_0(\Omega))'}.$$
\end{proof}
%\vspace{-2mm}

%%%%%%%%%%%%%%%%%%%%%%%%%%%%%%%
\section{Definition: Asymptotic weak convergence}  \label{sec:asymp-conv-weak}
%%%%%%%%%%%%%%%%%%%%%%%%%%%%%%%
Given Banach spaces $X$ and $X_\eps$, we aim to define a form of weak convergence for elements
$ \ell_\eps \in X'_\eps $ converging to an element $\ell \in X' $ as  $\varepsilon \to 0$. The primary challenge arises from the dependence of $X_\eps $ on $ \varepsilon$, where of course $X_\eps'$  and $X'$ are respectively the dual space of $X_\eps$ and $X$. Our objective is to preserve key features of standard weak convergence in a fixed Banach space, including: (a) boundedness of the sequence, (b) the weak convergence property, and (c) the weak-strong type lemma.

To alleviate the notations we simply adopt the following notional conventions.  For simplicity we write $\langle\cdot,\cdot \rangle_\eps$ (resp. $\langle\cdot,\cdot \rangle$) to denote the dual pairing between $X_\eps'$ and $X_\eps$ (resp. $X'$ and $X$).  Moreover, it should become clear  from the context that $X_\eps$ denotes one of the spaces $W^p_{\nu_\eps}(\Omega)$, $W^p_{\nu_\eps}(\Omega|\R^d)$, $W^p_{\nu_\eps, \Omega}(\Omega|\R^d)$  or  $T^p_{\nu_\eps}(\Omega^c)$, while  $X$ denotes one of the spaces $W^{1,p}(\Omega)$, $W^{1,p}_0\Omega)$ or $W^{1-1/p ,p}(\partial\Omega)$. Furthermore, the restriction of functions on $\Omega^c$ or on $\partial\Omega$  is understood in the sense of the traces. In other words, we merely write $\langle g,v\rangle$ instead of $\langle g,\gamma_0(v )\rangle$ for  $v\in W^{1,p}(\R^d)$ and $g\in (W^{1-1/p ,p}(\partial\Omega))'$  and $\langle g_\eps,v\rangle_\eps $ instead of $\langle g,\operatorname{Tr}(v )\rangle_\eps$   if $v\in \WnuOmRa $ and $g\in (\TnuOma)'$.

\vspace{2mm}
\par Let us recall that $W^{1,p}(\R^d)\hookrightarrow W^p_{\nu_\eps}(\Omega|\R^d)$, while, if $\Omega$ is an $W^{1,p}$-extension domain we have $W^{1,p}(\Omega)\hookrightarrow W^p_{\nu_\eps}(\Omega)$. Hence  $(W^p_{\nu_\eps}(\Omega))'\subset (W^p_{\nu_\eps}(\Omega|\R^d))'\subset (W^{1,p}(\Omega))'$.This motivates the following definition.

\begin{definition}\label{def:weak-asymp-form}
We say that  $(f_\eps)_\eps$ with $f_\eps\in (W^p_{\nu_\eps}(\Omega|\R^d))'$ (resp. $f_\eps\in (W^p_{\nu_\eps}(\Omega))'$) asymptotically weakly converges to $f\in (W^{1,p}(\Omega))'$ if the following conditions hold
\begin{enumerate}[$(F_1)$]
\item \textbf{Asymptotically bounded:}   $(f_\eps)_\eps$ is asymptotically bounded,  that is,
\begin{align*}
\sup_{\eps>0} \| f_\eps\|_{(W^p_{\nu_\eps}(\Omega|\R^d))'}<\infty \quad \big(\text{resp.}\quad \sup_{\eps>0} \| f_\eps\|_{(W^p_{\nu_\eps}(\Omega))'}<\infty \big).
\end{align*}
\item \textbf{$L^p$-Weak-strong type lemma:} For  a sequence $v$, $(v_\eps)_\eps$ with $v_\eps\in\WnuOmRa$ and $v\in W^{1,p}(\R^d)$ (resp. $v_\eps\in \WnuOma$ and $v\in W^{1,p}(\Omega)$) such that $v_\eps|_\Omega\to v|_\Omega$  in $L^p(\Omega)$  as $\eps\to 0$ we have
\begin{align*}
\lim_{\eps\to0} \langle f_\eps , v_\eps \rangle_\eps
= \langle f , v\rangle .
\end{align*}
In particular $(f_\eps)_\eps$  satisfies the  usual \textbf{$W^{1,p}$-weak convergence}, that is, for   any $v\in W^{1,p}(\R^d)$ (resp. $v\in W^{1,p}(\Omega)$)  we have
\begin{align*}
\lim_{\eps\to0} \langle f_\eps , v\rangle_\eps  = \langle f , v\rangle.
\end{align*}
\end{enumerate}
The asymptotic weak convergence of $ (f_\eps)_\eps $ with $ f_\eps \in ( W^p_{\nu_\eps, \Omega}(\Omega | \R^d))' $ to$ f \in W^{1,p}_0(\Omega)$  is defined, \textit{mutatis mutandis}, in a similar manner.
\end{definition}

\begin{example}
The conditions of Definition \ref{def:weak-asymp-form}  hold true  in the situation where  $f,f_\eps\in L^{p'} (\Omega)$ (in particular if $f_\eps=f$) and
$f_\eps \rightharpoonup f$ weakly in $L^{p'}(\Omega)$, that is
\begin{align*}
\lim_{\eps\to0} \int_{\Omega}f_\eps(x) v(x)\d x
= \int_{ \Omega}f (x)v(x)\d x\qquad\text{for all $v\in L^{p}(\Omega)$}.
\end{align*}
By the boundedness principle, the weak convergence of the sequence $(f_\eps)_\eps$ implies the boundedness of the sequence $(\| f_\eps\|_{L^{p'}(\Omega)})_\eps$ and hence it is readily seen that
\begin{align*}
\sup_{\eps>0}  \| f_\eps\|_{(W^p_{\nu_\eps}(\Omega|\R^d))'}\leq \sup_{\eps>0} \| f_\eps\|_{(W^p_{\nu_\eps}(\Omega))'}\leq
\sup_{\eps>0}  \| f_\eps\|_{L^{p'}(\Omega)}<\infty.
\end{align*}
Now if $(v_\eps)_\eps$  is a sequence such that $v_\eps \to v$ in $L^p(\Omega)$, then  the standard weak-strong lemma implies $\int_{\Omega}f_\eps(x) v_\eps (x)\d x
\to \int_{ \Omega}f (x)v(x)\d x$ since
\begin{align*}
 \int_{\Omega}|f_\eps(x) (v_\eps (x)-v(x))|\d x
 \leq  \sup_{\eps>0}  \| f_\eps\|_{L^{p'}(\Omega)} \|v_\eps-v\|_{L^p(\Omega)}\xrightarrow{\eps\to0}0.
\end{align*}
\end{example}

 \begin{definition}\label{def:asymp-strong-Dir}
 We say that a  sequence $ (g_\eps)_\eps$ with $g_\eps\in  \TnuOma$ converges asymptotically strongly to $g\in W^{1,p}(\R^d\setminus\overline{\Omega})$, if
\begin{align*}
\lim_{\eps\to0}	\|g_\eps-g\|_{T^p_{\nu_\eps}(\Omega^c)}=0.
\end{align*}
If $\Omega\subset \R^d$ is bounded Lipschitz, we say that  $ (g_\eps)_\eps$ with $g_\eps\in  \TnuOma$ converges asymptotically strongly to $g\in W^{1-1/p,p}(\partial\Omega)$ if
\begin{align*}
\lim_{\eps\to0}	\|g_\eps-\widehat{E}(g)\|_{T^p_{\nu_\eps}(\Omega^c)}=0.
\end{align*}
where $\widehat{E}(g)$ is defined as in Remark \ref{rem:ext-plus-trace}.   It is worth mentioning that, the asymptotic boundedness of $(g_\eps)_\eps$ is automatic under mild regularity assumptions on $\Omega$; for example, it holds whenever $\R^d \setminus \overline{\Omega}$ is a $W^{1,p}-$extension domain.
\end{definition}
\begin{example}
Let $g_\eps, g\in W^{1,p}(\R^d\setminus \overline{\Omega})$. Assume $\R^d\setminus \overline{\Omega}$ is a $W^{1,p}$-extension domain. If $g_\eps\to  g$ strongly in $W^{1,p}(\R^d\setminus \overline{\Omega})$. By Remark \ref{rem:ext-plus-trace} there is $C>0$ such that
\begin{align*}
	\|g_\eps-g\|_{T^p_{\nu_\eps}(\Omega^c)}   	\leq C\|g_\eps-g\|_{W^{1,p}(\R^d\setminus \overline{\Omega})}\to 0
\end{align*}
Namely $\|g_\eps-g\|_{T^p_{\nu_\eps}(\Omega^c)}\xrightarrow{\eps\to0}0$.
Likewise, if $g\in W^{1-1/p, p}(\partial \Omega)$,   $\Omega$ bounded Lipschitz and  $g_\eps\to  \widehat{E}(g)$ strongly in $W^{1,p}(\R^d\setminus \overline{\Omega})$
we also have $ \|g_\eps-\widehat{E}(g)\|_{T^p_{\nu_\eps}(\Omega^c)}
\xrightarrow{\eps\to0}0.$
\end{example}

\begin{definition}\label{def:weak-asymp-trace}
We say that a sequence  $(g_\eps)_\eps$ with $g_\eps\in (T^p_{\nu_\eps}(\Omega^c))'$ asymptotically weakly converges to $g\in (W^{1-1/p , p}(\partial\Omega))'$ if
the following conditions hold.
\begin{enumerate}[$(G_1)$]
\item \textbf{Asymptotically bounded:} The sequence $(g_\eps)_\eps$ satisfies
\begin{align*}
\sup_{\eps>0} \| g_\eps\|_{(T^p_{\nu_\eps}(\Omega^c))'}<\infty .
\end{align*}

\item \textbf{$L^p$-Weak-strong type lemma:}
For $v\in W^{1,p}(\R^d)$  and a sequence $(v_\eps)_\eps$ with $v_\eps\in\TnuOma$  such that $v_\eps|_\Omega\to v|_\Omega$  in $L^p(\Omega)$  as $\eps\to 0$ we have
\begin{align*}
\lim_{\eps\to0} \langle g_\eps , v_\eps \rangle_\eps  = \langle g , v\rangle
\end{align*}
In particular, the sequence $(g_\eps)_\eps$  satisfies \textbf{$W^{1,p}$-weak convergence}, that is, for all $v\in W^{1,p}(\R^d)$ we have $\lim\limits_{\eps\to0} \langle g_\eps , v\rangle_\eps = \langle g.  v\rangle.$
\end{enumerate}
\end{definition}

\noindent Note that by duality Theorem \ref{thm:trace-loc-thm} implies $L^{p'}(\partial\Omega)\hookrightarrow (W^{1-1/p ,p}(\partial \Omega))'$, while Theorem \ref{thm:trace-nonloc-thm} implies $ L^{p'} (\Omega^c, \omega_{\nu_\eps}^{1-p'})\hookrightarrow (\TnuOma)'$.
In  particular the case where $g_\eps\in L^{p'} (\Omega^c, \omega_{\nu_\eps}^{1-p'})$ and
$g\in L^{p'}(\partial\Omega)$  the dual pairings $\langle \cdot,\cdot \rangle_\eps$ and $\langle \cdot,\cdot \rangle$  respectively boil down to
\begin{align*}
\langle g_\eps , v \rangle_\eps&= \int_{\Omega^c}g_\eps(x) v(x)\d x\quad\text{and}\quad
\langle g , v \rangle=\int_{\partial \Omega}g (x)v(x)\d \sigma(x).
\end{align*}

\begin{example}
The asymptotic weak convergence in the sense of Definition~\ref{def:weak-asymp-trace}
holds clearly for the simple case $g_\eps = 0$ and $g = 0$.
The upcoming Theorem~\ref{thm:exam-weak-trace}, reminiscent of
\cite[Lemma~11.1]{Fog25}, provides a nontrivial example given for some $\varphi \in C_b^2(\R^d)$ by
\begin{align*}
    g_\eps = \mathcal{N}_\eps \varphi
    \quad\text{and}\quad
    g = K_{d,p}\,\partial_{n,p}\varphi
      := K_{d,p}\,|\nabla \varphi|^{p-2}\,\nabla \varphi \cdot n,
\end{align*}
where for $1<p<2$, we assume $\kappa_\varphi :=
    \sup_{\eps>0} \|L_\eps \varphi\|_{L^\infty(\Omega)} < \infty.$ According to  Theorem \ref{thm:point-evalua-pLevy-rad} the later is merely implied when $|\nabla \varphi(x))|\neq 0$ for all $x\in \overline{\Omega}$.
\end{example}

\noindent The following theorem provides pointwise estimates for the $p$-L\'evy operator when the corresponding kernel is radial, thereby also improving \cite[Proposition~B.2]{Fog25}.
\begin{theorem}\label{thm:point-evalua-pLevy-rad}
Let $\nu \in L^1(\R^d,\, 1 \wedge |h|^p)$ be a non-negative $p$-L\'evy radial kernel, and whose corresponding $p$-L\'evy operator is given by
\begin{align*}
Lu(x)= \pv 2\int_{\R^d}|u(x)-u(y)|^{p-2}(u(x)-u(y))\, \nu(x-y)\d y.
\end{align*}
Let $u \in C^{2}(B_{2}(x)) \cap L^{\infty}(\R^d)$ with $x \in \R^d$. Assume that $|\nabla u(x)| \neq 0$ only when $1 < p < 2$.
Then $Lu(x)$ is well defined, and for  $0<\delta<1$ small
\begin{align*}
|Lu(x)|&\leq 2^p\|u\|^{p-1}_{L^\infty(\R^d)}\int_{|h|\geq \delta}\nu(h)\d h\\
+&2^{p+1}p\Big(\|u\|_{C^2(B_\delta(x))} \frac{K_{d,p-2}}{|\nabla u(x)|^{2-p}}+
\|u\|^{\max(1, p-1)}_{C^2(B_\delta(x))}\Big)\int_{|h|<\delta}|h|^p\, \nu(h)\d h
\end{align*}
with  $K_{d,p-2}= \tfrac{d+p-2}{p-1} K_{d,p}. $ In particular, for some constant $c_{d,p}>0$ we have
\begin{align*}
\begin{split}
|Lu(x)|\leq c_{d,p}\Big(\|u\|_{C^2(B_1(x))} |\nabla u(x)|^{p-2}+
\|u\|^{\max(1, p-1)}_{C^2(B_1(x))}+\|u\|^{p-1}_{L^\infty(\R^d)}\Big)
\|\nu\|_{L^1(\R^d,1\land|h|^p)}.
\end{split}
\end{align*}
Furthermore,  if $u \in C^{2}_b(U) \cap L^{\infty}(\mathbb{R}^d)$ where $U\subset  \R^d$ is open, then for $p\geq 2$ we have
\begin{align*}
\|Lu\|_{L^\infty(U)}\leq c_{d,p}\Big(\|u\|^{p-1}_{C^2_b(U)} + \|u\|^{p-1}_{L^\infty(\R^d)}\Big)\|\nu\|_{L^1(\R^d,1\land|h|^p)}.
\end{align*}
 Meanwhile, for the case $1<p<2$,  if in addition $\frac{1}{|\nabla u|}\in L^\infty(U)$  then  we have
\begin{align*}
\|Lu\|_{L^\infty(U)}\leq c_{d,p} \Big( \|\, |\nabla u|^{-1}\|^{2-p}_{L^\infty(U)}\|u\|_{C^2_b(U)}  +\|u\|_{C^2_b(U)}  + \|u\|^{p-1}_{L^\infty(\R^d)}\Big)\|\nu\|_{L^1(\R^d,1\land|h|^p)}.
\end{align*}
\end{theorem}

\begin{proof}
We prove only the first estimate, as the remaining ones follow directly from it.   Let us put $a=\nabla u(x)\cdot h$ and $b=\int_0^1\nabla u(x+th)\cdot h\, \d t$ so that $u(x+h)-u(x)= \int_0^1\nabla u(x+th)\cdot h\d t$. Since  $\pv \int_{B_\delta(0)} \psi(\nabla u(x)\cdot h)\nu(h)\d h=0$ and $\psi(t)= |t|^{p-2}t$, by the fundamental theorem of calculus we obtain
\begin{align}\label{eq:xxpsi-taylor-nearzero}
\begin{split}
2\int_{|h|<\delta} &\psi(u(x+h)-u(x))\,\nu(h)\,\d h
\\&= 2\int_{|h|<\delta } \psi'(a)(b-a)+ [\psi(b)-\psi(a)-\psi'(a)(b-a)]\,\nu(h)\,\d h
\end{split}
\end{align}
where we emphasize that $\psi'(t)= (p-1)|t|^{p-2}$. It is worth noticing that
\begin{align*}
2|b-a| &=2\Big| \int_0^1 t\int_0^1 D^2u(x+sth)\cdot h \cdot h\, \d s\,  \d t\Big|\leq \|u\|_{C^2(B_1(x))} |h|^2.
\end{align*}
While, since   $\nabla u(x)\neq 0$ for $1<p<2$, using polar coordinates we find that
\begin{align*}
 2\Big|\int_{|h|<\delta } \psi'(a)&(b-a)\,\nu(h)\,\d h\Big|
\leq (p-1)\|u\|_{C^2(B_1(x))}\int_{|h|<\delta }  |\nabla u(x)\cdot h|^{p-2}  |h|^2\nu(h)\,\d h\\
&= (p-1)\|u\|_{C^2(B_1(x))}\int_{\mathbb{S}^{d-1}}  \hspace*{-2ex}|\nabla u(x)\cdot \xi|^{p-2}  \d\sigma_{d-1}(\xi)\, \int_0^\delta r^{p+d-1}\nu(r)\,\d r\\
&= (p-1)\|u\|_{C^2(B_1(x))} |\nabla u(x)|^{p-2}K_{d,p-2} \int_{|h|<\delta } |h|^p\nu(h)\,\d h,
\end{align*}
where we recall that
\begin{align*}
K_{d,p-2}= \frac{1}{|\mathbb{S}^{d-1}|} \int_{\mathbb{S}^{d-1}} |w_d|^{p-2}\d \sigma_{d-1}(w)= \frac{ \Gamma\big(\frac{d}{2}\big)  \Gamma\big(\frac{p-1}{2}\big)}{\Gamma\big(\frac{d+p-2}{2}\big) \Gamma\big(\frac{1}{2}\big)}= \frac{d+p-2}{p-1} K_{d,p}.
\end{align*}
Next, for  $1<p<2$ with $\nabla u(x)\neq 0$, note that  $\psi(t)= |t|^{p-2}t $ is $C^1$ on $\R\setminus\{0\}$  so that
\begin{align*}
\lim_{b\to a}\frac{\psi(b)-\psi(a) -\psi'(a)(b-a)}{b-a}=0.
\end{align*}
Since $b- a\to0$ as $|h|\to 0$, for $0<\delta<1$ small enough we have
\begin{align*}
|\psi(b)-\psi(a) -\psi'(a)(b-a)| \leq |b-a|\leq \|u\|_{C^2(B_1(x))}|h|^2\leq \|u\|_{C^2(B_1(x))}|h|^p.
\end{align*}
For the case $p\geq 2$ we have
\begin{align*}
|\psi(b)-\psi(a)-\psi'(a)(b-a)|
&\leq |b-a| \int_0^1  |\psi'(a+\tau (b-a))- \psi'(a)|\\
&\leq 2^{p+1}(p-1)\|u\|_{C^2(B_1(x))}
\|\nabla u\|^{p-2}_{L^\infty(B_1(x))}
|h|^p\\
&\leq 2^{p+1}(p-1)\|u\|^{p-1}_{C^2(B_1(x))} |h|^p.
\end{align*}
Thus in both cases $1<p<2$  and $p\geq 2$ we have
\begin{align*}
\footnotesize
2\Big|\int_{|h|<\delta }\hspace*{-2ex}|\psi(b)-\psi(a) -\psi'(a)(b-a)| \nu(h)
\d h\Big| \leq 2^{p+1}(p-1) \|u\|^{\max(1, p-1)}_{C^2(B_1(x))}\int_{|h|<\delta } \hspace*{-2ex}|h|^p\nu(h)\d h.
\end{align*}
Inserting altogether in \eqref{eq:xxpsi-taylor-nearzero} for $0<\delta<1$ sufficiently small we obtain
\begin{align*}
&2\Big|\int_{|h|<\delta } \psi(u(x+h)-u(x))\,\nu(h)\,\d h\Big|
\\
&\leq 2^{p+1}(p-1)
\Big( \|u\|_{C^2(B_1(x))} \frac{K_{d,p-2}}{|\nabla u(x)|^{2-p}} +
\|u\|^{\max(1, p-1)}_{C^2(B_1(x))}\Big)\int_{|h|<\delta }\hspace*{-2ex} |h|^p\nu(h)\d h.
\end{align*}
The desired estimate follows since for  $0<\delta<1$ we have
\begin{align*}
\int_{B^c_\delta(x)} |u(x)- u(y)|^{p-1}\nu(x-y)\d y
&= 2^{p}\|u\|^{p-1}_{L^\infty(\R^d)}\int_{|h|\geq \delta}  \nu(h) \d h.
\end{align*}
\end{proof}

\begin{theorem}\label{thm:exam-weak-trace}
Let $\Omega\subset \R^d$ be bounded Lipschitz and $\varphi\in C_b^2(\R^d)$.  In the case $1<p<2$, we assume that $\varphi$ satisfies $\kappa_\varphi:= \sup_{\eps>0}\|L_\eps \varphi\|_{L^\infty(\Omega)} <\infty.$ Then we  obtain the linear forms  $g_\eps\in  \big(\TnuOma\big)'$ and $g\in \big(W^{1-1/p ,p}(\partial \Omega)\big)'$ by defining
\begin{align*}
g_\eps =  \mathcal{N}_\eps \varphi
\qquad\text{and}\qquad
g= K_{d,p} \partial_{n,p}\varphi=  K_{d,p}  |\nabla \varphi|^{p-2} \nabla \varphi\cdot n.
\end{align*}
Moreover, $(g_\eps)_\eps$  asymptotically weakly converges to $g$.
\end{theorem}

\begin{remark}
As shown in Theorem~\ref{thm:point-evalua-pLevy-rad}, the condition
$\kappa_\varphi < \infty$ merely follows if we only assume that
$\nabla \varphi(x) \neq 0$ for all $x \in \overline{\Omega}$.
Indeed,  $|\nabla \varphi|^{-1} \in L^\infty(\Omega)$ since for some $x_0 \in \overline{\Omega}$ we have
\begin{align*}
    \big\| |\nabla \varphi|^{-1} \big\|_{L^\infty(\Omega)}
    \leq \Big( \min_{x \in \overline{\Omega}} |\nabla \varphi(x)| \Big)^{-1}
    = |\nabla \varphi(x_0)|^{-1}.
\end{align*}
\end{remark}

\begin{proof}
It is important to observe that  in both cases, i.e., $1<p<\infty$ we have $\kappa_\varphi=\sup_{\eps>0}\|L_\eps \varphi\|_{L^\infty(\Omega)} <\infty$; see Theorem \ref{thm:point-evalua-pLevy-rad}. Therefore, $\|\Delta_p  \varphi\|_{L^\infty(\Omega)} \leq \kappa_\varphi$ because  $L_\eps \varphi (x)\to -\Delta_p\varphi(x)$, by \cite[Theorem 9.9]{Fog25}.
Since  $\kappa_\varphi<\infty$ the local Gauss-Green formula and H{\"o}lder's inequality imply
\begin{align*}
\Big|\int_{\partial \Omega}\partial_{n,p} \varphi (x)v(x)\d \sigma(x)\Big|&= \Big|\int_{\Omega} \Delta_p \varphi(x) v(x)+|\nabla \varphi(x)|^{p-2} \nabla \varphi(x) \cdot \nabla v(x)\d x\Big|
\\& \leq C_\varphi \|v\|_{W^{1,p}(\Omega)}\quad \text{with}\quad C_\varphi= \kappa_\varphi|\Omega|^{1/p'}+ \|\nabla\varphi\|^{p-1}_{L^p(\Omega)}.
\end{align*}
From this we deduce $|\langle g, v \rangle|\leq C_\varphi\|v\|_{W^{1-1/p ,p}(\partial \Omega)}$ for all $v\in W^{1-1/p ,p}(\partial \Omega)$. In other words $g\in (W^{1-1/p ,p}(\partial \Omega))'$.  Analogously,   according to \cite[Lemma 11.1]{Fog25} there is a constant $C_\varphi>0$ independent of $\eps$ such that
\begin{align*}
\Big|\int_{\Omega^c} \mathcal{N}_\eps \varphi(y)v(y)\d y\Big|\leq C_\varphi\|v\|_{W^p_{\nu_\eps}(\Omega|\R^d)} \quad\text{for all $v\in \WnuOmRa$}.
\end{align*}
Accordingly, $|\langle g_\eps, v \rangle|\leq C_\varphi\|v\|_{\TnuOma}$ for all $v\in \TnuOma$. Thus,  $g_\eps\in (\TnuOma)'$ and  the sequence $(g_\eps)_\eps$ is asymptotically bounded with  $\|g_\eps\|_{(\TnuOma)'}\leq C_\varphi$. To prove the condition $(G_2)$ from Definition \ref{def:weak-asymp-trace} is valid, consider $(v_\eps)_\eps$ be asymptotically bounded  and $v_\eps\to v$ in $L^p(\Omega)$. We know that $L_\eps \varphi (x)\to -\Delta_p\varphi(x)$ and $|L_\eps\varphi | \leq \kappa_\varphi$ while Theorem \ref{thm:asymp-conv-form} implies $\cE^{\eps}(\varphi, v_\eps)\to K_{d,p}\cE^0(\varphi, v)$. This together with the (non)local Gauss-Green formula, see \cite[Theorem B.8]{Fog25} imply

\begin{align*}
\lim_{\eps\to0}\int_{\Omega^c} \mathcal{N}_\eps \varphi(y)v_\eps(y)\d y
&=
\lim_{\eps\to0} \mathcal{E}^\eps(\varphi, v_\eps )
-\lim_{\eps\to0} \int_{\Omega}L_{\eps}\varphi(x)v_\eps(x)\d x\\
&= K_{d,p} \int_{\Omega}|\nabla\varphi(x)|^{p-2}\nabla\varphi(x)\cdot \nabla v(x)-\Delta_p\varphi(x) v(x)\d x\\
&= K_{d,p}\int_{\partial\Omega} |\nabla\varphi(x)|^{p-2}\nabla\varphi(x)\cdot n(x) v(x)\,\d\sigma(x).
\end{align*}

\noindent In short, we get $\langle g_\eps, v_\eps \rangle_\eps\xrightarrow{\eps\to0} \langle g,v\rangle$ and hence the  property $(G_2)$ is verified.
\end{proof}

\begin{remark}
In view of Theorem \ref{thm:exam-weak-trace}, one is naturally tempted to  expect that  for a sufficiently smooth function $ u $ there holds the pointwise convergence $\cN_\eps \varphi (x)\to \partial_{n,p}  \varphi (x)$ for $x\in \partial \Omega$. This is however counterintuitive. In fact, using the argument used in the proof of \cite[Theorem 9.9]{Fog25}, with $\psi(t)=|t|^{p-2}t$  and $u\in C^2(B_1(x))\cap L^\infty(\R^d)$, by one finds that
\begin{align*}
\lim_{\eps\to0} \cN_\eps u (x)
&=-(p-1)\lim_{\eps\to0} \int_{(\Omega-x)\cap B_\delta(0)}
\hspace*{-2ex}|\nabla u(x)\cdot h|^{p-2}D^2 u(x)h\cdot h  \,\nu_\eps(h)\d h\\
&\qquad+ \lim_{\eps\to0} \int_{(\Omega-x)\cap B_\delta(0)} \psi(\nabla u(x)\cdot h)\, \nu_\eps( h)\d h=: H_p(x)+G_p(x).
\end{align*}
Here  the outcomes $H_p(x)$ and $G_p(x)$  are independent of  $0<\delta<1/2$ since
\begin{align*}
\int_{\Omega\cap B^c_\delta(x)} |u(x)- u(z)|^{p-1}\nu_\eps(x-z)\d z
&= 2^{p}\|u\|^{p-1}_{L^\infty(\R^d)}\int_{|h|\geq \delta}\hspace*{-2ex} \nu_\eps(h) \d h\xrightarrow{\eps\to0}0.
\end{align*}
The limit of the hessian term $H_p(x)$ always converges up to a subsequence,  provided $|\nabla u(x)|\neq 0$ for $1<p<2$ since the corresponding integral is  uniformly bounded.  Indeed, passing through polar coordinates yields
\begin{align*}
\int_{\Omega_x\cap B_\delta(0)}
&|\nabla u(x)\cdot h|^{p-2}|D^2 u(x)h\cdot h|  \nu_\eps(h)
\d h
\\&\leq  \|u\|_{C^2(B_{1/2}(x))} \frac{ |\nabla u(x)|^{p-2}}{|\mathbb{S}^{d-1}|}\int_{\mathbb{S}^{d-1}} |w_d|^{p-2} \d \sigma_{d-1}(w) \int_{B_\delta(0)} \hspace*{-2ex} |h|^p\nu_\eps(h)\d h\\
&\leq
|\nabla u(x)|^{p-2}\|u\|_{C^2(B_{1/2}(x))}K_{d,p-2}<\infty
\end{align*}
where we recall that $K_{d,p-2}=\tfrac{d+p-2}{p-1} K_{d,p}$ is positive and finite. Accordingly, $\lim\limits_{\eps\to0} \cN_\eps u (x)$ exists only if the gradient term $G_p(x)$ is convergent. However, $G_p(x)$ is likely divergent to $\pm\infty$ when $\nabla u(x)\neq 0$.
For instance, on the upper half plane $\Omega= \R^d_{+}=\{ x\in\R^d \,:\, x_d>0\}$,  assume $\partial_{x_d} u(a) \neq 0$, $\partial_{x_i} u(a)=0$, $i=1,2,\cdots, d-1$ for fixed $a\in\partial \Omega= \R^{d-1}$. Since $\Omega-a=\Omega$ for  $a\in \partial \Omega$, using polar coordinates, we have
\begin{align*}
G_p(a)&=  \lim_{\eps\to0} \int_{\R^d_+\cap B_\delta(0)} \psi(\nabla u(a)\cdot h)\nu_\eps( h)\d h\\
&= \psi(\partial_{x_d} u(a))  \lim_{\eps\to0} \int_{\R^d_+\cap B_\delta(0)} h_d^{p-1}\nu_\eps( h)\d h\\
&= \frac{K_{d,p-1}}{2}\psi(\partial_{x_d} u(a))  \lim_{\eps\to0} \int_{B_\delta(0)} |h|^{p-1}\nu_\eps( h)\d h=\pm\infty,
\end{align*}
where the fact that $\lim\limits_{\eps\to0} \int_{B_\delta(0)} |h|^{p-1}\nu_\eps( h)\d h=\infty$ can be found in  \cite[Proposition 9.1]{Fog25} or \cite[Remark 2.3]{Fog23}.   In conclusion we have the following:
\begin{align*}
\footnotesize
\lim_{\eps\to0} \cN_\eps u (x)=
\begin{cases}
-K_{d,p}\Delta_p u(x)  &x\in \Omega,\quad   p\geq 2,\\
-K_{d,p}\Delta_p u(x)  &x\in \Omega,\quad \nabla u(x)\neq 0, \quad1<p<2,\\
0 & x\in \R^d\setminus \overline{\Omega},\qquad 1<p<\infty,\\
0&x\in \partial \Omega,\quad \nabla u(x)=0, \quad p>2,\\
H_2(x) &x\in \partial \Omega,\quad \nabla u(x)=0, \quad p=2,\\
\pm \infty &x\in \partial \Omega,\quad \nabla u(x)=0, \quad 1<p<2,\\
\text{likely diverges to $\pm\infty$} & x\in \partial \Omega,\quad \nabla u(x)\neq 0,\quad 1<p<\infty.
\end{cases}
\end{align*}
The cases $x\in \Omega$ and $x\in \R^d\setminus\overline{\Omega}$ can be found in \cite[Section 9]{Fog25}.
An intriguing open question is to determine an explicit expression for $H_2(x)$ or for the general case $H_p(x)$, for  $x\in \partial\Omega$ assuming sufficient regularity  of
$\partial\Omega$, where we emphasize that
\begin{align*}
H_p(x)=-(p-1)\lim_{\eps\to0} \int_{(\Omega-x)\cap B_\delta(0)}
|\nabla u(x)\cdot h|^{p-2}D^2 u(x)h\cdot h  \,\nu_\eps(h)\d h.
\end{align*}
\end{remark}

\begin{remark}[Robustness of Gauss-Green formula]
 For $\varphi \in C^2_b(\R^d)$  and $v\in W^{1,p}(\R^d)$ and letting $\eps\to 0$ into the nonlocal Gauss-Green formula,see \cite[Theorem B.8]{Fog25},
\begin{align*}\tag{$G_\eps$}
\cE^\eps(\varphi, v)= \int_\Omega L_\eps \varphi(x)v(x)\d x+  \int_{\Omega^c}\mathcal{N}_\eps \varphi(y)v(y)\d y.
\end{align*}
wherein,  we get $\cE^\eps(\varphi, v)\to K_{d,p}\cE^0(\varphi, v)$ from Section \ref{sec:conv-forms},  $L_\eps \varphi (x)\to -\Delta_p\varphi(x)$ by \cite[Theorem 9.9]{Fog25}, and  $ \int_{\Omega^c}\mathcal{N}_\eps \varphi(y)v(y)\d y\to \int_{\partial \Omega}\partial_{n,p} \varphi(x)v(x)\d \sigma(x)$ from the foregoing, one recovers the usual local Gauss-Green formula
\begin{align*}\tag{$G_0$}
\cE^0 (\varphi, v)= \int_\Omega -\Delta_p \varphi(x)v(x)\d x+  \int_{\partial \Omega}\partial_{n,p} \varphi(x)v(x)\d \sigma(x).
\end{align*}
In other words, the nonlocal Gauss-Green formula  stays  robust as $\eps\to0$ and converges  to the local Gauss-Green formula. Along these lines, it is worth mentioning that a new proof of the classical divergence theorem using the nonlocal divergence theorem as the main tool is established in \cite{HK23}.
\end{remark}

%%%%%%%%%%%%%%%%%%%%%%%%%%%%%%%%
\section{Convergence of weak solutions}\label{sec:conv-weak-solution}
%%%%%%%%%%%%%%%%%%%%%%%%%%%%%%%%

In this section we establish the optimal convergence  of weak solutions $(u_\eps)_\eps$ of \eqref{eq:main-problem-nonlocal} to the local solution $u$ of \eqref{eq:main-problem-local}.  We improve and generalize the convergence results previously obtained in \cite{Fog25} wherein only the strong convergence in $L^p(\Omega)$ is obtained under more relaxed assumptions on the sequence $(f_\eps)_\eps$ and $(g_\eps)_\eps$. throughout this section condition~\eqref{eq:plevy-approx} is in force; while in the case $d=1$ we assume in addition that
\begin{align}\label{eq:ponce-one-dim-cond}
\text{There exist $c_0, \theta_0 \in (0,1)$ such that}
\quad
\inf_{0 < \varepsilon < 1} \inf_{\theta_0 \leq \theta \leq 1} \nu_\varepsilon(\theta x) \geq c_0
\quad \text{for all } x \in \R.
\end{align}
This condition allows  to apply the asymptotic compactness from Theorem~\ref{thm:asymp-compactness}.
\vspace{-1mm}
%%%%%%%%%%%%%%%%%%%%%%%%%%%%%%%%
\subsection{Convergence of Neumann problem}
%%%%%%%%%%%%%%%%%%%%%%%%%%%%%%%%
In this section we obtain  the optimal convergence of $(u_\eps)_\eps$ to $u$, in presence of Neumann condition, i.e., in the case where $\tau=1$ in \eqref{eq:main-problem-nonlocal} and \eqref{eq:main-problem-local}.

\begin{theorem}[\textbf{Neumann problem I}]\label{thm:convergence-neumann} Assume $\Omega\subset \R^d$ is open bounded Lipschitz and connected. Let $ (f_\eps)_\eps$ with  $f_\eps\in  (W^{p}_{\nu_\eps}(\Omega|\R^d))'$ and $ (g_\eps)_\eps$ with $g_\eps\in  (T^{p}_{\nu_\eps}(\Omega^c))'$  be sequences that converge asymptotically weakly to  $f\in (W^{1,p}(\Omega))'$ and
$g\in( W^{1-1/p ,p}(\partial\Omega))'$ respectively.
Assume  $w_\eps \in  \WnuOmRa$ is a weak solution to the Neumann problem
\begin{align*}
\mu L_\eps u = f_\eps \quad \text{ in } \,\,\Omega
\quad \text{and} \quad
\mu \mathcal{N}_\eps u = g_\eps
\quad \text{ on } \,\,\Omega^c ,
\end{align*}
where $\mu^{-1}=K_{d,p}$. Assume $w \in W^{1,p}(\Omega)$ is a  weak solution of local Neumann problem
\begin{align*}
-\Delta_p u = f \quad \text{ in } \,\, \Omega \quad\text{ and }\quad
\partial_{n,p} u = g \quad \text{ on }\,\, \partial\Omega.
\end{align*}
Then  $(u_\eps)_\eps$ with $u_\eps :=w_\eps- \frac{1}{|\Omega|}
\int_\Omega w_\eps$ optimally  converges  to $u:=w- \frac{1}{|\Omega|}\int_\Omega w$, i.e.,
\begin{align}
\lim_{\eps\to0}\|u_\eps -u\|_{W^p_{\nu_\eps}(\Omega)} = \lim_{\eps\to0}\|u_\eps- \overline{u}\|_{W^p_{\nu_\eps}(\Omega|\R^d)}=\lim_{\eps\to0}\cE^\eps_{cr}(u_\eps, u_\eps)=0,
\end{align}
where $\overline{u}\in W^{1,p}(\R^d)$ is any $W^{1,p}$-extension of $u$.
\end{theorem}

\begin{proof}
In virtue of the robust Poincar\'e inequality \cite[Corollary 10.2]{Fog25}, there exist $\eps_0\in (0,1)$ and $C= C(d,p,\Omega)>0$ such that for all $\eps\in (0, \eps_0)$ and all  $v \in\WnuOmRa$,
\begin{align}\label{eq:uniform-coercivity}
\|v-\mbox{$\fint_{\Omega} v$}\|^p_{\WnuOmRa}\leq C\cE^\eps(v,v).
\end{align}
\noindent By assumptions  $(f_\eps)_\eps$ and $(g_\eps)_\eps$ are asymptotically bounded  say
\begin{align*}
M:=\sup_{\eps\in (0,\eps_0)}\big(\|f_\eps\|_{(W^p_{\nu_\eps}(\Omega|\R^d))'}+ \|g_\eps\|_{(T^p_{\nu_\eps}(\Omega^c))'}\big)<\infty.
\end{align*}
Since $\|u_\eps\|_{\TnuOma}\leq \|u_\eps\|_{\WnuOmRa}$ we find that
\begin{align*}
\big| \langle  f_\eps ,u_\eps \rangle_\eps + \langle   g_\eps, u_\eps \rangle_\eps  \big|
\leq M\big(\|u_\eps\|_{\WnuOmRa}+ \|u_\eps\|_{\TnuOma}\big)\leq 2M\|u_\eps\|_{\WnuOmRa}.
\end{align*}
Note that $ u_\eps\in \WnuOmRa^\perp :=  \big\{ u\in \WnuOmRa: \int_{\Omega}u\d x=0\big\}$ satisfies
the weak formulation $\mu \cE^\eps(u_\eps, v) = \langle  f_\eps ,v \rangle_\eps + \langle   g_\eps, v \rangle_\eps$ for all $v\in \WnuOmRa^\perp$. This implies
\begin{align*}
\mathcal{E}^\eps(u_\eps,u_\eps)
&= \mu^{-1}\langle  f_\eps ,u_\eps \rangle_\eps + \mu ^{-1} \langle   g_\eps, u_\eps  \rangle_\eps\leq  C \|u_\eps\|_{\WnuOmRa}.
\end{align*}
Combining this with the estimate   \eqref{eq:uniform-coercivity} implies
\begin{align*}
\|u_\eps\|^{p-1}_{\WnuOma}\leq \|u_\eps\|^{p-1}_{\WnuOmRa}\leq C\qquad\text{for all } \eps\in (0,\eps_0)\,,
\end{align*}
  here   $C>0$ is generic and independent of $\eps$. We obtain the uniform boundedness
\begin{align}\label{eq:uniform-boundednessm}
\sup_{\eps\in (0,\eps_0)}\|u_\eps\|_{\WnuOma}\leq \sup_{\eps\in (0,\eps_0)} \|u_\eps\|_{\WnuOmRa}\leq C.
\end{align}
Since $\int_\Omega u_\eps\d x=0$,  by Theorem \ref{thm:asymp-compactness}, there exist   a subsequence $(\eps_n)_n,$ $\eps_n\to 0$ and  $u\in W^{1,p}(\Omega)^\perp:= \big\{ v\in W^{1,p}(\Omega): \int_{\Omega}v\d x=0\big\}$ such that
\begin{align}\label{eq:BBM-liminf}
\|u_{\eps_n} -u\|_{L^p(\Omega)}\xrightarrow{\eps_n\to0}0\quad \text{and}\quad
K_{d,p}\cE^0(u,u)\leq \liminf_{\eps\to0}\cE^\eps(u_\eps, u_\eps).
\end{align}
 Now, we must show that $u$
is in fact the unique solution  to minimization problem
\begin{align*}
\cJ_1(u)&= \min_{v\in W^{1,p}(\Omega)^\perp} \cJ_1(v)\quad \text{ with }\quad
\cJ_1(v) = \tfrac{1}{p} \cE^0(v,v)  -\langle f, u\rangle- \langle g, u\rangle.
\end{align*}
Recall that, each $u_\eps$ is solution of the minimization problem
\begin{align*}
\cJ^\eps_1(u_\eps)&= \min_{v\in \WnuOmRa^\perp } \cJ^\eps_1(v) \,\,\, \, \text{ with }\,\,\,\,
\mathcal{J}^\eps_1(v) = \tfrac{\mu}{p} \mathcal{E}^\eps(v,v) - \langle f_\eps, v\rangle_\eps - \langle g_\eps, v\rangle_\eps.
\end{align*}
 For fixed  $v\in W^{1,p}(\Omega)^\perp$, given that  $\partial \Omega$ is Lipschitz, let us consider $\overline{v}\in W^{1,p}(\R^d)$ be a Sobolev  extension of $v$.
 In view of asymptotic weak convergence of $f_\eps$ and $g_\eps$, and  Theorem \ref{thm:BBM-dual-limit} we have
\begin{align}\label{eq:limsup-J}
\begin{split}
\lim_{\eps\to0}\mathcal{J}^\eps_1(\overline{v})
&= \lim_{\eps\to0}\big(\tfrac{\mu}{p} \mathcal{E}^\eps(\overline{v},\overline{v}) -\langle f_\eps, \overline{v} \rangle_\eps - \langle g_\eps, \overline{v}\rangle_\eps\big) \\
&=\tfrac{1}{p}\mathcal{E}^0(v,v)-  \langle f, v\rangle- \langle g, v\rangle= \cJ_1(v).
\end{split}
\end{align}
Since  $u_{\eps_n}\to u$ in $L^p(\Omega)$, the asymptotic weak convergence of $(f_\eps)_\eps$ and $(g_\eps)_\eps$ imply
\begin{align*}
\lim_{n\to\infty}\langle f_{\eps_n}, u_{\eps_n}\rangle_{\eps_n}
+  \langle g_{\eps_n}, u_{\eps_n}\rangle_{\eps_n}= \langle f, u\rangle+ \langle g, u\rangle.
\end{align*}
This together with the liminf estimate in \eqref{eq:BBM-liminf}, gives $\cJ_1(u)
\leq\liminf\limits_{n \to \infty} \cJ^{\eps_n}_1(u_{\eps_n}).$ Since,  each $u_{\eps_n}$ minimizes $\cJ^{\eps_n}_1$, $\int_\Omega \overline{v}(x)\d x=0$ and   $\overline{v}\in W^{1,p}(\R^d)\subset \WnuOmRa$, we get
\begin{align*}
\cJ_1(u)\leq \liminf_{n\to \infty}\cJ^{\eps_n}_1(u_{\eps_n})\leq
\liminf_{n\to \infty} \cJ^{\eps_n}_1(\overline{v})= \cJ_1(v).
\end{align*}
It turns out that $\|u_{\eps_n}-u\|_{L^p(\Omega)}\to 0$ as $n\to\infty$ and  $u$ minimizes $\cJ_1$ as expected. Note that $u$ is therefore the unique weak solution in $W^{1,p}(\Omega)^\perp$  to the local Neumann problem $-\Delta_p u=f$ in $\Omega$ and $\partial_{n,p}u=g$ on $\partial\Omega$. Moreover, the uniqueness of $u$ implies  the convergence of the whole sequence, that is, $\|u_{\eps}-u\|_{L^p(\Omega)}\to 0$ as $\eps\to0$.   In view of Theorem \ref{thm:asymp-conv-form}, the latter implies
\begin{align*}
\cE^\eps(\overline{u},u_\eps-\overline{u}) \xrightarrow{\eps\to0}0\quad\text{and}\quad
\cE^\eps(\overline{u},u_\eps)\xrightarrow{\eps\to0}  K_{d,p}\,\cE^0(u,u).
\end{align*}
Recall that $\overline{u}\in W^{1,p}(\R^d)$ is a Sobolev extension of $u$. Furthermore,  the asymptotic weak convergence of  $ (f_\eps)_\eps$ and $ (g_\eps)_\eps$ yields
\begin{align*}
& \langle  f_\eps ,u_\eps \rangle_\eps + \langle   g_\eps, u_\eps \rangle_\eps\to  \langle  f ,u \rangle+ \langle   g, u  \rangle,\\
& \langle  f_\eps ,\overline{u} \rangle_\eps + \langle   g_\eps, \overline{u}  \rangle_\eps\,\, \to  \langle  f ,u \rangle+ \langle   g, u \rangle.
\end{align*}
Note that $\overline{u}\in W^{1,p}(\R^d)\subset \WnuOmRa$, $\int_\Omega u(x)\d x= 0$ so that  $u_\eps-\overline{u}\in \WnuOmRa^\perp$. Since $u_\eps \in \WnuOmRa^\perp$ is a weak solution, it follows that
\begin{align*}
\mathcal{E}^\eps(u_\eps,u_\eps-\overline{u})
%&= \langle  f_\eps ,u_\eps-\overline{u} \rangle_\eps + \langle   g_\eps, u_\eps-\overline{u}  \rangle_\eps
&=\langle  f_\eps ,u_\eps \rangle_\eps + \langle   g_\eps, u_\eps  \rangle_\eps- \langle  f_\eps, \overline{u} \rangle_\eps - \langle   g_\eps, \overline{u}  \rangle_\eps
\xrightarrow{\eps\to0}0.
\end{align*}
Altogether,  we have obtained  the convergences
\begin{align*}
\mathcal{E}^\eps(\overline{u},u_\eps-\overline{u}) \xrightarrow{\eps\to0}0 \quad\text{and}\quad 	\mathcal{E}^\eps(u_\eps, u_\eps-\overline{u}) \xrightarrow{\eps\to0}0.
\end{align*}
Accordingly, Theorem \ref{thm:equiv-conv-in-form} implies the  sought energy convergence, i.e.,
\begin{align*}
\cE^\eps_{cr}(u_\eps-\overline{u}, u_\eps-\overline{u})\leq \cE^\eps(u_\eps-\overline{u}, u_\eps-\overline{u}) \xrightarrow{\eps\to0}0.
\end{align*}
Finally, we obtain  $\|u_\eps-\overline{u}\|_{\WnuOmRa} \xrightarrow{\eps\to0}0$, while Lemma  \ref{lem:collap-bdary} implies that
\begin{align*}
\cE^\eps_{cr}(u_\eps, u_\eps)^{1/p}\leq \cE^\eps_{cr}(\overline{u}, \overline{u})^{1/p}+ \cE^\eps_{cr}(u_\eps-\overline{u}, u_\eps-\overline{u})^{1/p}\xrightarrow{\eps\to0}0.
\end{align*}
\end{proof}

Let us now look at the case of the regional $p$-L{\'e}vy operators $L_{\Omega,\eps}$ , $\eps>0$ with
\begin{align*}
L_{\Omega,\eps}u(x)= 2\pv \int_{\Omega} |u(x)-u(y)|^{p-2}(u(x)-u(y))\, \nu_\eps(x-y)\d y&& (x\in \Omega).
\end{align*}
\begin{theorem}[\textbf{Neumann problem II}]\label{thm:convergence-neumann regional}
Assume $\Omega\subset \R^d$ is open bounded Lipschitz and connected. Let $ (f_\eps)_\eps$ with  $f_\eps\in  (W^{p}_{\nu_\eps}(\Omega))'$ be a sequence  converging asymptotically weakly to  $f\in (W^{1,p}(\Omega))'$. Assume $u_\eps \in  \WnuOma^\perp$ is  a weak solution to the regional Neumann problem
\begin{align*}
\mu L_{\Omega,\eps} u= f_\eps \quad \text{ in \,$\Omega$ \,\, and} \quad  \int_\Omega u(x)\d x=0 .
\end{align*}
Let $u\in W^{1,p}(\Omega)^\perp$ be the  weak solution of Neumann problem
\begin{align*}
-\Delta_p u=f\quad \text{ in \,\, $\Omega$ and }\quad  \partial_{n,p} u=0\quad \text{on \,\, $\partial\Omega$}.
\end{align*}
\noindent Then $(u_\eps)_\eps$ converges to $u$ in the optimal sense, that is,
\begin{align*}
\lim_{\eps\to0}\|u_\eps-u\|_{\WnuOma}=0.
\end{align*}
\end{theorem}

\begin{proof}
The proof is analogous to that of Theorem \ref{thm:convergence-neumann}.
\end{proof}

\begin{remark}
It is worth mentioning that,  concerning the Neumann problems from Theorem \ref{thm:convergence-neumann}, $u\in W^{1,p}(\Omega)^\perp$ and $u_\eps\in\WnuOmRa^\perp$, always exist uniquely, thanks to the Poincar\'{e} type inequalities. However general solutions $w$ (resp. $w_\eps$) exists if and only if $f$ and $g$ (resp. $f_\eps$ and $g_\eps$) are compatible that is $	\langle f, 1\rangle+\langle g, 1\rangle=0$ (resp. $\langle f_\eps, 1\rangle_\eps+\langle g_\eps,  1\rangle_\eps =0$). Analogous observation holds for the regional problem from Theorem \ref{thm:convergence-neumann regional}. We refer the reader to see \cite{Fog25} for more details.
\end{remark}

As glimpsed in the proof of Theorem \ref{thm:convergence-neumann},  we implicitly use the fact that the asymptotic of the sequence of functionals $(\cJ^\eps_1)_\eps$ exhibits a seemingly form of $\Gamma$-convergence to the functional $\cJ_1$ with respect to the strong topology on  $L^p(\Omega)$.

\begin{theorem}\label{thm:gamma-conv-J}
Assume $\Omega\subset \R^d$ is open bounded Lipschitz.  Let $ (f_\eps)_\eps$ with  $f_\eps\in  (W^{p}_{\nu_\eps}(\Omega|\R^d))'$ and $ (g_\eps)_\eps$ with $g_\eps\in  (T^{p}_{\nu_\eps}(\Omega^c))'$  be sequences that converge asymptotically weakly to  $f\in (W^{1,p}(\Omega))'$ and
$g\in( W^{1-1/p ,p}(\partial\Omega))'$ respectively. Let us extend the functionals $\cJ^\eps_1: L^p(\R^d, \omega_{\nu_\eps})\to (-\infty,\infty]$  and $\cJ_1: L^p(\Omega)\to (-\infty,\infty]$ as follows
\begin{align*}
\cJ^\eps_1 (v)
&=\begin{cases}
\tfrac{\mu}{p} \mathcal{E}^\eps(v,v)  -\langle f_\eps, v\rangle_\eps - \langle g_\eps, v\rangle_\eps & \text{if $v\in \WnuOmRa$}\\
\infty &\text{ else },
\end{cases}\\
\cJ_1(v) &=\begin{cases}
\tfrac{1}{p} \mathcal{E}^0(v,v)  -\langle f, v\rangle- \langle g, v\rangle & \text{if $v\in W^{1,p}(\Omega)$}\\
\infty &\text{ else }.
\end{cases}
\end{align*}
The following assertions are true.
\begin{itemize}
\item \textbf{Limsup}.  For every  $v\in L^p(\Omega)$ there exists a sequence $(v_\eps)_\eps$ with each $v_\eps\in \WnuOmRa$ such that  $v_\eps\to v$ in $L^p(\Omega)$  and
\begin{align*}
\limsup_{\eps\to0} \cJ^\eps_1(v_\eps)\leq \cJ_1(v).
\end{align*}
\item \textbf{Liminf}. For any sequence $(v_\eps)_\eps$  of measurable  functions $v_\eps:\R^d\to\R$ such that $v_\eps\to v$ in $L^p(\Omega)$ we have
\begin{align*}
\liminf_{\eps\to0} \cJ^\eps_1(v_\eps)\geq \cJ_1(v).
\end{align*}
\end{itemize}
\end{theorem}

\begin{proof}
Let $v\in L^p(\Omega) $. The limsup is trivial if  $v\in L^p(\Omega)\setminus W^{1,p}(\Omega)$ since $\cJ_1(v)=\infty$. In this case, it suffices to consider any sequence $(v_\eps)_\eps$  such that each $ v_\eps \in C_c^\infty(\R^d)\subset \WnuOmRa$ is supported in $\Omega$ and  that $\|v_\eps-v \|_{L^p(\Omega)}\to 0$ as $\eps\to 0$. For  $v\in W^{1,p}(\Omega)$, we pick the constant sequence $v_\eps= \overline{v}$ where $\overline{v}\in W^{1,p}(\R^d)\subset \WnuOmRa$  is any $W^{1,p}$-extension of $v$. Thus $\|v_\eps-v \|_{L^p(\Omega)}= 0$ and by \eqref{eq:limsup-J} we have
\begin{align*}
\lim_{\eps\to0}\mathcal{J}^\eps_1(v_\eps)=\lim_{\eps\to0}\mathcal{J}^\eps_1(\overline{v})= \cJ_1(v).
\end{align*}
To prove the liminf condition, consider a sequence $(v_\eps)_\eps$ of measurable function such that $v_\eps\to v$ in $L^p(\Omega)$. Assume  $\liminf\limits_{\eps\to0}\cJ^\eps_1(v_\eps)<\infty$ since the case $\liminf\limits_{\eps\to0}\cJ^\eps_1(v_\eps)= \infty$ is trivial. In particular we have   $\liminf\limits_{\eps\to0}\cE^\eps(v_\eps,v_\eps)<\infty$.  By Theorem \ref{thm:BBM-liminf} we get
\begin{align*}
\liminf_{\eps\to0}\cE^\eps(v_\eps,v_\eps)
\geq
\liminf_{\eps\to0}\cE^\eps_\Omega(v_\eps,v_\eps)\geq K_{d,p}\cE^0(v,v).
\end{align*}
By  assumption we have $\langle f_\eps, v_\eps \rangle_\eps + \langle g_\eps,v_\eps\rangle_\eps
\to  \langle f, v\rangle+ \langle g, v\rangle$. Hence one obtains
\begin{align*}
\liminf_{\eps\to0}\cJ^\eps_1(v_\eps)\geq\cJ_1(v).
\end{align*}
\end{proof}

\noindent The next result deals with regional functionals and is analog to Theorem \ref{thm:gamma-conv-J}.
\begin{theorem}\label{thm:gamma-conv-JOm}
Assume $\Omega\subset \R^d$ is open bounded Lipschitz. Let $ (f_\eps)_\eps$ with  $f_\eps\in  (W^{p}_{\nu_\eps}(\Omega))'$ be a sequence  converging asymptotically weakly to  $f\in (W^{1,p}(\Omega))'$. Let us define the functionals $\cJ^\eps_\Omega$  and $\cJ_\Omega$ as follows
\begin{align*}
\cJ^\eps_\Omega (v)
&=\begin{cases}
\tfrac{\mu}{p} \mathcal{E}^\eps_\Omega(v,v)  -\langle f_\eps, v\rangle_\eps  & \text{if $v\in \WnuOma$}\\
\infty &\text{ else },
\end{cases}\\
\cJ_\Omega(v) &=\begin{cases}
\tfrac{1}{p} \mathcal{E}^0(v,v)  -\langle f, v\rangle& \text{if $v\in W^{1,p}(\Omega)$}\\
\infty &\text{ else }.
\end{cases}
\end{align*}
Then $\cJ^\eps_\Omega\to\cJ_\Omega$ in the Gamma sense, i.e., the following assertions are true
\begin{itemize}
\item \textbf{Limsup}.  For every  $v\in L^p(\Omega)$ there exists a sequence $(v_\eps)_\eps$ with  $v_\eps\in \WnuOma$ such that  $v_\eps\to v$ in $L^p(\Omega)$  and
\begin{align*}
\limsup_{\eps\to0} \cJ^\eps_\Omega (v_\eps)\leq \cJ_\Omega(v).
\end{align*}
\item \textbf{Liminf}. For any sequence $(v_\eps)_\eps\subset L^p(\Omega)$ such that $v_\eps\to v$ in $L^p(\Omega)$ we have
\begin{align*}
\liminf_{\eps\to0} \cJ^\eps_\Omega(v_\eps)\geq \cJ_\Omega(v).
\end{align*}
\end{itemize}
\end{theorem}

%%%%%%%%%%%%%%%%%%%%%%%%%%%%%%%
\subsection{Convergence of the Dirichlet problem}
%%%%%%%%%%%%%%%%%%%%%%%%%%%%%%%
In this section we establish the optimal convergence of $(u_\eps)_\eps$ to $u$, in presence of Dirichlet condition, i.e., in  the case $\tau=0$ in \eqref{eq:main-problem-nonlocal} and \eqref{eq:main-problem-local}.  The function spaces for the Dirichlet data  $g_\eps$ and $g$ are completely different.  First of all, by Remark \ref{rem:ext-plus-trace}, using a lifting of the trace operator and the standard Sobolev extension, any  function $g\in W^{1-1/p , p}(\partial \Omega)$, has a representative $\overline{g}\in W^{1,p}(\R^d)$, such that $\gamma_0(\overline{g})= g$ and
$\|\overline{g}\|_{ W^{1, p}(\R^d)}
\leq C\|g\|_{ W^{1-1/p , p}(\partial \Omega)}$ for some constant $C>0$ independent on $g$. Therefore, there is no loss of generality  if the  Dirichlet boundary data $g$ belongs to $ W^{1,p}(\R^d)$ or to $W^{1,p}(\R^d\setminus\overline{\Omega})$. Likewise,  it is also  sufficient to consider the Dirichlet data  $g_\eps$ in $\WnuOmRa$ for the nonlocal Dirichlet problem
\begin{theorem}[\textbf{Dirichlet problem I}]\label{thm:converg-dirichlet}
Assume $\Omega\subset \R^d$ is open bounded Lipschitz.   Let $ (f_\eps)_\eps$ with  $f_\eps\in  (W^{p}_{\nu_\eps,0}(\Omega|\R^d))'$ converges asymptotically weakly to $f\in (W^{1,p}_0(\Omega))'$ and $ (g_\eps)_\eps$ with $g_\eps\in  T^{p}_{\nu_\eps}(\Omega^c)$ converges asymptotically strongly to $g\in W^{1,p}(\R^d\setminus\overline{\Omega})$, i.e.,
\begin{align*}
\lim_{\eps\to0}	\|g_\eps-g\|_{\TnuOma}=0.
\end{align*}
Let $\mu^{-1}=K_{d,p}$, and $u_\eps \in  \WnuOmRa$ be the  weak solution of Dirichlet problem
\begin{align*}
\mu L_\eps u = f_\eps \quad \text{ in } \,\, \Omega
\quad \text{ and } \quad u = g_\eps
\quad \text{ on }\,\, \Omega^c ,
\end{align*}
 Let  $u \in W^{1,p}(\Omega)$ be  the weak solution of local Dirichlet problem
\begin{align*}
-\Delta_p u = f \quad \text{ in } \,\, \Omega \quad\text{ and }\quad
u = g \quad \text{ on } \,\, \partial\Omega.
\end{align*}
\noindent Then $(u_\eps)_\eps$ optimally converges to $u$ in the following sense
\begin{align*}
\lim_{\eps\to0}\|u_{\eps}-u\|_{W^p_{\nu_\eps}(\Omega)}= \lim_{\eps\to0}\|u_{\eps}-u_g\|_{W^p_{\nu_\eps}(\Omega|\R^d)} =\lim_{\eps\to0}\|u_{\eps}-\overline{u}\|_{W^p_{\nu_\eps}(\Omega| \R^d)}=0.
\end{align*}
In particular $\cE^\eps_{cr}(u_\eps-\overline{u}, u_\eps-\overline{u})\xrightarrow{\eps\to0}0$. Here $\overline{u}\in W^{1,p}(\R^d)$ is any $W^{1,p}$-extension of $u$,  while $u_g\in W^{1,p}(\R^d)$ refers to the extension of  $u$ by $g$ on $\Omega^c$.
\end{theorem}

\begin{proof}
\textbf{Extension of $g_\eps$ and $g$}. Since  $\Omega\subset \R^d$ is bounded Lipschitz, we extend $g$ to some $\overline{g}\in W^{1,p}(\R^d)\subset\WnuOmRa$ with
$\overline{g}= g$ a.e. on $\R^d\setminus\overline{\Omega}$ and $\|\overline{g}\|_{W^{1,p}(\R^d)}\leq C\|g\|_{W^{1,p}(\R^d\setminus\overline{\Omega})}$ for some  $C>0$ independent of $g$. Thus  $g\in \TnuOma$ and we get
\begin{align*}
\|g\|^p_{\TnuOma}\leq \|\overline{g}\|^p_{\WnuOmRa} \leq 2^{p+1} \|\overline{g}\|^p_{W^{1,p}(\R^d)}\leq C\|g\|_{W^{1,p}(\R^d\setminus\overline{\Omega})}.
\end{align*}
Therefore we have $g_\eps-g\in \TnuOma$. Hence, for each $\eps>0$ there is $h_\eps\in \WnuOmRa$ such that $h_\eps= g_\eps-g$ a.e. on $\Omega^c$ and
\begin{align*}
\|h_\eps\|_{\WnuOmRa}\leq \eps+ \|g_\eps-g\|_{\TnuOma}.
\end{align*}
Let us define $ \overline{g}_\eps= h_\eps +\overline{g}\in\WnuOmRa$ so that $\overline{g}_\eps= g_\eps$ a.e. on $\Omega^c$ and
\begin{align}\label{eq:geps-extens}
\|\overline{g}_\eps-\overline{g}\|_{\WnuOmRa}= \|h_\eps\|_{\WnuOmRa}\leq \eps+ \|g_\eps-g\|_{\TnuOma}\xrightarrow{\eps\to0}0.
\end{align}
By means of the foregoing extension device, there is no loss of generality in assuming that $g \in W^{1,p}(\R^d) $ and $g_\eps \in \WnuOmRa$ such that
\begin{align*}
\|g_\eps-g \|_{\WnuOmRa}\xrightarrow{\eps\to0}0.
\end{align*}
In particular, we can assume $\|g_\eps-g \|_{\WnuOmRa}<1$ so that, \eqref{eq:taylor-exapnsion} yields
\begin{align*}
\sup_{\eps>0}\|g_\eps \|_{\WnuOmRa}\leq 1+ \sup_{\eps>0}
\|g \|_{\WnuOmRa}\leq 1+  2^{p+1} \|g\|^p_{W^{1,p}(\R^d)}<\infty.
\end{align*}
\textbf{Boundedness of $(u_\eps)_\eps$.} This leads to  the uniform estimate
\begin{align}\label{eq:uniform-bdd-f-g-D}
M:= 	\sup_{\eps>0}\big(\|f_\eps \|_{(\WnuOmRao)'}+\|g_\eps \|_{\WnuOmRa}\big)<\infty.
\end{align}
By the robust Poincar\'e-Friedrichs inequality see  \cite[Section 10]{Fog25} there exist $\eps_0 \in (0,1)$ and $C>0$ such that, for all $\eps\in (0, \eps_0)$ and $v \in\WnuOmRao$ we have
\begin{align}\label{eq:uniform-coercivity-D}
\|v\|^p_{\WnuOmRa}\leq C\mathcal{E}^\eps(v,v).
\end{align}

\noindent It worth recalling that the weak solution  $u_\eps$  verifies  $\cE^\eps(u_\eps,u_\eps-g_\eps)=  \langle f_\eps , u_\eps-g_\eps \rangle_\eps$ and $ u_\eps-g_\eps\in \WnuOmRao$. Thus we find that
\begin{align*}
\begin{split}
\| u_\eps-g_\eps \|_{\WnuOmRa}
&\leq C \cE^\eps(u_\eps-g_\eps,u_\eps-g_\eps)^{1/p}\\
&	\leq C\cE^\eps(u_\eps,u_\eps)^{1/p}+ C\cE^\eps(g_\eps,g_\eps)^{1/p}.
\end{split}
\end{align*}
Furthermore, we also find that
\begin{align*}
\begin{split}
&\cE^\eps(u_\eps,u_\eps)
= \cE^\eps(u_\eps,u_\eps-g_\eps) +\cE^\eps(u_\eps,g_\eps)= \langle f_\eps , u_\eps-g_\eps \rangle_\eps +\cE^\eps(u_\eps,g_\eps)\\
&\leq \| u_\eps-g_\eps \|_{\WnuOmRa} \|f_\eps\|_{\WnuOmRao'}+\cE^\eps(u_\eps,u_\eps)^{1/p'}\cE^\eps(g_\eps,g_\eps)^{1/p}\\
&\leq  C \|f_\eps\|_{\WnuOmRao'}\big(
\cE^\eps(u_\eps,u_\eps)^{1/p}+\cE^\eps(g_\eps,g_\eps)^{1/p}\big)
+\cE^\eps(u_\eps,u_\eps)^{1/p'}\cE^\eps(g_\eps,g_\eps)^{1/p}.
\end{split}
\end{align*}

\noindent By exploiting the Young inequality $	|ab|\leq \frac{\delta^p |a|^p}{p}+ \frac{|b|^{p'}}{p'\delta^{p'}} $, $a,b\in \R$ and $\delta>0$,  we get

\begin{align*}
\cE^\eps(u_\eps,u_\eps)^{1/p'}\cE^\eps(g_\eps,g_\eps)^{1/p}
&\leq  \frac{\delta^{p}\cE^\eps(u_\eps,u_\eps)}{p'}+ \frac{\cE^\eps(g_\eps,g_\eps)}{p\delta^{\frac{p^2}{p'}}},\\
C\|f_\eps\|_{\WnuOmRao'} \cE^\eps(u_\eps,u_\eps)^{1/p}
&\leq  \frac{\delta^p\cE^\eps(u_\eps,u_\eps)}{p}+ \frac{C^{p'}\|f_\eps\|_{\WnuOmRao'} ^{p'}}{p'\delta^{p'}},\\
C\|f_\eps\|_{\WnuOmRao'}\cE^\eps(g_\eps,g_\eps)^{1/p}
&\leq \frac{\delta^{p}\cE^\eps(g_\eps,g_\eps)}{p} + \frac{C^{p'}\|f_\eps\|_{\WnuOmRao'} ^{p'}}{p'\delta^{p'}}.
\end{align*}
Inserting altogether in the previous estimate gives
\begin{align*}
\cE^\eps(u_\eps,u_\eps)
&\leq 	 \delta^{p}\cE^\eps(u_\eps,u_\eps) + ( \frac{1}{p\delta^{\frac{p^2}{p'}}}+ \frac{\delta^{p}}{p}) \cE^\eps(g_\eps,g_\eps)
+ \frac{2C^{p'}}{p'\delta^{p'}} \|f_\eps\|_{\WnuOmRao'} ^{p'}.
\end{align*}
\noindent Assigning the  particular choice  $\delta^p=\frac12$ we arrive at
\begin{align*}
\cE^\eps(u_\eps,u_\eps)&\leq C(\|f_\eps\|_{\WnuOmRao'} ^{p'} + \cE^\eps(g_\eps,g_\eps)) .
\end{align*}
On the other hand, the coercivity estimate \eqref{eq:uniform-coercivity-D} implies
\begin{align}\label{eq:xlp-bound-dirich}
\begin{split}
\|u_\eps\|_{L^p(\Omega)}
&\leq  \|g_\eps\|_{L^p(\Omega)}+ C\cE^\eps(u_\eps-g_\eps,u_\eps-g_\eps)^{1/p}\\
&\leq C\|g_\eps\|_{\WnuOmRa }+ C\cE^\eps(u_\eps,u_\eps)^{1/p}.
\end{split}
\end{align}
By combining everything together and accounting \eqref{eq:uniform-bdd-f-g-D}  leads to the uniform estimate
\begin{align}\label{eq:unifo-bdedness-D}
\begin{split}
\|u_\eps\|_{\WnuOmRa}
&\leq C\big(\|f_\eps \|^{\frac{1}{p-1}}_{(\WnuOmRao)'}+\|g_\eps \|_{\WnuOmRa}\big)
\leq C
\end{split}
\end{align}
for a generic  constant  $C>0$ independent of $\eps$. Namely, we arrive at the conclusion that, the sequence $(w_\eps)_\eps$ with $w_\eps= u_\eps-g_\eps$ is asymptotically bounded  with respect to $W^p_{\nu_{\eps}}(\R^d)$, since $u_\eps-g_\eps=0$ a.e. on $\Omega^c$ and accounting  \eqref{eq:uniform-bdd-f-g-D},  we have
\begin{align*}
\begin{split}
\|u_\eps-g_\eps\|_{W^p_{\nu_{\eps} }(\R^d)}= \|u_\eps-g_\eps\|_{\WnuOmRa}\leq \|u_\eps\|_{\WnuOmRa}+ \|g_\eps\|_{\WnuOmRa}
\leq C.
\end{split}
\end{align*}
\textbf{$L^p$-Convergence of $(w_\eps)_\eps$ and $(u_\eps)_\eps$.} Accordingly, by the asymptotic compactness\footnote{Applying the asymptotic compactness Theorem \ref{thm:asymp-compactness} directly on $(u_\eps)_\eps$ would necessitate invoking the assumption that the domain  $\Omega$ is  Lipschitz, which we intentionally refrain from using here.}
Theorem \ref{thm:asymp-compactness}, there is $w\in W^{1,p}(\R^d)$ and subsequence $\eps_n\to 0$ such that $(u_{\eps_n} -g_{\eps_n} )_n$ converges to $w$ in $L^p_{\loc}(\R^d)$.  In particular, since $w_\eps= u_\eps-g_\eps=0$ a.e. on $\Omega^c$ we find that $w=0$ a.e. on $\Omega^c$, leading to the conclusion that $w\in W^{1,p}_0(\Omega)$ owing to  the fact that the boundary $\partial\Omega$ is continuous.  In addition,  given that $(g_{\eps})_\eps$ converges to $g$ in $L^p(\Omega)$,  then  $ u_\eps= w_\eps+ g_\eps$  also converges to $u:= w+g$ in $L^p(\Omega)$. Therefore, letting $u_g=w+ g$  in $\R^d$ we have  $u_g= w+g\in W^{1,p}(\R^d)$, $u_g|_{\Omega^c}= g$ and   $u_g|_\Omega= u$, that is, $u_g$ extends $u$ by $g$ on $\Omega^c$. By Theorem \ref{thm:BBM-liminf}, there also holds that
\begin{align*}
K_{d,p}\cE^0(u,u)&\leq \liminf_{n\to\infty}\cE^{\eps_n}(u_{\eps_n}, u_{\eps_n}).
\end{align*}
Moreover, the strong convergence of $(u_{\eps_n}-g_{\eps_n})_n$ to $u-g$ in $L^p(\Omega)$ and the asymptotic weak convergence of $(f_{\eps_n} )_n $ yield $\langle f_{\eps_n}, u_{\eps_n}-g_{\eps_n} \rangle_{\eps_n }\to \langle f ,u-g\rangle$ and  hence
\begin{align*}
\cJ_0(u)\leq \liminf_{n\to \infty}\cJ_0^{\eps_n}(u_{\eps_n}),
\end{align*}
where we recall that each $u_\eps$  uniquely  satisfies the minimization problem
\begin{align*}
\cJ_0^\eps(u_\eps) &= \min_{v-g_\eps\in \WnuOmRao} \cJ_0^\eps(v)
\quad \text{with}\quad
\cJ_0^\eps(v) = \tfrac{\mu}{p} \cE^\eps(v,v)
-\langle f_\eps , v-g_\eps\rangle_\eps.
\end{align*}
\textbf{The limit $u$ minimizes $\cJ_0$.} Next, we show that  $u$  solves the minimization problem
\begin{align*}
\cJ_0(u)&= \min_{v-g\in W^{1,p}_0(\Omega)} \cJ_0(v)  \quad\text{with}\quad
\cJ_0(v) = \tfrac{1}{p} \mathcal{E}^0(v,v)
- \langle f,v-g\rangle.
\end{align*}
Let $v\in g+W^{1,p}_0(\Omega)$, that is, $v-g\in W^{1,p}_0(\Omega)\subset \WnuOmRao$ and consider $v_\eps=v-g+g_\eps$ so that $v_\eps\in g_\eps+ \WnuOmRao$.  Further, since  $\cE^\eps(v_\eps-v,v_\eps-v)^{1/p}\leq \|g_\eps-g\|_{\WnuOmRa}, $ the triangular inequality implies
\begin{align*}
\big|\cE^\eps(v_\eps,v_\eps)^{1/p}-K^{1/p}_{d,p}\cE^0(v,v)^{1/p}\big|
&\leq  \|g_\eps-g\|_{\WnuOmRa}+ \big|\cE^\eps(v,v)^{1/p}-K^{1/p}_{d,p}\cE^0(v,v)^{1/p}\big|.
\end{align*}
Given that $\Omega$ is Lipschitz, Theorem \ref{thm:BBM-dual-limit} implies that $ \cE^\eps(v,v)\xrightarrow{\eps\to0}K_{d,p}\cE^0(v,v)$  and from the foregoing we have $\|g_\eps-g\|_{\WnuOmRa}\xrightarrow{\eps\to0}0$. Therefore, we conclude that
\begin{align}\label{eq:Eeps-veps-conv}
\lim_{\eps\to0}\cE^\eps(v_\eps,v_\eps)
= K_{d,p}
\cE^0(v,v).
\end{align}
Since $v_\eps-g_\eps= v-g$, this and  the asymptotic weak convergence of $(f_\eps)_\eps$ yield
\begin{align*}
\lim_{\eps\to0}\mathcal{J}^\eps_0(v_\eps)
&= \lim_{\eps\to0}\big(\tfrac{\mu}{p} \cE^\eps(v_\eps, v_\eps)-\langle f_\eps , v_\eps-g_\eps \rangle_\eps \big)\\
&=\tfrac{1}{p}\mathcal{E}^0(v,v) -\langle f ,v-g\rangle = \cJ_0(v).
\end{align*}
Whence, since  each $u_{\eps_n}$ minimizes $\cJ^{\eps_n}_0$, i.e.,  $\cJ_0^{\eps_n}(u_{\eps_n})\leq
\cJ_0^{\eps_n}(v_{\eps_n}) $, we deduce that
\begin{align*}
\cJ_0(u)\leq \liminf_{n\to \infty}\cJ_0^{\eps_n}(u_{\eps_n})\leq
\liminf_{n\to \infty}\cJ_0^{\eps_n}(v_{\eps_n})= \cJ_0(v).
\end{align*}
We arrive at the conclusion that
\begin{align*}
\cJ_0(u)&= \min_{v-g\in W^{1,p}_0(\Omega)} \cJ_0(v).
\end{align*}
In other words, $u$ is in fact the unique weak solution  to the local Dirichlet problem $-\Delta_p u=f$ in $\Omega$ and $u=g$ on $\partial\Omega$ on $W^{1,p}(\Omega)$. The uniqueness of $u$ implies  that $\|u_{\eps}-u\|_{L^p(\Omega)}\to 0$ as $\eps\to0$.
From this, Theorem \ref{thm:asymp-conv-form} implies the asymptotic convergences
\begin{align*}
\mathcal{E}^\eps(u_g,u_\eps-u_g) \xrightarrow{\eps\to0}0
\quad\text{and}\quad
\mathcal{E}^\eps(u_g,u_\eps)
\xrightarrow{\eps\to0}
K_{d,p}\,\mathcal{E}^0(u,u).
\end{align*}
\textbf{Optimal convergence.} Furthermore, $(u_\eps-g_\eps)_\eps$ converges to $u-g$ in $L^p(\Omega)$ and $u-g\in W^{1,p}_0(\Omega)$, since $\|g_\eps-g\|_{\WnuOmRa}\xrightarrow{\eps\to0}0$.
By the asymptotic weak convergence of  $ (f_\eps)_\eps$  we obtain
\begin{align*}
& \langle  f_\eps ,u_\eps-g_\eps \rangle_\eps \to  \langle  f ,u-g \rangle\quad \text{and}\quad
\langle  f_\eps ,u_g-g \rangle_\eps
\to  \langle  f ,u-g \rangle.
\end{align*}
Moreover, recalling that $(u_\eps)_\eps$ is asymptotically bounded in $\WnuOmRa$, we have
\begin{align*}
|\mathcal{E}^\eps(u_\eps,g_\eps-g)|\leq  \mathcal{E}^\eps(u_\eps,u_\eps)^{1/p'}
\mathcal{E}^\eps(g_\eps-g, g_\eps-g)^{1/p}\leq C\|g_\eps-g\|_{\WnuOmRa}\xrightarrow{\eps\to0}0.
\end{align*}
Thence,  since $u_\eps \in \WnuOmRa$ is a weak solution and  $g-u_g\in W^{1,p}_0(\Omega)\subset \WnuOmRa$, it follows by linearity  and by $u_\eps-g_\eps, u_g-g\in W^p_{\nu_\eps,0}(\Omega|\R^d)$ that
\begin{align*}
\mathcal{E}^\eps(u_\eps,u_\eps-u_g)
&= 	\mathcal{E}^\eps(u_\eps,g_\eps-g)
+ \langle  f_\eps ,u_\eps-g_\eps \rangle_\eps
+ \langle  f_\eps ,g-u_g \rangle_\eps
\xrightarrow{\eps\to0}0.
\end{align*}
Altogether,  we have obtained  the convergences
\begin{align*}
\mathcal{E}^\eps(u_g, u_\eps-u_g) \xrightarrow{\eps\to0}0
\quad\text{and}\quad
\mathcal{E}^\eps(u_\eps, u_\eps-u_g) \xrightarrow{\eps\to0}0.
\end{align*}
Accordingly,  Theorem \ref{thm:equiv-conv-in-form} implies $\mathcal{E}^\eps(u_\eps-u_g, u_\eps-u_g) \xrightarrow{\eps\to0}0$ and hence  we have
\begin{align*}
\|u_\eps-u_g\|_{\WnuOmRa} \xrightarrow{\eps\to0}0.
\end{align*}
Finally, let $\overline{u}\in W^{1,p}(\R^d)$ be another $W^{1,p}$-extension of $u$, then we have  $u_g-\overline{u}\in W^{1,p}(\R^d)$  and $u_g-\overline{u}=0$ a.e. on $\Omega$. Accordingly, since $\Omega$ is bounded Lipschitz, by Lemma  \ref{lem:collap-bdary} we get
$\cE^\eps_{cr}(u_g-\overline{u}, u_g-\overline{u})^{1/p}\xrightarrow{\eps\to0}0$. It follows that
\begin{align*}
\|u_\eps-\overline{u}\|_{\WnuOmRa}
&\leq \|u_\eps-u_g\|_{\WnuOmRa}+ \|u_g-\overline{u}\|_{\WnuOmRa}\\
&=\|u_\eps-u_g\|_{\WnuOmRa}+\cE^\eps_{cr}(u_g-\overline{u}, u_g-\overline{u})^{1/p}\xrightarrow{\eps\to0}0,
\end{align*}
while also Lemma  \ref{lem:collap-bdary} implies that
\begin{align*}
\cE^\eps_{cr}(u_\eps, u_\eps)^{1/p}\leq \cE^\eps_{cr}(\overline{u}, \overline{u})^{1/p}+ \cE^\eps_{cr}(u_\eps-\overline{u}, u_\eps-\overline{u})^{1/p}\xrightarrow{\eps\to0}0.
\end{align*}
\end{proof}

In the special case where $\Omega$ is only bounded in one direction,  that is,  we have $\Omega\subset H_R$ with $H_R=\{ x\in \R^d\,:\, |x\cdot e|\leq R\}$  for some $R>0$ and $e\in \R^d$, the convergence of $(u_\eps)_\eps$ to $u$ is only obtained in  $L^p_{\loc}(\Omega)$, see instance  \cite{Fog25},  thereby excluding the possibility of applying the asymptotic weak convergence of $(f_\eps)_\eps $.  Consequently,  this underscores the need of  weakening the assumption on $(f_\eps)_\eps$.
\begin{definition}\label{def:strong-asymp-form}
We say that a sequence  $(f_\eps)_\eps$ with $f_\eps\in (W^p_{\nu_\eps,0}(\Omega|\R^d))'$ asymptotically strongly converges to $f\in (W^{1,p}_0(\Omega))'$ if $(f_\eps)_\eps$ is asymptotically bounded (see Definition \ref{def:weak-asymp-form}) and satisfies the following condition
\begin{enumerate}[$(F'_2)$]
\item \textbf{$L^p$-Strong-Weak type lemma:} For  a sequence $(v_\eps)_\eps$ with $v_\eps\in\WnuOmRao$ and $v\in W^{1,p}_0(\Omega)$ such that $(v_\eps|_\Omega)_\eps$ weakly converges  to $ v|_\Omega$  in $L^p(\Omega)$  as $\eps\to 0$,
\begin{align*}
\lim_{\eps\to0} \langle f_\eps , v_\eps \rangle_\eps
= \langle f , v\rangle .
\end{align*}
\end{enumerate}
The condition $(F'_2)$ above obviously occurs when  $\|f_\eps-f\|_{L^{p'}(\Omega)}\to0$ as $\eps\to0$.
\end{definition}
\noindent Let us recall that the topology on $L^p_{\loc}(\Omega)$ is metrizable with the metric
\begin{align*}
\|u-v\|_{L^p_{\loc}(\Omega)}= \sum_{j=1}^{\infty}\frac{1}{2^j}\frac{\| u-v\|_{L^p(\Omega_j)}}{1+\| u-v\|_{L^p(\Omega_j)}},
\end{align*}
where $(\Omega_j)_j$ is any exhaustion of $\Omega$ such that each $\Omega_j$ is  open bounded.

\begin{theorem}[\textbf{Dirichlet problem II}]\label{thm:converg-dirichlet-bis}
Let the setting of Theorem \ref{thm:converg-dirichlet} applies. Assume $\Omega\subset \R^d$ is open bounded in one direction with Lipschitz boundary. Assume $(f_\eps)_\eps$ strongly asymptotically converges $f$ in the sense of Definition \ref{def:strong-asymp-form}.
\noindent Then $(u_\eps)_\eps$ locally strongly converges to $u$ in the following sense
\begin{align*}
&\lim_{\eps\to0}\|u_{\eps}-u\|_{L^p_{\loc}(\Omega)}+ \cE^\eps_\Omega(u_{\eps}-u,u_{\eps}-u)= 0,\\
&\lim_{\eps\to0}\|u_{\eps}-u_g\|_{L^p_{\loc}(\R^d)}+ \cE^\eps(u_{\eps}-u_g,u_{\eps}-u_g)=\lim_{\eps\to0}\cE^\eps_{cr}(u_\eps, u_\eps)= 0.
\end{align*}
Moreover, if in addition $g_\eps= g$ then we have
\begin{align*}
\lim_{\eps\to0}\|u_{\eps}-u\|_{W^p_{\nu_\eps}(\Omega)}
=\lim_{\eps\to0}\|u_{\eps}-u_g\|_{W^p_{\nu_\eps}
(\Omega| \R^d)}
=\lim_{\eps\to0}\|u_{\eps}-\overline{u}\|_{W^p_{\nu_\eps}(\Omega| \R^d)}=0.
\end{align*}
\end{theorem}

\begin{proof}
The proof follows by  proceeding similarly  as in the proof of Theorem \ref{thm:converg-dirichlet},  thus we only make adjustments solely at the points of discrepancy of arguments.
\smallskip

\textbf{Fact 1.} (Robust Poincar\'e inequality). The robust Poincar\'e inequality still holds when $\Omega$ is bounded in one direction.  Namely see \cite[Section 10]{Fog25} there exist $\eps_0 \in (0,1)$ and $C>0$ such that, for all $\eps\in (0, \eps_0)$ and $v \in\WnuOmRao$ we have
\begin{align}\label{eq:uniform-coercivity-D-bis}
\|v\|^p_{\WnuOmRa}\leq C\mathcal{E}^\eps(v,v).
\end{align}

\textbf{Fact 2.} (Weak convergence). From this we get that $(u_\eps)_\eps$ is asymptotically bounded with respect to $W^p_{\nu_\eps}(\R^d)$, thereby yielding that $(u_\eps)_\eps$ converges to $u_g$ in $L^p_{\loc}(\R^d)$. This in conjunction with the boundedness  in $L^p(\R^d)$ implies  $(u_\eps)_\eps$ converges weakly to $u_g$ in $L^p(\R^d)$. In particular, $((u_\eps-g_\eps)|_\Omega)_\eps$ converges weakly to $u-g$ in $L^p(\Omega)$. By assumption on $(f_\eps)_\eps$ we obtain that $\langle f_\eps, u_\eps-g_\eps\rangle_\eps\to  \langle f, u-g\rangle $, while by Theorem \ref{thm:asymp-conv-form}  implies $\cE^\eps(u_\eps, u_\eps-u_g)\to0$. The sought optimal convergence follows likewise.

\smallskip
\textbf{Fact 3.} (Case $g_\eps=g$). If $g_\eps=g$ then  $u_\eps-u_g=0$ on $\Omega^c$, i.e., $u_\eps-u_g\in \WnuOmRao$. Thence,  the estimate  \eqref{eq:uniform-coercivity-D-bis} implies
 $\|u_\eps-u_g\|^p_{\WnuOmRa}\xrightarrow{\eps\to0}0$ since
  \begin{align*}
\|u_\eps-u_g\|^p_{\WnuOmRa}\leq C\mathcal{E}^\eps(u_\eps-u_g, u_\eps-u_g)\xrightarrow{\eps\to0}0.
\end{align*}
\end{proof}

Interestingly, other variants of Theorem \ref{thm:converg-dirichlet} can be derived by altering the assumptions on the Dirichlet Data $g_\eps$ or on the domain $\Omega$. For instance, by strengthening the assumption on the Dirichlet data $(g_\eps)_\eps$ may allow for a potential relaxation of the regularity requirements on
$\Omega$.
\begin{theorem}[\textbf{Dirichlet problem III}]\label{thm:converg-dirichlet-3}
Assume $\Omega\subset \R^d$ is open,  with a continuous boundary. Assume $\|g_\eps-g\|_{W^p_{\nu_\eps}(\R^d)}\to0$ with $g_\eps\in  W^p_{\nu_\eps}(\R^d)$ and $g\in W^{1,p}(\R^d)$.  Assume that $g=0$ or  $|\partial\Omega|=0$.

\medskip
If $\Omega$ is bounded and the setting of Theorem \ref{thm:converg-dirichlet} is in force then
\begin{align*}
\lim_{\eps\to0}\|u_{\eps}-u\|_{W^p_{\nu_\eps}(\Omega)}= \lim_{\eps\to0}\|u_{\eps}-u_g\|_{W^p_{\nu_\eps}(\R^d)}=0,
\end{align*}

If $\Omega$ is bounded in one direction  and  the setting of Theorem \ref{thm:converg-dirichlet-bis} is in force then
\begin{align*}
&\lim_{\eps\to0}\|u_{\eps}-u\|_{L^p_{\loc}(\Omega)}+ \cE^\eps_\Omega(u_{\eps}-u,u_{\eps}-u)=\lim_{\eps\to0} \cE^\eps(u_{\eps}-u_g,u_{\eps}-u_g)= 0.
\end{align*}
Moreover, if in addition $g_\eps= g$ then we have
\begin{align*}
\lim_{\eps\to0}\|u_{\eps}-u\|_{W^p_{\nu_\eps}(\Omega)}
=\lim_{\eps\to0}\|u_{\eps}-u_g\|_{W^p_{\nu_\eps}
(\R^d)}=0.
\end{align*}
%where we recall $u_g\in W^{1,p}(\R^d)$ refers to the extension of  $u$ by $g$ on $\Omega^c$.
\end{theorem}

\begin{proof}
The first claim is proved by proceeding in the same manner as in the proof of Theorem~\ref{thm:converg-dirichlet}, while the second claim follows analogously by adapting the proof of Theorem~\ref{thm:converg-dirichlet-bis}. We therefore prove only the first claim, highlighting the necessary modifications at the points where the arguments differ.

\smallskip

\textbf{Fact 1.} (Convergence holds in $L^p(\R^d)$).  By assumption $(g_\eps)_\eps$ we find that
\begin{align*}
\lim_{\eps\to0}	\|g_\eps-g\|_{W^p_{\nu_\eps}(\Omega|\R^d)}
\leq  \lim_{\eps\to0}	\|g_\eps-g\|_{W^p_{\nu_\eps}(\R^d)}=0.
\end{align*}
This enabled us  to prove that $(u_\eps)_\eps$ converges (up to a subsequence) to $u_g\in W^{1,p}(\R^d)$  in $ L^p_{\loc}(\R^d)$. However, if $\Omega$ is bounded then $(u_\eps)_\eps$ converges to $u_g$ in $L^p(\R^d)$ since
\begin{align}\label{eq:xxconver-ueps-Lp}
\| u_\eps-u_g\|^p_{L^p(\R^d)}= \| u_\eps-u\|^p_{L^p(\Omega)}+ \| g_\eps-g\|^p_{L^p(\Omega^c)}\xrightarrow{\eps\to0}0.
\end{align}
\textbf{Fact 2.} (The recovery sequence $(v_\eps)_\eps$).  Let $v\in g+W^{1,p}_0(\Omega)$.  If $g=0$ then $v\in W^{1,p}_0(\Omega)\cap W^{1,p}(\R^d)$ and hence Theorem \ref{thm:BBM-dual-limit} implies  $\cE^\eps(v,v)\to K_{d,p}\cE^0(v,v)$. Meanwhile, if $|\partial\Omega|=0$ then we know (see  for instance  \cite{Fog23}) that
\begin{align*}
K_{d,p}\cE^0(v,v)&\leq \liminf_{\eps\to0}\cE^\eps(v,v)\leq
\limsup_{\eps\to0}\cE^\eps(v,v)\\
&\leq K_{d,p}\int_{ \overline{\Omega}}|\nabla v(x)|^p\d x = K_{d,p}\cE^0(v,v).
\end{align*}
In either case $\cE^\eps(v,v)\to K_{d,p}\cE^0(v,v)$. As in \eqref{eq:Eeps-veps-conv}, this enables us to deduce that $\cE^\eps(v_\eps,v_\eps)\to K_{d,p}\cE^0(v,v)$ with $v_\eps = v-g+g_\eps$ and $v\in g+W^{1,p}_0(\Omega)$. The latter is used to imply that $u$ is the solution to the local Dirichlet problem and that the whole sequence $(u_\eps)_\eps$ converges to $u_g$ in $L^p_{\loc}(\R^d)$ or in $L^p(\R^d)$ if $\Omega$ is bounded.

\smallskip

\textbf{Fact 3.} (Optimal convergence).  In light of Theorem \ref{thm:asymp-conv-form}  we find that $\cE^\eps_{\R^d}(u_g, u_\eps-u_g)\to0$.  The fact that $u_\eps$ is a weak solution implies
\begin{align*}
\lim_{\eps\to0}\cE^\eps(u_\eps, u_\eps-u_g)= \lim_{\eps\to0}\langle f_\eps, u_\eps-u_g\rangle_\eps=0.
\end{align*}
Since $u_\eps= g_\eps$  and $u_g= g$ on $\Omega^c$, by linearity we find that
\begin{align*}
\lim_{\eps\to0}|\cE^\eps_{\R^d}(u_\eps, u_\eps-u_g)|
%&=\lim_{\eps\to0}\big|\cE^\eps(u_\eps, u_\eps-u_g)
%+ \cE^\eps_{\Omega^c}(g_\eps, g_\eps-g)\big|\\
&\leq \lim_{\eps\to0}\big(|\cE^\eps(u_\eps, u_\eps-u_g)|+ C\|g_\eps-g\|_{W^{p}_{\nu_\eps}(\R^d)} \big)=0.
\end{align*}
Finally Theorem \ref{thm:equiv-conv-in-form} and Poincar\'e inequality \eqref{eq:uniform-coercivity-D-bis} imply  $\|u_{\eps}-u_g\|^p_{W^p_{\nu_\eps}(\R^d)}\leq C\cE^\eps_{\R^d}(u_\eps -u_g,u_\eps -u_g)\to0.$ In particular $\|u_{\eps}-u_g\|_{L^p(\R^d)}\to0.$
\end{proof}
\smallskip

Given  $h:\R^d\to \R$ measurable, we denote $L^p_h(\Omega)$ the metric space of all functions in $L^p(\Omega)$ extended by $h$ on $\Omega^c$. Namely we have
\begin{align*}
L^p_h(\Omega)= \{ u:\R^d\to\R\,\,\text{meas., s.t. $u|_\Omega\in L^p(\Omega)$  and $u=h$ a.e. on $\Omega^c$} \}.
\end{align*}
Note that $L^p_h(\Omega)$ is a metric space with the metric $(u,v)\mapsto \| u-v\|_{L^p(\Omega)}$.  The asymptotic of the functionals $(\cJ^\eps_0)_\eps$ towards $\cJ_0$ also exhibits a form of Gamma-convergence.
\begin{theorem}\label{thm:gamma-conv-Jo}
Let the assumptions of Theorem \ref{thm:converg-dirichlet}  be in force. Let us extend the functionals $\cJ^\eps_0: L^p_{g_\eps}(\Omega)\to (-\infty,\infty]$  and $\cJ_0: L^p_{g}(\Omega)\to (-\infty,\infty]$ as follows
\begin{align*}
\cJ^\eps_0 (v)
&=\begin{cases}
\tfrac{\mu}{p} \mathcal{E}^\eps(v,v)  -\langle f_\eps, v-g_\eps\rangle_\eps  & \text{if $v\in g_\eps+\WnuOmRao$}\\
\infty &\text{ else },
\end{cases}\\
\cJ_0(v) &=\begin{cases}
\tfrac{1}{p}\mathcal{E}^0(v,v)  -\langle f, v-g\rangle
& \text{if $v\in g+W^{1,p}_0(\Omega)$}\\
\infty &\text{ else }.
\end{cases}
\end{align*}
The following assertions are true.

\begin{itemize}
\item \textbf{Limsup}.  For every  $v\in L^p_g(\Omega)$ there exists a sequence $(v_\eps)_\eps$ with  $v_\eps-g_\eps \in \WnuOmRao$, in particular $v_\eps \in L^p_{g_\eps}(\R^d) $, such that  $v_\eps\to v$ in $L^p(\Omega)$  and
\begin{align*}
\limsup_{\eps\to0} \cJ^\eps_0(v_\eps)\leq \cJ_0(v).
\end{align*}
\item \textbf{Liminf}. For any sequence $(v_\eps)_\eps$  with $v_\eps \in L^p_{g_\eps}(\R^d) $ and $v \in L^p_{g}(\R^d) $ such that $v_\eps\to v$ in $L^p(\Omega)$ we have
\begin{align*}
\liminf_{\eps\to0} \cJ^\eps_0(v_\eps)\geq \cJ_0(v).
\end{align*}
\end{itemize}
Further, this result is true if $\Omega$ is bounded in one direction, under the setting of Theorem \ref{thm:converg-dirichlet-bis}, provided that the $L^p_{\mathrm{loc}}(\Omega)$-topology is used instead of  the $L^p(\Omega)$-topology.
\end{theorem}

\begin{proof}
As for the expression in \eqref{eq:geps-extens}, we can find   $ \overline{g}_\eps\in\WnuOmRa$  and $\overline{g}\in W^{1,p}(\R^d)\subset\WnuOmRa$  such that $\overline{g}_\eps= g_\eps$ and  $\overline{g}= g$ a.e. on $\R^d\setminus\overline{\Omega}$  satisfying
\begin{align*}
\|\overline{g}_\eps-\overline{g}\|_{\WnuOmRa}\leq \eps+ \|g_\eps-g\|_{\TnuOma}\xrightarrow{\eps\to0}0.
\end{align*}
If  $v\in L^p_g(\Omega)\setminus (g+W^{1,p}_0(\Omega))$ then $\cJ_0(v)=\infty$ and the limsup holds.  In this case, consider any sequence $ w_\eps \in C_c^\infty(\Omega)\subset \WnuOmRao$ and  $\|w_\eps-(v-\overline{g})\|_{L^p(\Omega)}\to 0$ as $\eps\to 0$ and define $(v_\eps)_\eps$  with $v_\eps= w_\eps+ \overline{g}_\eps$. Hence $v_\eps\in \overline{g}_\eps +\WnuOmRao$ and
\begin{align*}
\|v_\eps-v\|_{L^p(\Omega)}\leq  \|w_\eps-(v-\overline{g})\|_{L^p(\Omega)}+  \|\overline{g}_\eps-\overline{g}\|_{\WnuOmRa}
\xrightarrow{\eps\to0}0.
\end{align*}
Now assume $v\in g+ W^{1,p}_0(\Omega)$ and put $v_\eps=\overline{g}_\eps+v -\overline{g} $
then $v_\eps\in \overline{g}_\eps+ \WnuOmRao$ and  $v_\eps\to v$ in $L^p(\Omega)$ since
\begin{align*}
\|v_\eps-v\|_{L^p(\Omega)}
\leq \|\overline{g}_\eps-\overline{g}\|_{\WnuOmRa}
\xrightarrow{\eps\to0}0.
\end{align*}
In both settings we have  $\cJ^\eps_0 (v_\eps)\to \cJ_0 (v)$, in particular, the limsup follows. Indeed $\langle f_\eps, v_\eps -\overline{g}_\eps\rangle_\eps
= \langle f_\eps, v -\overline{g}\rangle_\eps\xrightarrow{\eps\to0} \langle f, v-\overline{g}\rangle$ while by \eqref{eq:Eeps-veps-conv} we have
\begin{align*}
\lim_{\eps\to0}\cE^\eps (v_\eps,v_\eps)= K_{d,p}\cE^0 (v,v).
\end{align*}
To prove the liminf condition,  let  $(v_\eps)_\eps$ be a sequence of measurable functions such that $v_\eps\to v$ in $L^p(\Omega)$ (or in $L^p_{\loc}(\Omega)$). The case $\liminf_{\eps\to0}\cJ^\eps(v_\eps)= \infty$ is trivial. Assume $\liminf\cJ^\eps(v_\eps)<\infty$ yielding in particular that  $\liminf_{\eps\to0}\cE^\eps(v_\eps,v_\eps)<\infty$. According to Theorem \ref{thm:BBM-liminf} we find that
\begin{align*}
\liminf_{\eps\to0}\cE^\eps(v_\eps,v_\eps)\geq
\liminf_{\eps\to0}\cE^\eps_\Omega(v_\eps,v_\eps)\geq K_{d,p}\cE^0(v,v).
\end{align*}
Hence if the setting of Theorem \ref{thm:converg-dirichlet} applies and $v_\eps\to v $ in $L^p(\Omega),$ we have
$\langle f_\eps,v_\eps-\overline{g}_\eps \rangle_\eps  \to  \langle f, v-\overline{g}\rangle$  as $\eps \to$ and  thus one gets the liminf condition
\begin{align*}
\liminf_{\eps\to0}\cJ^\eps(v_\eps)\geq\cJ(v),
\end{align*}
If the setting of Theorem \ref{thm:converg-dirichlet-bis} applies and $v_\eps\to v $ in $L_{\loc}^p(\Omega)$, the liminf follows likewise if we show that  $\langle f_\eps,v_\eps-\overline{g}_\eps \rangle_\eps  \to  \langle f, v-\overline{g}\rangle$ up to a subsequence as $\eps \to0$.  The latter holds  true since  $(v_\eps-\overline{g}_\eps)_\eps$ converges weakly to $v-\overline{g}$ in $L^p(\Omega)$, and $(f_\eps)_\eps$ strongly asymptotically converges to $f$ by assumption.
Indeed, since $\liminf\limits_{\eps\to0}\cE^\eps(v_\eps,v_\eps)<\infty$ we can assume without loss of generality that $\sup\limits_{\eps>0}\cE^\eps(v_\eps,v_\eps)<\infty.$ The asymptotic boundedness of $(\overline{g}_\eps)_\eps$ and the Poincar\'e inequality \eqref{eq:uniform-coercivity-D-bis} imply
\begin{align*}
\|v_\eps-\overline{g}_\eps\|_{\WnuOmRa}
&\leq
C\cE^\eps(v_\eps-\overline{g}_\eps,
v_\eps-\overline{g}_\eps)^{1/p}\\
&\leq C\|\overline{g}_\eps\|_{\WnuOmRa}+ C\cE^\eps(v_\eps,v_\eps)^{1/p}\leq C.
\end{align*}
Thus the sequence $(v_\eps-\overline{g}_\eps)_\eps $ is bounded in $L^p(\Omega)$ and converges to $v-\overline{g}$ in $L^p_{\loc}(\Omega)$. This implies that $(v_\eps-\overline{g}_\eps)_\eps$ weakly converges to $v-\overline{g}$ in $L^p(\Omega)$.
\end{proof}

\section{Case of fractional \texorpdfstring{$p$}{p}-Laplacian}\label{sec:conv-weak-solution-frac}
\noindent Now we are interested in  the case of the normalized  fractional $p$-Laplacian and the corresponding $p$-normal derivative, see \cite{Fog25}, respectively given by
\begin{align*}
(-\Delta)^s_pu(x)&:=\widetilde{C}_{d,p,s}\pv \int_{\R^d}\frac{|u(x)-u(y)|^{p-2}(u(x)-u(y))}{|x-y|^{d+sp}}\d y, \\
\cN_p^su(y)&:=\widetilde{C}_{d,p,s}\int_{\Omega}\frac{|u(x)-u(y)|^{p-2}(u(x)-u(y))}{|x-y|^{d+sp}}\d x
\end{align*}

\noindent The results in the fractional setting are retrieved from the general framework developed in Section~\ref{sec:conv-weak-solution} by considering $\nu_\varepsilon(h)=a_{d,p,\eps}|h|^{-d-(1-\eps)p}$ with $a_{d,p,\eps}=\frac{p\eps(1-\eps)}{|\mathbb{S}^{d-1}|}$
together with the asymptotic behavior of the normalizing constant $\widetilde{C}_{d,p,s}$
\begin{align*}
\lim_{\eps\to0}\frac{\widetilde{C}_{d,p,1-\eps}}{a_{d,p,\eps}}
= \lim_{s\to1^-}
\frac{\widetilde{C}_{d,p,s}|\mathbb{S}^{d-1}|}{ps(1-s)}
=\frac{2}{K_{d,p}}.
\end{align*}
\noindent The spaces $\WnuOma$, $ \WnuOmRa$ and $\TnuOma$ respectively become $W^{s,p}(\Omega)$, $W^{s,p}(\Omega|\R^d)$ and $T^{s,p}(\Omega)$, where  norms $\|\cdot\|_{W^{s,p}(\Omega)}$ and 	$\|\cdot\|_{W^{s,p}(\Omega|\R^d)}$ are given by
\begin{align*}
\|u\|^p_{W^{s,p}(\Omega)}
&=  \|u\|^p_{L^p(\Omega)}+ \cE^{s,p}_\Omega(u,u),
\\
\|u\|^p_{W^{s,p}(\Omega|\R^d)}
&=  \|u\|^p_{L^p(\Omega)}+ \cE^{s,p}(u,u),
\end{align*}
with fractional forms $ \cE^{s,p}_\Omega(\cdot,\cdot)$ and $ \cE^{s,p}(\cdot,\cdot)$ given by
 \begin{align*}
\quad \cE^{s,p}_\Omega(u,v)
&:=\frac{\widetilde{C}_{d,p,s}}{2}\iil_{ \Omega \Omega} \frac{|u(x)-u(y)|^{p-2}(u(x)-u(y))(v(x)-v(y))}{|x-y|^{d+sp}}\d y\d x,\\
\quad \cE^{s,p}(u,v) &:=\frac{\widetilde{C}_{d,p,s}}{2} \iil_{( \Omega^c\times \Omega^c)^c} \frac{|u(x)-u(y)|^{p-2}(u(x)-u(y))(v(x)-v(y))}{|x-y|^{d+sp}}\d y \d x.
\end{align*}
The next two results are implied by Theorem \ref{thm:convergence-neumann} and Theorem \ref{thm:convergence-neumann regional} respectively.
\begin{theorem}\label{thm:convergence-neumann-frac}
Assume $\Omega\subset \R^d$ is open bounded Lipschitz and connected. Let $ (f_s)_s$, with  $f_s\in  (W^{s,p}(\Omega|\R^d))'$, and $ (g_s)_s$, with $g_s\in  (T^{s,p}(\Omega^c))'$, $s\in (0,1)$, converge asymptotically weakly to  $f_1\in (W^{1,p}(\Omega))'$ and
$g_1\in( W^{1-1/p ,p}(\partial\Omega))'$ respectively as $s\to1^-$. Let  $u_s $ be the weak solution in $W^{s,p}(\Omega|\R^d)^\perp$ to the Neumann problem
\begin{align*}
(-\Delta)^s_p u = f_s \quad \text{ in } \,\, \Omega
\quad \text{ and } \quad
\mathcal{N}^s_p u = g_s \quad \text{ on } \,\, \Omega^c .
\end{align*}
Let $u_1$ be the  weak solution in $W^{1,p}(\Omega)^\perp$ of  the Neumann problem
\begin{align*}
-\Delta_p u = f_1 \quad \text{ in }\,\,  \Omega
\quad \text{and} \quad
\partial_{n,p} u = g_1 \quad \text{ on } \,\, \partial \Omega .
\end{align*}

\noindent Then $(u_s)_s$ strongly converges to $u$ in the following sense
\begin{align*}
\lim_{s\to1^-}\|u_s -u_1\|_{W^{s,p}(\Omega)} = \lim_{s\to1^-}\|u_s- \overline{u}_1\|_{W^{s,p}(\Omega|\R^d)}= \lim_{s\to1^-} \cE^{s,p}_{cr}(u_s, u_s)=0.
\end{align*}
Here $\overline{u}_1\in W^{1,p}(\R^d)$ is any $W^{1,p}$-extension of $u_1$.
\end{theorem}

\begin{theorem}\label{thm:converregio-frac}
Assume $\Omega\subset \R^d$ is open bounded Lipschitz and connected. Let $ (f_s)_s$, with  $f_s\in  (W^{s,p}(\Omega))'$, $s\in (0,1)$, converges asymptotically weakly to  $f_1\in (W^{1,p}(\Omega))'$ as $s\to1^-$.   Let  $u_s \in  W^{s,p}(\Omega)^\perp$ be the weak solution to the Neumann problem
\begin{align*}
(-\Delta)^s_{p,\Omega}u= f_{s}\quad
\text{ in\,\,  $\Omega$ \,\, and } \quad  \int_\Omega u(x) \, \d x=0,
\end{align*}
where $(-\Delta)^s_{p,\Omega}u$ is regional fractional $p$-Laplacian
\begin{align*}
(-\Delta)^s_{p,\Omega}u(x)&:=\widetilde{C}_{d,p,s}\pv \int_{\Omega}\frac{|u(x)-u(y)|^{p-2}(u(x)-u(y))}{|x-y|^{d+sp}}\d y.
\end{align*}
Let $u_1\in W^{1,p}(\Omega)^\perp$ be the  weak solution of  the Neumann problem
\begin{align*}
-\Delta_p u = f_1 \quad \text{ in } \,\, \Omega
\quad \text{and} \quad
\partial_{n,p} u = 0 \quad \text{ on } \,\, \partial \Omega .
\end{align*}

\noindent Then $(u_s)_s$ strongly converges to $u$ in the following sense
\begin{align*}
\lim_{s\to1^-}\|u_s -u_1\|_{W^{s,p}(\Omega)} = 0.
\end{align*}
\end{theorem}

\begin{theorem}\label{thm:converg-dirichlet-frac}
Assume $\Omega\subset \R^d$ is open.  For each $s\in(0,1)$, let  $f_s\in  (W^{s,p}_0(\Omega|\R^d))'$, $f_1\in (W^{1,p}_0(\Omega))'$,  $g_s\in  T^{s,p}(\Omega^c)$ and $g_1\in W^{1,p}(\R^d\setminus\overline{\Omega})$.
Let  $u_s \in  W^{s,p}(\Omega|\R^d)$ be the weak solution to the Dirichlet problem
\begin{align*}
(-\Delta)^s_pu= f_{s} \quad 	\text{ on $\Omega$ \,\, and }	\quad u=g_{s} \quad \text{ on \,\, $\Omega^c$.}
\end{align*}
Let $u_1\in W^{1,p}(\Omega)$ be the  weak solution of  the Dirichlet  problem
\begin{align*}
-\Delta_p u = f_1 \quad \text{ in }\,\,  \Omega
\quad \text{and} \quad
 u = g_1 \quad \text{ on }\,\, \partial \Omega .
\end{align*}
Let $\overline{u}_1$ be any $W^{1,p}$-extension of $u_1$ and $u_{1,g_1}$ be its the extension of $u_1$ by $g_1$.
\begin{enumerate}[$(I)$]
\item  \label{item:optimal} If $\Omega$ is bounded Lipschitz, $ (f_s)_s$  converges asymptotically weakly to  $f_1$
and $g_s$ strongly converges to $g_1$, i.e.,  $\|g_s-g_1\|_{T^{s,p}(\Omega^c)}\to0$ as $s\to1^-$,
then $(u_s)_s$ strongly converges to $u_1$ in the following sense
\begin{align}\label{eq:converg-dirichlet-bdd}
&\lim_{s\to1^-}\|u_s -u_1\|_{W^{s,p}(\Omega)} = \lim_{s\to1^-}
\|u_s- u_{1,g_1}\|_{W^{s,p}(\Omega|\R^d)}=0, \\
& \lim_{s\to1^-}
\|u_s- \overline{u}_1\|_{W^{s,p}(\Omega|\R^d)}=\lim_{s\to1^-} \cE^{s,p}_{cr}(u_s, u_s)=0.
\end{align}
\item \label{item:optimal-one-direc} If  $\Omega$ is bounded in one direction and Lipschitz, $ (f_s)_s$  converges asymptotically strongly to  $f_1$
and $g_s$ strongly converges to $g_1$, i.e.,
$\|g_s-g_1\|_{T^{s,p}(\Omega^c)}\to0$ as $s\to1^-$,
then $(u_s)_s$ strongly converges to $u_1$ in the following sense
\begin{align}\label{eq:converg-dirichlet-loc}
\begin{split}
&\lim_{s\to1^-}\|u_s -u_1\|^p_{L^p_{\loc}(\Omega)}+ \cE^{s,p}_{\Omega}(u_s-u_1,u_s-u_1)=0,\\
& \lim_{s\to1^-}\|u_s -u_1\|^p_{L^p_{\loc}(\R^d)}+ \cE^{s,p}(u_s-u_{1,g_1},u_s-u_{1,g_1})+ \cE^{s,p}_{cr}(u_s, u_s)=0.
\end{split}
\end{align}
Moreover, if in addition $g_s=g_1$ then we have $\|u_s-u_1\|_{L^p(\R^d)}\xrightarrow{s\to1^-}0$.
\item If $g_1=0$ or $|\partial\Omega|=0$,  and we assume that  $\Omega$ only has a continuous boundary and the sequence $(g_s)_s$ satisfies
$\|g_s-g_1\|_{W^{s,p}(\R^d)}\xrightarrow{s\to1^-}$
then  the strong convergence \eqref{eq:converg-dirichlet-bdd} from \ref{item:optimal} and the conclusions of \ref{item:optimal-one-direc} remain true.
\end{enumerate}
\end{theorem}

\begin{proof}
The result follows from Theorem \ref{thm:converg-dirichlet}, Theorem \ref{thm:converg-dirichlet-bis} and Theorem \ref{thm:converg-dirichlet-3}.
\end{proof}

%%%%%%%%%%%%%%%%%%%%%%%%%%%%%%
%\appendix
%%%%%%%%%%%%%%%%%%%%%%%%%%%%%%%%
%%%%%% produces listoftodos and seems to work with the amsart package
%%%%%%%%%%%%%%%%%%%%%%%%%%%%%%%%%
%\makeatletter
%\providecommand\@dotsep{5}
%\makeatother
%\listoftodos\relax
%%%%%%%%%%%%%%%%%%%%%%%%%%%%%%%% %%%%%%%%%%%
\vspace{1mm}
\noindent {\small \textbf{Data Availability Statement (DAS)}: Data sharing not applicable, no datasets were generated or analyzed during the current study.}

\vspace{1mm}
\noindent {\small \textbf{Conflict of Interest:}
{The author declares that there is no conflict of interest regarding the publication of this paper.}
%\bibliographystyle{plain}
%\bibliography{p-nonloc-bvp.bib}

%\begin{bibdiv}
%	\begin{biblist}
%		\bibselect{p-nonloc-bvp.bib} you need amsrefs-package
%	\end{biblist}
%\end{bibdiv}
\end{document}